\documentclass[reqno,11pt]{amsart}
\usepackage{setspace,tikz,xcolor,mathrsfs,listings,multicol,amssymb}
\usepackage{rotating}
\usepackage[vcentermath]{youngtab}
\usepackage[margin=1in,includefoot,footskip=30pt]{geometry}
\usepackage{enumerate}
\usepackage{booktabs}
\usepackage[centertableaux]{ytableau}
\usepackage[all,cmtip]{xy}
\usetikzlibrary{arrows,matrix}
\tikzset{tab/.style={matrix of math nodes,column sep=-.35, row sep=-.35,text height=7pt,text width=7pt,align=center,inner sep=2,font=\footnotesize}}

\usepackage[colorlinks=true, pdfstartview=FitV, linkcolor=blue, citecolor=blue, urlcolor=blue]{hyperref}

\newcommand{\ssyt}{\operatorname{SSYT}}
\newcommand{\svt}{\operatorname{SVT}}
\newcommand{\rpp}{\operatorname{RPP}}
\newcommand{\et}{\operatorname{ET}}  
\newcommand{\set}{\operatorname{SET}}  

\newcommand{\absval}[1]{\left\lvert #1 \right\rvert}

\newcommand{\comp}{\dagger}

\newcommand{\bbb}{\mathsf{b}}

\newcommand{\G}{G}
\newcommand{\dG}{g}

\DeclareMathOperator{\wt}{wt} 
\DeclareMathOperator{\lf}{left} 
\DeclareMathOperator{\inv}{inv} 
\DeclareMathOperator{\RSK}{RSK} 

\newcommand{\mcE}{\mathcal{E}}

\newcommand{\mcT}{\mathcal{T}}

\newcommand{\fG}{\mathfrak{G}}
\newcommand{\fM}{\mathfrak{M}}
\newcommand{\fN}{\mathfrak{N}}

\newcommand{\ta}{\mathsf{a}}
\newcommand{\tb}{\mathsf{b}}
\newcommand{\tc}{\mathsf{c}}

\newcommand{\bt}{\mathbf{t}}
\newcommand{\ff}{\mathbf{f}}
\newcommand{\pp}{\mathbf{p}}
\newcommand{\uu}{\mathbf{u}}
\newcommand{\vv}{\mathbf{v}}
\newcommand{\xx}{\mathbf{x}}
\newcommand{\yy}{\mathbf{y}}
\newcommand{\zz}{\mathbf{z}}
\newcommand{\bG}{\mathbf{G}}

\newcommand{\ZZ}{\mathbb{Z}}

\newcommand{\CC}{\mathbb{C}}

\newcommand{\vertex}[4]{ 
\begin{tikzpicture}[baseline=0,scale=0.7,>=latex]
  \draw[->] (-1,0) node[anchor=east] {$#4$} -- (1,0) node [anchor=west] {$#2$};
  \draw[->] (0,-1) node[anchor=north] {$#3$} -- (0,1) node [anchor=south] {$#1$};
  \draw[color=darkred] (-0.2,0) arc (180:270:0.2);
  \draw (0,0) node[anchor=north east, color=darkred] {\small $z$};
\end{tikzpicture}
}
\newcommand{\vertexF}[4]{ 
\begin{tikzpicture}[baseline=0,scale=0.7,>=latex]
  \draw[->] (1,0) node[anchor=west] {$#2$} -- (-1,0) node [anchor=east] {$#4$};
  \draw[->] (0,-1) node[anchor=north] {$#3$} -- (0,1) node [anchor=south] {$#1$};
  \draw[color=darkred] (-0.2,0) arc (180:270:0.2);
  \draw (0,0) node[anchor=north east, color=darkred] {\small $z$};
\end{tikzpicture}
}
\newcommand{\rmatrix}[4]{ 
\begin{tikzpicture}[baseline=0,scale=0.7,>=latex]
  \draw[->] (-1,1) node[anchor=east] {$#1$} -- (1,-1) node [anchor=west] {$#3$};
  \draw[->] (-1,-1) node[anchor=east] {$#4$} -- (1,1) node [anchor=west] {$#2$};
  \draw[color=darkred] (135:0.2) arc (135:225:0.2);
  \draw (-0.1,0) node[anchor=east, color=darkred] {\small $z_i,z_j$};
\end{tikzpicture}
}

\newcommand{\rmatrixD}[4]{ 
\begin{tikzpicture}[baseline=0,scale=0.7,>=latex]
  \draw[->] (1,1) node[anchor=west] {$#4$} -- (-1,-1) node [anchor=east] {$#2$};
  \draw[->] (-1,1) node[anchor=east] {$#1$} -- (1,-1) node [anchor=west] {$#3$};
  \draw[color=darkred] (135:0.2) arc (135:225:0.2);
  \draw (-0.1,0) node[anchor=east, color=darkred] {\small $z_i,z_j$};
\end{tikzpicture}
}
\newcommand{\rmatrixF}[4]{ 
\begin{tikzpicture}[baseline=0,scale=0.7,>=latex]
  \draw[->] (1,1) node[anchor=west] {$#3$} -- (-1,-1) node [anchor=east] {$#1$};
  \draw[->] (1,-1) node[anchor=west] {$#2$} -- (-1,1) node [anchor=east] {$#4$};
  \draw[color=darkred] (135:0.2) arc (135:225:0.2);
  \draw (-0.1,0) node[anchor=east, color=darkred] {\small $z_i,z_j$};
\end{tikzpicture}
}

\definecolor{darkred}{rgb}{0.7,0,0} 
\newcommand{\defn}[1]{{\color{darkred}\emph{#1}}} 

\definecolor{UQgold}{RGB}{196, 158, 54} 
\definecolor{UQpurple}{RGB}{73, 7, 94} 
\definecolor{UMNgold}{RGB}{255,200,46} 
\definecolor{UMNmaroon}{RGB}{106,0,50} 

\usepackage{listings}
\lstdefinelanguage{Sage}[]{Python}
{morekeywords={False,sage,True},sensitive=true}
\definecolor{dblackcolor}{rgb}{0.0,0.0,0.0}
\definecolor{dbluecolor}{rgb}{0.01,0.02,0.7}
\definecolor{dgreencolor}{rgb}{0.2,0.4,0.0}
\definecolor{dgraycolor}{rgb}{0.30,0.3,0.30}

\theoremstyle{plain}
\newtheorem{thm}{Theorem}[section]
\newtheorem{lemma}[thm]{Lemma}

\newtheorem{prop}[thm]{Proposition}
\newtheorem{cor}[thm]{Corollary}
\theoremstyle{definition}
\newtheorem{dfn}[thm]{Definition}
\newtheorem{ex}[thm]{Example}
\newtheorem{remark}[thm]{Remark}

\numberwithin{equation}{section}

\usepackage[colorinlistoftodos]{todonotes}

\begin{document}
\title[Refined dual Grothendiecks]{Refined dual Grothendieck polynomials, integrability, and the Schur measure}

\author[K.~Motegi]{Kohei Motegi}
\address[K.~Motegi]{Faculty of Marine Technology, Tokyo University of Marine Science and Technology, Etchujima 2-1-6, Koto-Ku, Tokyo, 135-8533, Japan}
\email{kmoteg0@kaiyodai.ac.jp}
\urladdr{https://sites.google.com/site/motegikohei/home}

\author[T.~Scrimshaw]{Travis Scrimshaw}
\address[T.~Scrimshaw]{School of Mathematics and Physics, The University of Queensland, St.\ Lucia, QLD 4072, Australia}
\email{tcscrims@gmail.com}
\urladdr{https://people.smp.uq.edu.au/TravisScrimshaw/}


\thanks{T.S.~was partially supported by the Australian Research Council DP170102648.
K.M. was partially supported by grant-in-Aid for and Scientific Research (C) No. 18K03205 and No. 20K03793.
}

\begin{abstract}
We construct a vertex model whose partition function is a refined dual Grothendieck polynomial, where the states are interpreted as nonintersecting lattice paths.
Using this, we show refined dual Grothendieck polynomials are multi-Schur functions and give a number of identities, including a Littlewood and Cauchy(--Littlewood) identity.
We then refine Yeliussizov's connection between dual Grothendieck polynomials and the last passage percolation (LPP) stochastic process discussed by Johansson.
By refining algebraic techniques of Johansson, we show Jacobi--Trudi formulas for skew refined dual Grothendieck polynomials conjectured by Grinberg and recover a relation between LPP and the Schur process due to Baik and Rains.
Lastly, we extend our vertex model techniques to show some identities for refined Grothendieck polynomials, including a Jacobi--Trudi formula.
\end{abstract}

\maketitle

\section{Introduction}
\label{sec:introduction}

There has been a rich history discussing the interplay between geometry, algebra, combinatorics, integrable systems, and probability theory.
Here we focus on the (connective) K-theory ring of the Grassmannian, the set of $k$-dimensional planes in $\CC^n$, where we look at the basis of K-theory classes induced from the Schubert varieties.
Algebraically, we can represent these K-theory classes using (symmetric $\beta$-)Grothendieck polynomials from the work of Lascoux and Sch\"utzenberger~\cite{LS82,LS82II}, which were first described combinatorially by Fomin and Kirillov~\cite{FK94,FK96} with a Littlewood--Richardson rule given by Buch~\cite{Buch02}.
Grothendieck polynomials have been connected with integrable vertex models by numerous authors in various ways~\cite{BS20,BSW20,GK17,HKZJ18,KZJ17,MS13,MS14,WZJ16}, among additional literature discussing their algebraic, combinatorial, and geometric properties.

Symmetric Grothendieck polynomials are a basis for the ring of (an appropriate completion of) symmetric functions $\Lambda$, and we can construct the dual basis under the Hall inner product, which is defined by setting the Schur functions to be an orthonormal basis. Lam and Pylyavskyy showed that the dual Grothendieck polynomials were shown to be described combinatorially as the sum over reverse plane partitions and their Schur decomposition~\cite{LamPyl07} . The $\beta$ parameter from connective K-theory was refined to parameters $\bt = (t_1, \dotsc, t_{\ell-1})$ by Galashin, Grinberg, and Liu~\cite{GGL16} using the combinatorial data. These refined dual Grothendieck polynomials $\dG_{\lambda}(\xx; \bt)$, where $\lambda$ is a partition, are the main focus of this paper.

Yeliussizov in a series of papers has been studying the refined dual Grothendieck polynomials by refining the Lam--Pylyavskyy Schur decomposition~\cite{Yel17} and building a relation to the totally asymmetric simple exclusion process (TASEP) and the last-passage percolation (LPP) corner growth model~\cite{Yel19II,Yel20}.
Furthermore, a dual Jacobi--Trudi formula for skew shapes was conjectured by Darij Grinberg and recently independently proved by Kim~\cite{Kim20II} and by Amanov and Yeliussizov~\cite{AY20}.
Grinberg also conjectured a Jacobi--Trudi formula for skew shapes.
For $\bt = \beta$, additional identities have been proven, such as a Littlewood, Cauchy, and Jacobi--Trudi identities by Amanov and Yeliussizov using combinatorial techniques~\cite{AY20,Yel19II,Yel19III} and (dual) Jacobi--Trudi for skew shapes using free fermions by Iwao~\cite{Iwao19,Iwao20}.
Additionally, Iwao and Nagai have shown that dual Grothendieck polynomials for rectangular shapes are a scalar multiple of the $\tau_n$ function used to solve the discrete Toda equation when the Lax matrix degenerates to a single eigenvalue~\cite{IN18}.

In this paper, we study refined dual Grothendieck polynomials using integrable systems and the LPP model. To define our vertex model for $\dG_{\lambda}(\xx; \bt)$, we use the Schur decomposition and interpret monomials as a family of nonintersecting lattice paths (NILP). We then reinterpret the NILP as a state in an integrable five-vertex model. We then apply ideas and techniques from integrable systems to give
\begin{itemize}
\item a new proof of the Jacobi--Trudi formula for straight shapes;
\item refined versions of Yeliussizov's identities, including the Littlewood and Cauchy(--Littlewood) identities for dual Grothendieck polynomials;
\item a previously unnoticed symmetry; and
\item an analog of identities due to Feh\'er, N\'emethi, and Rim\'anyi~\cite{FNR12} for Schur functions and to Guo and Sun~\cite{GS19} for factorial Grothendieck polynomials.
\end{itemize}
Moreover, by using results of Chen, Li, and Louck~\cite{CLL02}, we show that $\dG_{\lambda}(\xx; \bt)$ is a multi-Schur function defined by Lascoux~\cite{Lascoux03}, yielding refined versions of results by Lascoux and Naruse~\cite{LN14}.
We remark that our lattice model is not acting on a rectangular grid, but on a ``jagged'' grid, which we later found this idea has only been previously used in~\cite{BBW20}.
Our lattice model is also distinct from the recent work of Gunna and Zinn-Justin~\cite{GZJ20}, which is partially bosonic as opposed to our completely fermionic model.

Next we explore the relation with the general version of LPP and show that the refined version of Yeliussizov's bijection between the random matrices for a particular joint probability $P(\bG(n) = \lambda)$ and reverse plane partitions~\cite[Thm.~1]{Yel20} composed with an iterated version of RSK is equal to the so-called inflation map of Lam--Pylyavskyy~\cite[Thm.~9.8]{LamPyl07}.
This immediately gives a new probabilistic proof of our Cauchy--Littlewood identity.
Consequently, the joint probability $P(\bG(n) = \lambda)$ is $\dG_{\lambda}(\xx; \bt)$ up to an overall factor, which then allows us to give a natural definition of $\dG_{\lambda/\mu}(\xx; \bt)$ using the transition probability $P(\bG(n) = \lambda | \bG(0) = \mu)$ up to an overall factor.
We then refine the algebraic proof of Johansson~\cite{Johansson10} at the case $\xx = \bt = q^{1/2}$ to prove the (dual) Jacobi--Trudi formulas for $\dG_{\lambda/\mu}(\xx; \bt)$, which was recently independently proved by Kim~\cite{Kim20} using plethystic techniques.
We further refine this to give new proof of some of our other identities and a new proof of the result of Baik and Rains~\cite{BR01} (also Johansson~\cite{Johansson00} at $\xx = \bt = q^{1/2}$) relating LPP to the Schur measure on partitions due to Okounkov~\cite{Okounkov00,Okounkov01}

Lastly, we utilize our integrability techniques on the refined version of Grothendieck polynomials $\G_{\lambda}(\xx; \bt)$ recently introduced by Chan and Pflueger~\cite{CP19}. We construct a lattice model on a jagged grid using two different $L$-matrices by using the a refined version of the Schur decomposition due to Lenart~\cite{Lenart00}, which we then relate with nonintersecting lattice paths. This allows us to give a Jacobi--Trudi formula for $\G_{\lambda}(\xx; \bt)$ and a Feh\'er--N\'emethi--Rim\'anyi identity, which does not specialize to the Guo--Sun identity for factorial Grothendieck polynomials~\cite{GS19}.
We remark that this further expands on the connection between Grothendieck polynomials and TASEP in~\cite{MS13}, which is further related with integrable vertex models~\cite{KMO15,KMO16,KMO16II}.


This paper is organized as follows.
In Section~\ref{sec:background}, we give some background on refined (dual) Grothendieck polynomials and the related combinatorics.
In Section~\ref{sec:refined_models}, we relate refined dual Grothendieck polynomials with an integrable five-vertex lattice model and nonintersecting lattice paths.
In Section~\ref{sec:probability}, we describe refined dual Grothendieck polynomials using the transition probability of the last-passage percolation stochastic process.
In Section~\ref{sec:refined_grothendiecks}, we apply our techniques from Section~\ref{sec:refined_models} to refined Grothendieck polynomials.

\subsection*{Acknowledgements}

The authors thank Jang Soo Kim for sharing his preprint~\cite{Kim20} and discussions about his results.
KM would like to thank Hiroshi Naruse for invaluable discussions on the paper~\cite{LN14}.
TS would like to thank Kazumitsu Sakai for invaluable discussions about his five-vertex model with the first author~\cite{MS13,MS14}. TS also would like to thank Valentin Buciumas for the numerous discussions on vertex models, Damir Yeliussizov for useful discussions on his results, and Darij Grinberg for useful discussions.
The authors would like to thank Darij Grinberg and Damir Yeliussizov for comments on an earlier draft of this paper.
This work benefited from computations using {\sc SageMath}~\cite{sage,combinat}.

\section{Combinatorics}
\label{sec:background}

In this section, we give the necessary background and combinatorics.
Let $S_n$ denote the permutation group on the letters $\{1, 2, 3, \dotsc, n\}$.
Let $\xx = (x_1, x_2, \dotsc, x_n)$ and $\yy = (y_1, y_2, \dotsc, y_n)$ be sequences of indeterminates.
Unless otherwise stated, all indeterminates will commute.
We will also allow $n = \infty$.
Let $\xx \sqcup \yy := (x_1, \dotsc, x_n, y_1, \dotsc, y_n)$ denote the concatenation of sequences.
For a sequence $\alpha = (\alpha_1, \alpha_2, \dotsc, \alpha_n)$, define $\xx^{\alpha} := x_1^{\alpha_1} x_2^{\alpha_2} \dotsm x_n^{\alpha_n}$.

\subsection{Partitions}

A \defn{partition} $\lambda$ is a weakly decreasing finite sequence of nonnegative integers. The \defn{length} $\ell(\lambda)$ is the index of the last positive entry of $\lambda$.
We equate partitions up to trailing zeros, so, \textit{e.g.}, $(3,2) = (3,2,0)$.
We will often write partitions as words when the largest entry is at most $9$, and we will sometimes drop the trailing $0$ entries.
Let $m_i(\lambda)$ denote the multiplicity of $i$ in $\lambda$.
We equate partitions with their Young diagram, which we draw using English convention and the $(i,j)$-th box is in the $i$-th row (from the top) and $j$-th column (from the left).
We will often consider partitions in a $n \times k$ box, which can be considered as a sequence of $n$ $1$'s and $k$ $0$'s with each vertical step being a $1$ and horizontal step being a $0$.
We call this the \defn{$01$-sequence} of $\lambda$.
The \defn{complement} partition of $\lambda$, denoted $\lambda^{\comp}$, is the other squares in the $n \times k$ box (after a rotation by $\pi$) and obtained by reversing the $01$-sequence.
The \defn{dual} partition of $\lambda$, denoted $\lambda^{\vee}$, is obtained by interchanging $0 \leftrightarrow 1$ in the $01$-sequence (this is also the complement shape but instead reflected over the line $y = x$).
The \defn{conjugate} partition of $\lambda$ is $\lambda' = \lambda^{\comp\vee}$; that is to say it is the partition formed by interchanging $0 \leftrightarrow 1$ and reversing the $01$-sequence (these operations clearly commute) or formed by flipping the partition along the main diagonal (\textit{i.e.}, the line $y = -x$).

\begin{ex}
\label{ex:partition_defs}
Consider the partition $\lambda = (5, 3, 3)$ inside of a $4 \times 6$ rectangle. Then we have
\[
\begin{array}{c@{\hspace{30pt}}c}
\lambda = \ydiagram{6,6,6,6} * [*(gray)]{5,3,3}\,,
&
\begin{array}{l}
\text{$01$-sequence: } 1000110010, \\
\lambda^{\vee} = 0111001101 = (4,3,3,1,1,1), \\
\lambda^{\comp} = 0100110001 = (6,3,3,1), \\
\lambda' = 1011001110 = (3,3,3,1,1),
\end{array}
\end{array}
\]
where $\lambda$ is the shaded portion.
\end{ex}

A partition $\mu$ is \defn{contained} within $\lambda$ if the boxes of $\mu$ are a subset of those of $\lambda$, which we denote $\mu \subseteq \lambda$.
Given partitions $\mu \subseteq \lambda$, the \defn{skew shape} $\lambda / \mu$ is the set of boxes in $\lambda$ that are not in $\mu$.
Note that the white boxes in Example~\ref{ex:partition_defs} are precisely those in the skew shape $6666 / 533$.

A \defn{Grassmannian permutation} is a permutation with precisely one descent in position $k$. These are in bijection with partitions inside of a $k \times (n-k)$ rectangle given by
\[
w \mapsto (w(k) - k, \dotsc, w(1) - 1).
\]
Let $w(\lambda)$ denote the Grassmannian permutation corresponding to the partition $\lambda$.

Let $\rho_{\ell} = (\ell-1, \ell-2, \dotsc, 1, 0)$ denote the staircase partition. For $\lambda = (\lambda_1, \dotsc, \lambda_\ell)$, define
\[
\lambda \pm \rho_{\ell} := (\lambda_1 \pm (\ell-1), \lambda_2 \pm (\ell-2), \dotsc, \lambda_{\ell-1} \pm 1, \lambda_{\ell}).
\]

Let $e_{\lambda}(\xx)$ and $h_{\lambda}(\xx)$ denote the elementary and homogeneous symmetric functions corresponding to $\lambda$. For more on symmetric functions, we refer the reader to~\cite{ECII}.

For the remainder of this section, we fix a partition $\lambda = (\lambda_1, \dotsc, \lambda_{\ell})$ of length $\ell$.

\subsection{The Lindstr\"om--Gessel--Viennot lemma}

We give the \defn{Lindstr\"om--Gessel--Viennot (LGV) Lemma}~\cite{GV85,Lindstrom73}, which allows us to express the sum over a family of nonintersecting lattice paths as a determinant.

Let $\Delta$ be an edge-weighted directed graph.
We will sometimes refer to an edge as an \defn{arc}, and for an arc $a$, denote the \defn{weight} by $\wt(a)$.
Consider two tuples of vertices $\uu = (u_1, u_2, \dotsc, u_k)$ and $\vv = (v_1, v_2, \dotsc, v_k)$ of $\Delta$.
Then a \defn{nonintersecting lattice path tuple (NILP)} from $\uu$ to $\vv$ is a tuple $(p_1, p_2, \dotsc, p_k)$ of (directed) paths in $\Delta$ such that
\begin{itemize}
\item each $p_i$ is a path from $u_i$ to $v_i$;
\item no two of the paths $p_1, p_2, \dotsc, p_k$ have a vertex in common.
\end{itemize}
Let $N(\uu, \vv)$ denote the set of all NILPs from $\uu$ to $\vv$.
Define the \defn{weight} of a path $p = (a_1, a_2, \dotsc, a_{\ell})$, and NILP $\pp = (p_1, p_2, \dotsc, p_k)$ to be
\[
\wt(p) = \prod_{i=1}^{\ell} \wt(a_i),
\qquad\qquad\qquad
\wt(\pp) = \prod_{i=1}^k \wt(p_i).
\]

\begin{lemma}[{LGV lemma~\cite{GV85,Lindstrom73}}]
We have
\[
\det \left[ \sum_{\pp \in N(u_i, v_j)} \wt(\pp) \right] = \sum_{\pp \in N(\uu, \vv)} \wt(\pp).
\]
\end{lemma}

\subsection{Semistandard tableaux, Schur functions, and related combinatorics}
\label{sec:bg_tableaux}

A \defn{semistandard (Young) tableau} is a filling of the boxes of $\lambda$ with positive integers such that rows weakly increase from left-to-right and columns strictly increase from top-to-bottom.
Let $\ssyt^n(\lambda)$ denote the set of semistandard (Young) tableaux of shape $\lambda$ and every entry is at most $n$.
For a semistandard tableau $T$ with maximum entry $n$, denote $\xx^T := x_1^{m_1} x_2^{m_2} \dotsm x_n^{m_n}$, where $m_i$ is the number of $i$'s that occur in $T$.
We define a \defn{Schur function} as $s_{\lambda}(\xx) = \sum_{T \in \ssyt^n(\lambda)} \xx^T$.

Next we recall the description of semistandard tableaux using nonintersecting lattice paths (NILPs) coming from the classical proof of the Jacobi--Trudi identities for Schur functions~\cite[First proof of Thm.~7.16.1]{ECII} using the LGV lemma.
We first need to consider the (infinite) directed graph with vertex set $\ZZ^2$ and arcs
\begin{align*}
\text{\defn{north-steps}: } (i,j) & \to (i,j+1) & & \text{for all } (i,j) \in \ZZ^2, \qquad \text{and} \\
\text{\defn{east-steps}: } (i,j) & \to (i+1,j) & & \text{for all } (i,j) \in \ZZ^2.
\end{align*}
Define the weight of an arc $a$ as
\[
\wt(a) :=
\begin{cases}
x_j & \text{if $a$ is an east-step } (i,j) \to (i+1,j), \\
1 & \text{if $a$ is a north-step } (i,j) \to (i,j+1),
\end{cases}
\]
Therefore, all north-steps have weight $1$, while east-steps with $y$-coordinate $j$ have weight $x_j$.
There is a well-known weight-preserving bijection between between $\ssyt^n(\lambda)$, and $N(\uu, \vv)$ with $\uu = \bigl( (\ell,1) \dotsc, (1,1) \bigr)$ and $\vv = \bigl( (\ell + \lambda_1, n), \dotsc, (1 + \lambda_{\ell}, n) \bigr)$ given as follows.
For $T \in \ssyt^n(\lambda)$, construct the path $p_i$ for the $i$-th row $(r_1, \dotsc, r_{\lambda_i})$ by the $j$-th horizontal step being at height~ $r_j$.

Another combinatorial object that are in bijection with $\ssyt^n(\lambda)$ are \defn{Gelfand--Tsetlin (GT) patterns}, which are triangular arrays of integers $(a_{i,j} \mid 1 \leq i \leq n \text{ and } 1 \leq j \leq i)$ that satisfy the conditions $a_{i+1,j} \geq a_{i,j} \geq a_{i+1,j+1}$ and $a_{n,i} = \lambda_i$.
We draw GT patterns with the first row at the bottom, and so the top row is equal to $\lambda$.
The GT pattern that corresponds to a semistandard Young tableau $T$ is given by the $i$-th row equals to the shape containing all entries at most $i$ in $T$.

\begin{ex}
Under the bijections described above with $n = 4$, we have
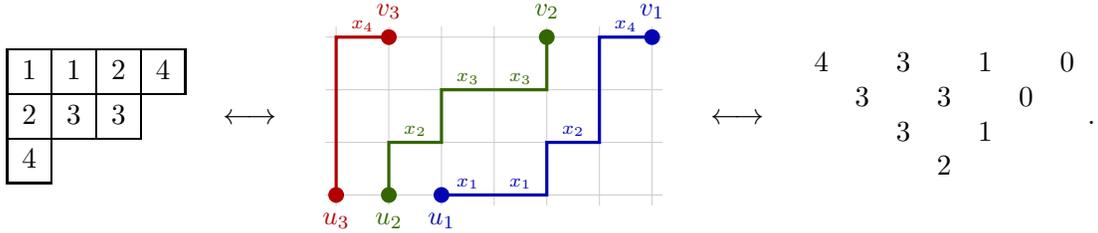
\[
\ytableaushort{1124,233,4}
\quad \longleftrightarrow \quad
\begin{tikzpicture}[>=latex,baseline=1.65cm, scale=0.7]
\draw[very thin, black!20] (1-0.2, 1-0.2) grid (7.2, 4.2);
\draw[-, very thick, darkred] (1,1) -- (1,4) -- node[midway,above=-2pt] {\tiny $x_4$} (2,4);
\fill[darkred] (1,1) circle (0.15) node[below=3pt] {\small $u_3$};
\fill[darkred] (2,4) circle (0.15) node[above=3pt] {\small $v_3$};
\draw[-, very thick, dgreencolor] (2,1) -- (2,2) -- node[midway,above=-2pt] {\tiny $x_2$} (3,2) -- (3,3) -- node[midway,above=-2pt] {\tiny $x_3$} (4,3) -- node[midway,above=-2pt] {\tiny $x_3$} (5,3) -- (5,4);
\fill[dgreencolor] (2,1) circle (0.15) node[below=3pt] {\small $u_2$};
\fill[dgreencolor] (5,4) circle (0.15) node[above=3pt] {\small $v_2$};
\draw[-, very thick, dbluecolor] (3,1) -- node[midway,above=-2pt] {\tiny $x_1$} (4,1) -- node[midway,above=-2pt] {\tiny $x_1$} (5,1) -- (5,2) -- node[midway,above=-2pt] {\tiny $x_2$} (6,2) -- (6,4) -- node[midway,above=-2pt] {\tiny $x_4$} (7,4);
\fill[dbluecolor] (3,1) circle (0.15) node[below=3pt] {\small $u_1$};
\fill[dbluecolor] (7,4) circle (0.15) node[above=3pt] {\small $v_1$};
\end{tikzpicture}
\quad \longleftrightarrow \quad
\begin{array}{ccccccc}
4 && 3 && 1 && 0
\\ & 3 && 3 && 0
\\ && 3 && 1
\\ &&& 2
\end{array}.
\]
\end{ex}

A semistandard tableau of shape $\lambda$ is \defn{(row) flagged} by $\ff = (f_1, \dotsc, f_{\ell}) \in \ZZ^{\ell}$ if every entry in row $i$ is at most $f_i$. In terms of the bijection above, the end points are instead $\vv = \bigl( (\ell + \lambda_1, f_1), \dotsc, (1 + \lambda_{\ell}, f_{\ell}) \bigr)$.
A \defn{(row) flagged Schur function} $s_{\lambda}^{\ff}(\xx)$ is the sum over all semistandard tableaux (with maximum entry at most $n$) of shape $\lambda$ with a fixed flagging $\ff = (f_1, \dotsc, f_{\ell})$ and are related to Schubert polynomials~\cite{LS82II}. They have a \defn{Jacobi--Trudi formula} arising from the LGV lemma~\cite{GV85,Wachs85} as indicated above.

All of the above naturally extends to skew shapes, where the starting points for the paths are instead $\uu = \bigl( (\ell+\mu_1,0) \dotsc, (1+\mu_{\ell},0) \bigr)$.
For a flagged semistandard tableau $T$ such that the $i$-th row is strictly less than $i$ (thus is flagged by $i-1$), we follow~\cite{LamPyl07} and call $T$ an \defn{elegant tableau}.
Let $\et(\lambda/\mu)$ denote the set of elegant tableaux of shape $\lambda/\mu$, and note that necessarily $\lambda_1 = \mu_1$.

Note we can generalize of the Jacobi--Trudi form of a Schur function to obtain the multi-Schur functions of Lascoux~\cite{Lascoux03}. Indeed, consider any sequence $I = (i_1, \dotsc, i_{\ell}) \in \ZZ^{\ell}$ and define the notation $S_i(\xx - \bt)$, where $\bt = (t_1, \dotsc, t_m)$ is a (possibly infinite) sequence of indeterminates, by the formal power series identity
\[
\sum_{i \geq 0} S_i(\xx - \bt) u^i = \prod_{i=1}^m (1 - t_i u) \prod_{i=1}^n (1 - x_i u)^{-1}.
\]
Next, given sequences of distinct indeterminates $\xx^{(1)}, \dotsc, \xx^{(\ell)}$ and $\bt^{(1)}, \dotsc, \bt^{(\ell)}$, the \defn{multi-Schur function} is defined by
\[
s_I\Bigl(\xx^{(1)} - \bt^{(1)}, \dotsc, \xx^{(\ell)} - \bt^{(\ell)} \Bigr) = \det \Bigl[ S_{i_k+h-k}\Bigl(\xx^{(k)} - \bt^{(k)} \Bigr) \Bigr]_{1 \leq h, k \leq \ell}.
\]
For a partition $\lambda$, we write $I(\lambda)$ as the sequence given by reversing $\lambda$; \textit{i.e.}, $I(\lambda) = (\lambda_{\ell}, \dotsc, \lambda_1)$.

In the sequel, we require the finite Cauchy formula for Schur functions:
\begin{equation}
\label{eq:finite_Cauchy_formula}
\sum_{\lambda \subseteq m^\ell}
s_\lambda(x_1,\dotsc,x_n)
s_{\lambda^\dagger}(t_1^{-1},\dotsc,t_\ell^{-1})=s_{m^\ell}(x_1,\dotsc,x_n,t_1^{-1},\dotsc,t_\ell^{-1}).
\end{equation}

\subsection{Generalized tableaux}

A \defn{(semistandard) set-valued tableau} is a filling of the boxes of $\lambda$ with finite non-empty sets of positive integers such that for any boxes
\[
\ytableaushort{AB,C}\,, \qquad
\text{we have }
\max A \leq \min B \text{ and } \max A < \min C.
\]
Let $\svt^n(\lambda)$ denote the set of set-valued tableaux of shape $\lambda$ and max entry $n$.
Let $\bt = (t_1, \dotsc, t_{\ell-1})$ be indeterminates.
Following Chan and Pflueger~\cite{CP19}, the \defn{refined (symmetric) Grothendieck polynomial} is defined as
\begin{equation}
\label{eq:Grothendieck_defn}
\G_{\lambda}(\xx; \bt) := \sum_{T \in \svt^n(\lambda)} (-1)^{\absval{\mathsf{e}(T)}} \bt^{\mathsf{e}(T)} \xx^T,
\end{equation}
where $\mathsf{e}(T) = (\mathsf{e}_1, \mathsf{e}_2, \dotsc, \mathsf{e}_n)$ with $e_i$ being the number of extra entries in row $i$ and $\absval{\mathsf{e}(T)} = \sum_{i=1}^n \mathsf{e}_i$.

\begin{remark}
If we specialize $t_1 = t_2 = \cdots = t_{\ell-1} = -\beta$, we recover the notion of a (symmetric) Grothendieck polynomial given by Buch~\cite{Buch02}.
The fact we specialize to $-\beta$, as opposed to $\beta$, comes from the $(-1)^{\absval{\mathsf{e}(T)}}$ coefficient in the definition~\eqref{eq:Grothendieck_defn}.
\end{remark}

A \defn{reverse plane partition} is a filing of the boxes of $\lambda$ such that both columns and rows are weakly increasing. Let $\rpp^n(\lambda)$ denote the set of reverse plane partitions of shape $\lambda$ and max entry $n$.
Define $a(T) = (a_1, a_2, \dotsc, a_n)$ by $a_i$ equal to the number of \emph{columns} containing an $i$.
Define $b(T) = (b_1, b_2, \dotsc, b_n)$ by $b_r$ begin the number of boxes $\bbb$ in the $r$-th row of $T$ whose entry equals the entry in the box directly below $\bbb$.
Grinberg, Galashin, and Liu in~\cite{GGL16} defined the \defn{refined dual Grothendieck polynomial} as
\[
\dG_{\lambda}(\xx; \bt) := \sum_{T \in \rpp^n(\lambda)} \bt^{b(T)} \xx^{a(T)}.
\]
Note that for $a(T)$, the $a_i$ counts the number of boxes $\bbb$ with an $i$ such that the box (immediately) below it does not contain an $i$.

\begin{thm}[{\cite[Eq.~(72)]{Yel17}}]
\label{thm:dualG_schur_decomposition}
We have
\[
\dG_{\lambda}(\xx, \bt) = \sum_{\mu \subseteq \lambda} e_{\lambda}^{\mu}(\bt) s_{\mu}(\xx),
\qquad \text{ where }
e_{\lambda}^{\mu}(\bt) := \sum_{T \in \et(\lambda/\mu)} \bt^T.
\]
\end{thm}

We can see this by considering the bijection
\[
\phi \colon \rpp^n(\lambda) \mapsto \bigsqcup_{\mu \subseteq \lambda} \ssyt^n(\mu) \times \et(\lambda/\mu)
\]
from~\cite[Thm.~9.8]{LamPyl07}, which we call the \defn{inflation map}.
It is straightforward to see the inflation map $\phi$ takes a box labeled by $i$ in the elegant tableau to a column-duplicate entry in the $i$-th row of the reverse plane partition.\footnote{This was sketched out in~\cite[Sec.~10.1]{Yel17}.}
This can also be seen as a direct consequence of the crystal structure on reverse plane partitions given by Galashin~\cite{Galashin17} as the inflation map is defined using RSK (we refer the reader to~\cite[Ch.~7.11]{ECII} for the definition of the RSK bijection), which is a crystal isomorphism.
We refer the reader to~\cite{BS17} for more information on crystals.

The refined Grothendieck polynomials and the refined dual Grothendieck polynomials form dual bases under the Hall inner product, which is defined by $\langle s_{\lambda}, s_{\mu} \rangle = \delta_{\lambda\mu}$.
The following proof is due to Darij Grinberg (see also~\cite[Rem.~3.9]{CP19}).

\begin{prop}[Grinberg]
We have
\[
\langle \G_{\lambda}(\xx; \bt), \dG_{\mu}(\xx; \bt) \rangle = \delta_{\lambda\mu}.
\]
\end{prop}

\begin{proof}
This follows from the natural generalization of~\cite[Thm.~4.1]{CP19} (starting instead from~\cite[Thm.~3.4]{CP19}) or~\cite[Thm.~2.7]{Lenart00} (using $\bt$ weighted paths) to
\[
s_{\lambda}(\xx) = \sum_{\mu \supseteq \lambda} e_{\lambda}^{\mu}(\bt) \G_{\mu}(\xx; \bt),
\]
Theorem~\ref{thm:dualG_schur_decomposition}, and the orthogonality of the Schur functions.
\end{proof}

\section{Integrability}
\label{sec:refined_models}

In this section, we describe refined dual Grothendieck polynomials $\dG_{\lambda}(\xx; \bt)$ using integrable lattice models that arise from nonintersecting lattice paths.
For this section, we fix a partition $\lambda = (\lambda_1, \dotsc, \lambda_{\ell})$ and a finite sequence of indeterminates $\bt = (t_1, t_2, \dotsc, t_{\ell-1})$; note that there are $\ell-1$ indeterminates.

Let us give a brief overview of our approach.
To construct the nonintersecting lattice paths, we begin by claiming that the dual Grothendieck polynomial $\dG_{\lambda}(\xx; \bt)$ is equal to the specialization of the flagged Schur function $s_{\lambda}^{\ff}(\xx, \bt)$ with $2n-1$ variables and flagging $\ff = (n, n+1, \dotsc, n + \ell - 1)$ as the refined version of Lascoux and Naruse~\cite[Eq.~(3)]{LN14}.
We then reinterpret $s_{\lambda}^{\ff}(\xx, \bt)$ as the multi-Schur function $s_{\lambda}(\xx, \xx + (t_1), \xx + (t_1, t_2), \dotsc, \xx + \bt)$.
Then we obtain an NILP interpretation from this using the results of Chen, Li, and Louck~\cite[Thm.~3.2]{CLL02},\footnote{We will reflect these constructions over the line $y = x$ to match our construction for the vertex models.} but in this interpretation, the boundary points are along a triangular region of length $\ell$. Yet because there are $\ell$ endpoints, the paths in this region can be uniquely extended so they all have the same $y$-coordinate. Furthermore, they do not alter the weight.
Yet our proof is actually noticing that these NILPs are a combinatorial translation of Theorem~\ref{thm:dualG_schur_decomposition} using the LGV lemma (see also~\cite{Wachs85}).
Therefore, we obtain the multi-Schur function description $\dG_{\lambda}(\xx; \bt) = s_{\lambda}^{\ff}(\xx, \bt)$ as a corollary.

By translating the lattice paths into vertex weights, we obtain our model. Note that there is a choice of which side of the vertical steps we count for the weight of the vertex model; this choice will be important later on.
We split the model into two parts: a \defn{rectangular part} corresponding to a Schur function with the spectral parameters $\xx$ and a \defn{jagged part} that gives the corresponding elegant tableau with the spectral parameters $\bt$.

\begin{ex}
\label{ex:NILP_tableaux}
We consider an NILP for the flagged Schur function in~\cite{LN14} for $\lambda = 4322$ and $n = 5$ given below. We compute the corresponding semistandard tableau of shape $\mu = 41$ and elegant tableau in the shaded portion:
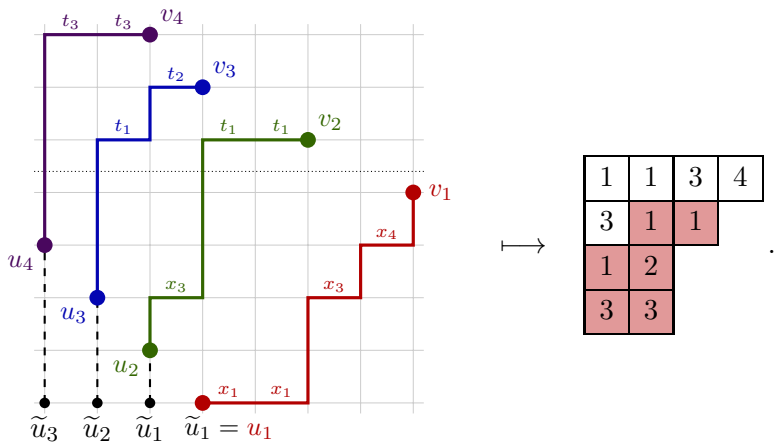
\[
\begin{tikzpicture}[>=latex,baseline=2.7cm,scale=0.7]
\draw[very thin, black!20] (-0.2, 1-0.2) grid (7.2, 8.2);
\draw[densely dotted] (-0.2,5.4) -- (7.2,5.4);
\foreach \x in {1,2,3} {
  \draw[-, thick, dashed] (\x-1,1) -- (\x-1,5-\x);
  \fill (3-\x,1) circle (0.1) node[below=1pt] {$\widetilde{u}_{\x}$};
}
\draw[-, very thick, darkred] (3,1) -- node[midway,above=-2pt] {\tiny $x_1$} (4,1) -- node[midway,above=-2pt] {\tiny $x_1$} (5,1) -- (5,3) -- node[midway,above=-2pt] {\tiny $x_3$} (6,3) -- (6,4) -- node[midway,above=-2pt] {\tiny $x_4$} (7,4) -- (7,5);
\fill[darkred] (3,1) circle (0.15) node[below=1pt] {\small \qquad ${\color{black} \widetilde{u}_1 = }\; u_1$};
\fill[darkred] (7,5) circle (0.15) node[right=2pt] {\small $v_1$};
\draw[-, very thick, dgreencolor] (2,2) -- (2,3) -- node[midway,above=-2pt] {\tiny $x_3$} (3,3) -- (3,6) -- node[midway,above=-2pt] {\tiny $t_1$} (4,6) -- node[midway,above=-2pt] {\tiny $t_1$} (5,6);
\fill[dgreencolor] (2,2) circle (0.15) node[below left] {\small $u_2$};
\fill[dgreencolor] (5,6) circle (0.15) node[above right] {\small $v_2$};
\draw[-, very thick, dbluecolor] (1,3) -- (1,6) -- node[midway,above=-2pt] {\tiny $t_1$} (2,6) -- (2,7) -- node[midway,above=-2pt] {\tiny $t_2$} (3,7);
\fill[dbluecolor] (1,3) circle (0.15) node[below left] {\small $u_3$};
\fill[dbluecolor] (3,7) circle (0.15) node[above right] {\small $v_3$};
\draw[-, very thick, UQpurple] (0,4) -- (0,8) -- node[midway,above=-2pt] {\tiny $t_3$} (1,8) -- node[midway,above=-2pt] {\tiny $t_3$} (2,8);
\fill[UQpurple] (0,4) circle (0.15) node[below left] {\small $u_4$};
\fill[UQpurple] (2,8) circle (0.15) node[above right] {\small $v_4$};
\end{tikzpicture}
\quad \longmapsto \quad
\ytableaushort{1134,311,12,33}  *[*(white)]{4,1} *[*(darkred!40)]{4,3,2,2}\,.
\]
We have indicated the extension of the lattice model in the fixed region with the dashed black lines. The portion of the NILP below of the dotted line is the rectangular part and the portion above is the jagged part.
\end{ex}

In the remainder of the section, we will use the NILP and lattice model interpretations to prove a number of additional identities.

\subsection{Vertex models and the Yang--Baxter equation}

A (vertex) model will be a subset of the rectangular grid such that each crossing will be a \defn{vertex} with half edges on the boundary and each edge between vertices or half edge will be labeled. A single vertex with a given set of labels on its half edges will be given a function, called its \defn{Boltzmann weight}, which depends on a single parameter $z$. We will also require a vertex that has been tilted by $\pi/4$ whose Boltzmann weight depends on two parameters $z_i, z_j$. The single vertices
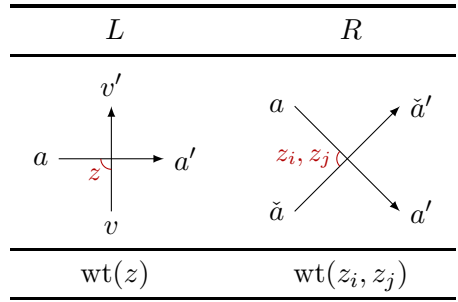
\[
\begin{array}{c@{\hspace{20pt}}c}
\\\toprule
L & R
\\\midrule
\vertex{v'}{a'}{v}{a}
&
\rmatrix{a}{\check{a}'}{a'}{\check{a}}
\\\midrule
\wt(z)
&
\wt(z_i, z_j)
\\\bottomrule
\end{array}
\]
will be graphical representations of the maps
\[
\begin{array}{c@{\hspace{40pt}}c}
L \colon A_z \otimes V \to V \otimes A_z, & R \colon A_{z_i} \otimes \check{A}_{z_j} \to \check{A}_{z_j} \otimes A_{z_i},
\\ a \otimes v \mapsto \wt(z) v' \otimes a', & a \otimes \check{a} \mapsto \wt(z_i, z_j) \check{a}' \otimes a'.
\end{array}
\]
These maps $L$ and $R$ are called the \defn{$L$-matrix} and \defn{$R$-matrix}, respectively.
The space $V$ is called the \defn{quantum space}, and the space $A_z$ (and $\check{A}_z$) is called the \defn{auxiliary space} with spectral parameter $z$.
In this paper, we both take the auxiliary and quantum spaces to be $\CC^2$, and we denote the basis vectors $e_0, e_1$ as $0, 1$ in the graphical descriptions of the $R$- and $L$-matrices and partition functions.
For more on vertex models, see, \textit{e.g.},~\cite[Ch.~19]{BBF11b}.

In our vertex models, often we will fix boundary conditions (\textit{i.e.}, the half-edge labels).
For a model $\fM$ a \defn{state} will be an assignment of labels to the edges such that every vertex has a non-zero Boltzmann weight, where we consider vertices that do not appear in the $L$-matrix to have a Boltzmann weight of $0$. Suppose all of the Boltzmann weights of $\fM$ have a parameter using an indeterminate of $\xx$, then the \defn{weight} of a state $S \in \fM$ will be the product of all of the Boltzmann weights of the vertices in the state. The \defn{partition function} is the sum of the weights of all possible states of $\fM$.

\begin{dfn}[Yang--Baxter equation]
The \defn{Yang--Baxter equation}, written in the $RLL$-equation form, states that
the partition functions of the following two models are equal for any fixed boundary values $\{a, b, c, d, e, f\}$:
\begin{equation}
\label{eq:RLL_relation}
\begin{tikzpicture}[>=latex,baseline=0,scale=0.8]
\draw[->] (-1,-1) node[anchor=east] {$f$} -- (1,1) -- (3,1) node[anchor=west] {$c$};
\draw[->] (-1,1) node[anchor=east] {$a$} -- (1,-1) -- (3,-1) node[anchor=west] {$d$};
\draw[->] (2,-2) node[anchor=north] {$e$} -- (2,2) node[anchor=south] {$b$};
\draw[color=darkred] (2-0.2,-1) arc (180:270:0.2);
\draw (2,-1) node[anchor=north east,color=darkred] {\small $z_j$};
\draw[color=darkred] (2-0.2,1) arc (180:270:0.2);
\draw (2,1) node[anchor=north east,color=darkred] {\small $z_i$};
\draw[color=darkred] (135:0.2) arc (135:225:0.2);
\draw (-0.1,0) node[anchor=east,color=darkred] {\small $z_i,z_j$};
\end{tikzpicture}
\hspace{5em}
\begin{tikzpicture}[>=latex,baseline=0,scale=0.8]
\draw[->] (-1,-1) node[anchor=east] {$f$} -- (1,-1) -- (3,1) node[anchor=west] {$c$};
\draw[->] (-1,1) node[anchor=east] {$a$} -- (1,1) -- (3,-1) node[anchor=west] {$d$};
\draw[->] (0,-2) node[anchor=north] {$e$} -- (0,2) node[anchor=south] {$b$};
\draw[color=darkred] (-0.2,-1) arc (180:270:0.2);
\draw (0,-1) node[anchor=north east,color=darkred] {\small $z_i$};
\draw[color=darkred] (-0.2,1) arc (180:270:0.2);
\draw (0,1) node[anchor=north east,color=darkred] {\small $z_j$};
\draw[color=darkred] (2,0) + (135:0.2) arc (135:225:0.2);
\draw (2-0.1,0) node[anchor=east,color=darkred] {\small $z_i,z_j$};
\end{tikzpicture}
\end{equation}
The name $RLL$-equation comes from the fact it is an equality between two partition functions, each of which is constructed from an $R$-matrix and two $L$-matrices.
Any model built from an $L$-matrix such that there exists an $R$-matrix satisfying the $RLL$-equation will be called \defn{integrable}.
\end{dfn}

For a vertex model, we define operators $A(z), B(z), C(z), D(z)$, where $z$ is a spectral parameter, acting on the tensor product of
quantum spaces as follows:
\begin{align*}
A(z) & = \begin{tikzpicture}[>=latex,scale=0.8,baseline=-3]
\draw[->] (-1,0) node [anchor=east] {$0$} -- (3,0);
\draw[dotted] (3,0) -- (5,0);
\draw[->] (4,0) -- (6,0) node [anchor=west] {$0$};
\foreach \x in {0,2,5} {
  \draw[->] (\x,-1) -- (\x,1);
  \draw[color=darkred] (\x-0.2,0) arc (180:270:0.2);
  \draw (\x,0) node[anchor=north east, color=darkred] {\small $z$};
}
\end{tikzpicture},
&
B(z) & = \begin{tikzpicture}[>=latex,scale=0.8,baseline=-3]
\draw[->] (-1,0) node [anchor=east] {$1$} -- (3,0);
\draw[dotted] (3,0) -- (5,0);
\draw[->] (4,0) -- (6,0) node [anchor=west] {$0$};
\foreach \x in {0,2,5} {
  \draw[->] (\x,-1) -- (\x,1);
  \draw[color=darkred] (\x-0.2,0) arc (180:270:0.2);
  \draw (\x,0) node[anchor=north east, color=darkred] {\small $z$};
}
\end{tikzpicture},
\\
C(z) & = \begin{tikzpicture}[>=latex,scale=0.8,baseline=-3]
\draw[->] (-1,0) node [anchor=east] {$0$} -- (3,0);
\draw[dotted] (3,0) -- (4,0);
\draw[->] (4,0) -- (6,0) node [anchor=west] {$1$};
\foreach \x in {0,2,5} {
  \draw[->] (\x,-1) -- (\x,1);
  \draw[color=darkred] (\x-0.2,0) arc (180:270:0.2);
  \draw (\x,0) node[anchor=north east, color=darkred] {\small $z$};
}
\end{tikzpicture},
&
D(z) & = \begin{tikzpicture}[>=latex,scale=0.8,baseline=-3]
\draw[->] (-1,0) node [anchor=east] {$1$} -- (3,0);
\draw[dotted] (3,0) -- (5,0);
\draw[->] (4,0) -- (6,0) node [anchor=west] {$1$};
\foreach \x in {0,2,5} {
  \draw[->] (\x,-1) -- (\x,1);
  \draw[color=darkred] (\x-0.2,0) arc (180:270:0.2);
  \draw (\x,0) node[anchor=north east, color=darkred] {\small $z$};
}
\end{tikzpicture},
\end{align*}
where the unlabeled edges between vertices mean that
their labels are not fixed to either 0 or 1.
These operators arise from the monodromy matrix, whose trace along the auxiliary space results in the transfer matrix.

Next, we recall from~\cite{MS13,MS14,GK17,WZJ16} an integrable five-vertex model such that the resulting partition function is a Grothendieck polynomial (with an appropriate gauge transformation).
For this model, the Boltzmann weights for the $L$-matrix in Table~\ref{table:L_matrix_Grothendieck} and the corresponding $R$-matrix that satisfies the Yang--Baxter equation is given in Table~\ref{table:R_matrix_Grothendieck}.

\begin{table}
\[
\begin{array}{ccccc}
\toprule
\ta_1 & \ta_2 & \tb_2 & \tc_1 & \tc_2
\\\midrule
\vertex{0}{0}{0}{0}
&
\vertex{1}{1}{1}{1}
&
\vertex{0}{1}{0}{1}
&
\vertex{1}{0}{0}{1}
&
\vertex{0}{1}{1}{0}
\\\midrule
1 & 1 & z& 1 & 1 + \beta z
\\\bottomrule
\end{array}
\]
\caption{The $L$-matrix for the modified fermionic five-vertex model.}
\label{table:L_matrix_Grothendieck}
\end{table}

\begin{table}
\[
\begin{array}{ccccc}
\toprule
\rmatrix{0}{0}{0}{0}
&
\rmatrix{1}{1}{1}{1}
&
\rmatrix{1}{0}{1}{0}
&
\rmatrix{1}{1}{0}{0}
&
\rmatrix{0}{0}{1}{1}
\\\midrule
1 + \beta z_i &1 + \beta z_i & z_j - z_i & 1 + \beta z_i  & 1 + \beta z_j
\\\bottomrule
\end{array}
\]
\caption{The $R$-matrix for the modified fermionic five-vertex model.}
\label{table:R_matrix_Grothendieck}
\end{table}

\begin{remark}
We can give the canonical Grothendieck polynomials from~\cite{Yel17} using the substitutions $\beta \mapsto \alpha + \beta$ and $z \mapsto \frac{z}{1 - \alpha z}$ by~\cite[Prop.~3.4]{Yel17}. As a consequence~\cite[Prop.~3.4]{MS14} is equivalent to~\cite[Prop.~8.8]{Yel17}.
\end{remark}

Using this $L$-matrix and $R$-matrix, we define a model $\fM_{\lambda}$ on an $n \times m$ grid such that the top boundary is given by the $01$-sequence of $\lambda$, the right and bottom boundary is all $0$, and the left boundary is all $1$. Thus, we have the following.

\begin{prop}[{\cite{MS13,MS14}}]
\label{prop:integrability}
The model $\fM_{\lambda}$ is integrable with the partition function
\[
Z(\fM_{\lambda}; \xx) = \G_{\lambda}(\xx; \beta).
\]
\end{prop}

In this model, we can show the operators $A(z)$ and $A(z')$, and more generally the transfer matrices, commute by repeatedly applying the $RLL$-equation.
This is called the \defn{standard train argument} and shows the two partition functions
\begin{equation}
\label{eq:standard_train}
\begin{tikzpicture}[>=latex,baseline=0,scale=0.8]
\draw[->] (-1,-1) node[anchor=east] {$f$} -- (1,1) -- (3,1);
\draw[dotted] (3,1) -- (4,1);
\draw[->] (4,1) -- (6,1) node[anchor=west] {$c$};
\draw[->] (-1,1) node[anchor=east] {$a$} -- (1,-1) -- (3,-1);
\draw[dotted] (3,-1) -- (4,-1);
\draw[->] (4,-1) -- (6,-1) node[anchor=west] {$d$};
\draw[->] (2,-2) node[anchor=north] {$e_1$} -- (2,2) node[anchor=south] {$b_1$};
\draw[->] (5,-2) node[anchor=north] {$e_n$} -- (5,2) node[anchor=south] {$b_n$};
\foreach \x in {2,5}
{
  \draw[color=darkred] (\x-0.2,-1) arc (180:270:0.2);
  \draw (\x,-1) node[anchor=north east,color=darkred] {\small $z_j$};
  \draw[color=darkred] (\x-0.2,1) arc (180:270:0.2);
  \draw (\x,1) node[anchor=north east,color=darkred] {\small $z_i$};
}
\draw[color=darkred] (135:0.2) arc (135:225:0.2);
\draw (-0.1,0) node[anchor=east,color=darkred] {\small $z_i,z_j$};
\end{tikzpicture}
\hspace{5em}
\begin{tikzpicture}[>=latex,baseline=0,scale=0.8]
\draw[->] (-4,-1) node[anchor=east] {$f$} -- (-2,-1);
\draw[dotted] (-2,-1) -- (-1,-1);
\draw[->] (-1,-1) -- (1,-1) -- (3,1) node[anchor=west] {$c$};
\draw[->] (-4,1) node[anchor=east] {$a$} -- (-2,1);
\draw[dotted] (-2,1) -- (-1,1);
\draw[->] (-1,1) -- (1,1) -- (3,-1) node[anchor=west] {$d$};
\draw[->] (-3,-2) node[anchor=north] {$e_1$} -- (-3,2) node[anchor=south] {$b_1$};
\draw[->] (0,-2) node[anchor=north] {$e_n$} -- (0,2) node[anchor=south] {$b_n$};
\foreach \x in {0,-3}
{
  \draw[color=darkred] (\x-0.2,-1) arc (180:270:0.2);
  \draw (\x,-1) node[anchor=north east,color=darkred] {\small $z_i$};
  \draw[color=darkred] (\x-0.2,1) arc (180:270:0.2);
  \draw (\x,1) node[anchor=north east,color=darkred] {\small $z_j$};
}
\draw[color=darkred] (2,0) + (135:0.2) arc (135:225:0.2);
\draw (2-0.1,0) node[anchor=east,color=darkred] {\small $z_i,z_j$};
\end{tikzpicture}
\end{equation}
are equal for any fixed boundary values.

\begin{remark}
\label{rem:trivial_coloring}
We note that we can transform this into a colored model as follows.
We can considering paths on the (directed) edges labeled by $1$, and we color these edges such that the two input colors equals the two output colors as a set.
Thus there are two possibilities for $\ta_2$, but we only consider the vertex where the colors touch and do not cross.
We call this coloring the \defn{trivial coloring}.
The states of this model are naturally in bijection with the NILPs of a Schur function, with the $i$-th (colored) paths corresponding to each other.
\end{remark}

\subsection{Model from nonintersecting lattice paths}
\label{sec:nilp_model}

We begin by defining the $L$-matrix for our model by using Table~\ref{table:L_matrix_nilp}, and the corresponding $R$-matrix is given in Table~\ref{table:R_matrix_nilp}. By a direct finite computation, we have that these satisfy the Yang--Baxter equation.

\begin{table}
\[
\begin{array}{ccccc}
\toprule
\ta_1 & \tb_1 & \tb_2 & \tc_1 & \tc_2
\\\midrule
\vertex{0}{0}{0}{0}
&
\vertex{1}{0}{1}{0}
&
\vertex{0}{1}{0}{1}
&
\vertex{1}{0}{0}{1}
&
\vertex{0}{1}{1}{0}
\\\midrule
1 & 1 & z & z & 1
\\\bottomrule
\end{array}
\]
\caption{The $L$-matrix for the NILP vertex model.}
\label{table:L_matrix_nilp}
\end{table}

\begin{table}
\[
\begin{array}{ccccc}
\toprule
\rmatrix{0}{0}{0}{0}
&
\rmatrix{1}{0}{1}{0}
&
\rmatrix{0}{0}{1}{1}
&
\rmatrix{1}{1}{0}{0}
&
\rmatrix{1}{1}{1}{1}
\\\midrule
z_j & z_j - z_i & z_i & z_j & z_i
\\\bottomrule
\end{array}
\]
\caption{The $R$-matrix for the NILP vertex model.}
\label{table:R_matrix_nilp}
\end{table}

\begin{prop}
\label{prop:refined_dual_integrable}
The $L$-matrix in Table~\ref{table:L_matrix_nilp} and the $R$-matrix in Table~\ref{table:R_matrix_nilp} satisfy the $RLL$-equation.
\end{prop}


We construct a model $\fN_{\lambda}$ (which is integrable by Proposition~\ref{prop:refined_dual_integrable}) with the following non-standard boundary conditions. Consider $n$ rows of $\lambda_1 + \ell$ vertices, then stack on top a row of length $\lambda_2 + (\ell-1)$ aligned to the left, and so on for every part of $\lambda$. The left and right boundary conditions will all be $0$. The bottom boundary will consist of $\ell$ $1$'s as the leftmost columns and the remaining entries being $0$. The top boundary will have a $1$ at column $\lambda_i + \ell + 1 - i$ for all $1 \leq i \leq \ell$ and the remaining entries $0$.
For $1 \leq i \leq n$, the $i$-th row from the bottom will have a spectral parameter of $x_i$ and the $(n + i)$-th row from the bottom will have a spectral parameter of $t_i$.
For brevity, we say the spectral parameters of the model are $\xx \sqcup \bt$.

\begin{ex}
\label{ex:NILP_5V_model}
For the NILP from Example~\ref{ex:NILP_tableaux}, the corresponding vertex model is given by
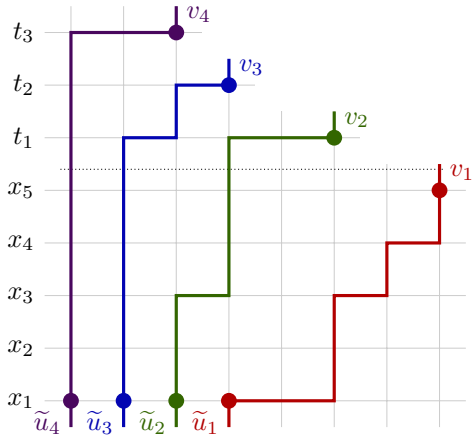
\[
\begin{tikzpicture}[>=latex,baseline=2.7cm,scale=0.7]
\foreach \y in {1,2,3,4,5}
  \draw node[anchor=east] at (-0.5,\y)  {\small $x_{\y}$};
\foreach \y in {1,2,3}
  \draw node[anchor=east] at (-0.5,\y+5)  {\small $t_{\y}$};
\draw[very thin, black!20] (-0.5, 1-0.5) grid (7.5, 8.5);
\draw[densely dotted] (-0.2,5.4) -- (7.1,5.4);
\fill[white] (2.5,7.5) rectangle (7.6,8.6);
\fill[white] (3.5,6.5) rectangle (7.6,7.6);
\fill[white] (5.5,5.5) rectangle (7.6,6.6);
\draw[-, very thick, darkred] (3,0.5) -- (3,1) -- (5,1) -- (5,3) -- (6,3) -- (6,4) -- (7,4) -- (7,5.5);
\fill[darkred] (3,1) circle (0.15) node[below left] {\small $\widetilde{u}_1$};
\fill[darkred] (7,5) circle (0.15) node[above right] {\small $v_1$};
\draw[-, very thick, dgreencolor] (2,0.5) -- (2,3) -- (3,3) -- (3,6) -- (4,6) -- (5,6) -- (5,6.5);
\fill[dgreencolor] (2,1) circle (0.15) node[below left] {\small $\widetilde{u}_2$};
\fill[dgreencolor] (5,6) circle (0.15) node[above right] {\small $v_2$};
\draw[-, very thick, dbluecolor] (1,0.5) -- (1,6) -- (2,6) -- (2,7) -- (3,7) -- (3,7.5);
\fill[dbluecolor] (1,1) circle (0.15) node[below left] {\small $\widetilde{u}_3$};
\fill[dbluecolor] (3,7) circle (0.15) node[above right] {\small $v_3$};
\draw[-, very thick, UQpurple] (0,0.5) -- (0,8) -- (1,8) -- (2,8) -- (2,8.5);
\fill[UQpurple] (0,1) circle (0.15) node[below left] {\small $\widetilde{u}_4$};
\fill[UQpurple] (2,8) circle (0.15) node[above right] {\small $v_4$};
\end{tikzpicture}
\]
where each vertex is an $L$-matrix with the (half) edges with a color are $1$ (and those without are a $0$). Thus, the Boltzmann weight for this state is $t_1^3 t_2 t_3^2 x_1^2 x_3^2 x_4$.
\end{ex}

\begin{thm}
\label{thm:refined_dual_model}
The partition function of the model $\fN_{\lambda}$ is a refined dual Grothendieck polynomial:
\[
Z(\fN_{\lambda}; \xx \sqcup \bt) = \dG_{\lambda}(\xx; \bt).
\]
\end{thm}

\begin{proof}
We note that the rectangular part is precisely the vertex model whose partition function is a Schur function $s_{\mu}$. Similarly, the partition function of the jagged part is precisely $e_{\lambda}^{\mu}(\bt)$.
In particular, each state in the model corresponds to a semistandard Young tableau of shape $\mu$ and an elegant tableau of shape $\lambda/\mu$ as each state is a NILP and using the bijection described in Section~\ref{sec:bg_tableaux}.
Hence, the claim follows from Theorem~\ref{thm:dualG_schur_decomposition}.
\end{proof}

As mentioned in the beginning of the section, from~\cite{LN14}, a flagged Schur function coincides with a certain multi-Schur function, whose natural NILP interpretation coincides with that from~\cite[Thm.~3.2]{CLL02}.
Therefore, we obtain that refined dual Grothendieck polynomials are particular multi-Schur functions.
We could also prove this by noting that refined dual Grothendieck polynomials as flagged Schur functions using the LGV lemma and the argument in~\cite{Wachs85}.

\begin{cor}
\label{cor:refined_multiSchur}
We have
\[
\dG_{\lambda}(\xx; \bt) = s_{\lambda}(\xx, \xx + (t_1), \xx + (t_1, t_2), \dotsc, \xx + \bt).
\]
\end{cor}

\subsection{Identities}

Using our vertex model, we will now (re)prove some identities.
All of these results follow from the partition function of our model being a
refined dual Grothendieck polynomial (Theorem~\ref{thm:refined_dual_model}) or the corresponding NILP interpretation.

Our first identity is a Jacobi--Trudi formula for refined dual Grothendieck polynomials, which is a dual version of~\cite[Eq.~(73)]{Yel17}, from the LGV lemma with numbering the paths from the right (see also~\cite{Wachs85}).
This was shown for $\bt = \beta$ in~\cite[Cor.~10.3]{Yel17} and~\cite[Prop.~4.4]{Iwao19}, which was also implicit from the results of~\cite{CLL02,LN14}.

\begin{cor}
We have
\[
\dG_{\lambda}(\xx; \bt) = \det \bigl[
h_{\lambda_i+j-i}(\xx, t_1, \dotsc, t_{i-1})
\bigr]_{i,j=1}^n.
\]
\end{cor}

Using Corollary~\ref{cor:refined_multiSchur}, we can also obtain a dual Jacobi--Trudi formula by refining the computation for~\cite[Eq.~(5)]{LN14}.

\begin{cor}
We have
\[
\dG_{\lambda}(\xx; \bt) = \det \bigl[ e_{\lambda'_i+j-i}(\xx, t_1, \dotsc, t_{\lambda'_i-1})
\bigr]_{i,j=1}^n.
\]
\end{cor}

Next we give a new analog of the Cauchy identity for refined dual Grothendieck polynomials.

\begin{cor}[Cauchy identity]
\label{cor:refined_Cauchy}
We have
\[
s_{m^{\ell}}(\xx, \bt, \yy) = \sum_{\lambda \subseteq m^{\ell}} \dG_{\lambda}(\xx; \bt) \dG_{\lambda^{\comp}}(\yy; \bt^{\comp}),
\]
where $\bt^{\comp} = (t_{\ell-1}, \dotsc, t_1)$.
\end{cor}

\begin{proof}
Note that for a NILP, we can choose either the starting corner or the ending corner to carry the additional $z$ weight. So we can also use the following $L$-matrix for the model:
\[
\begin{array}{ccccc}
\toprule
\vertex{0}{0}{0}{0}
&
\vertex{1}{0}{1}{0}
&
\vertex{0}{1}{0}{1}
&
\vertex{0}{1}{1}{0}
&
\vertex{1}{0}{0}{1}
\\\midrule
1 & 1 & z & z & 1
\\\bottomrule
\end{array}
\]
which is the rotation by $\pi$ of the vertices from Table~\ref{table:L_matrix_nilp}.
Now we form the model for $s_{m^{\ell}}(\xx, \bt, \yy)$ by taking a $(2n+\ell-1) \times (m + \ell)$ rectangle with boundary conditions of $1$ in the lowest (resp.\ highest) $\ell$ columns on the bottom (resp.\ top) and $0$ everywhere else. The spectral parameters will be $\xx \sqcup \bt \sqcup \yy$.
We form the sum by taking cuts through the $\bt$ variables, where the cut for the $t_i$ variable comes from the $\tb_1$ or $\tc_1$ vertex in the $(i+1)$-th path.
The partition function of the model to the left of the cut will be $\dG_{\lambda}(\xx; \bt)$ by Theorem~\ref{thm:refined_dual_model}. The portion to the right of the cut will be $\dG_{\lambda^{\comp}}(\yy; \bt^{\comp})$ by Theorem~\ref{thm:refined_dual_model} using the previous fact about an equivalent model.
\end{proof}

\begin{ex}
Let $n = 2$. We give an example of a state from $s_{6^4}(\xx, \bt, \yy)$ with the dotted line indicating the cut for the decomposition from Corollary~\ref{cor:refined_Cauchy} and their corresponding pair of semistandard tableaux:
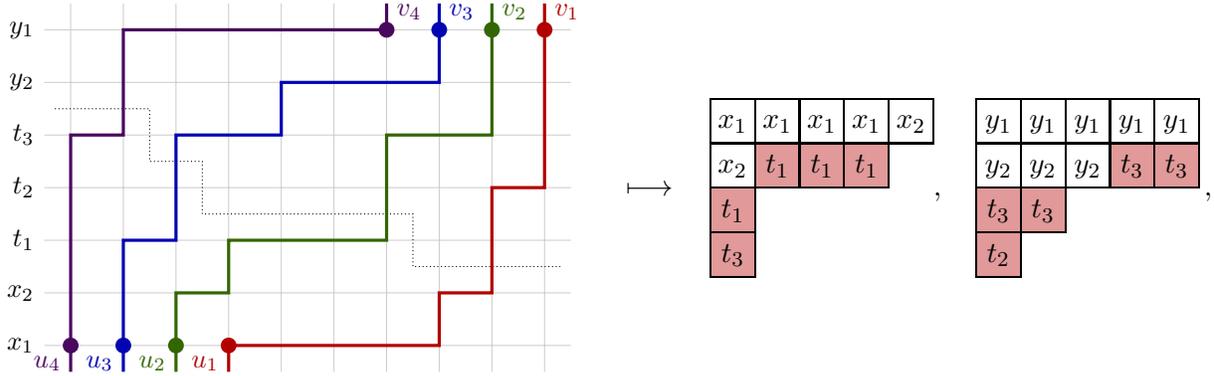
\[
\begin{tikzpicture}[>=latex,baseline=2.7cm,scale=0.7]
\foreach \y in {1,2} {
  \draw node[anchor=east] at (-0.5,\y)  {\small $x_{\y}$};
  \draw node[anchor=east] at (-0.5,8-\y)  {\small $y_{\y}$};
}
\foreach \y in {1,2,3}
  \draw node[anchor=east] at (-0.5,\y+2)  {\small $t_{\y}$};
\draw[very thin, black!20] (-0.5, 1-0.5) grid (9.5, 7.5);
\draw[densely dotted] (-0.3,5.5) -- (1.5,5.5) -- (1.5,4.5) -- (2.5,4.5) -- (2.5,3.5) -- (6.5,3.5) -- (6.5,2.5) -- (9.3,2.5);
\draw[-, very thick, darkred] (3,0.5) -- (3,1) -- (7,1) -- (7,2) -- (8,2) -- (8,4) -- (9,4) -- (9,7.5);
\fill[darkred] (3,1) circle (0.15) node[below left] {\small $u_1$};
\fill[darkred] (9,7) circle (0.15) node[above right] {\small $v_1$};
\draw[-, very thick, dgreencolor] (2,0.5) -- (2,2) -- (3,2) -- (3,3) -- (6,3) -- (6,5) -- (8,5) -- (8,7.5);
\fill[dgreencolor] (2,1) circle (0.15) node[below left] {\small $u_2$};
\fill[dgreencolor] (8,7) circle (0.15) node[above right] {\small $v_2$};
\draw[-, very thick, dbluecolor] (1,0.5) -- (1,3) -- (2,3) -- (2,5) -- (4,5) -- (4,6) -- (7,6) -- (7,7.5);
\fill[dbluecolor] (1,1) circle (0.15) node[below left] {\small $u_3$};
\fill[dbluecolor] (7,7) circle (0.15) node[above right] {\small $v_3$};
\draw[-, very thick, UQpurple] (0,0.5) -- (0,5) -- (1,5) -- (1,7) -- (6,7) -- (6,7.5);
\fill[UQpurple] (0,1) circle (0.15) node[below left] {\small $u_4$};
\fill[UQpurple] (6,7) circle (0.15) node[above right] {\small $v_4$};
\end{tikzpicture}
\quad \longmapsto \quad
\ytableaushort{{x_1}{x_1}{x_1}{x_1}{x_2},{x_2}{t_1}{t_1}{t_1},{t_1},{t_3}}  *[*(white)]{5,1} *[*(darkred!40)]{5,4,1,1},
\quad
\ytableaushort{{y_1}{y_1}{y_1}{y_1}{y_1},{y_2}{y_2}{y_2}{t_3}{t_3},{t_3}{t_3},{t_2}}  *[*(white)]{5,3} *[*(darkred!40)]{5,5,2,1}\,,
\]
where the second tableau is semistandard with respect to the alphabet $y_1 < y_2 < t_3 < t_2 < t_1$.
Notice that when we rotate the second tableau by $\pi$ and join it to the first tableau, we obtain a semistandard tableau of (rectangular) shape $6^4$ for the alphabet $x_1 < x_2 < t_1 < t_2 < t_3 < y_2 < y_1$.
\end{ex}

We also have a generalized Littlewood identity, where the $\bt = 1$ specialized version was first given in~\cite[Cor.~3.5]{Yel19II}.

\begin{cor}[Littlewood identity]
\label{cor:littlewood_identity}
Let $m \geq \ell$.
We have
\[
s_{m^{\ell}}(\xx, \bt, t_{\ell}) = \sum_{\lambda \subseteq m^{\ell}} \prod_{i=1}^{\ell} t_i^{m-\lambda_i} \dG_{\lambda}(\xx; \bt).
\]
\end{cor}

\begin{proof}
We form a $(n+\ell) \times (m+\ell)$ rectangular grid with boundary conditions of the left $\ell$ columns on the bottom and right $\ell$ columns on the top being $1$, with $0$ everywhere else.
Let the spectral parameters be $\xx \sqcup \bt \sqcup (t_{\ell})$.
The resulting partition function is the Schur function $s_{m^{\ell}}(\xx,\bt, t_{\ell})$.
By taking the model corresponding to $\dG_{\lambda}(\xx; \bt)$, we note that the remaining portion of the rectangle model becomes fixed and has weight $t_1^{m-\lambda_1} t_2^{m-\lambda_2} \dotsm t_{\ell}^{m-\lambda_{\ell}}$.
We note that the shape $\lambda$ is determined by the first vertical step of the paths as in the proof of Corollary~\ref{cor:refined_Cauchy}.
\end{proof}

\begin{ex}
\label{ex:schur_to_dual_grothendiecks}
Consider $n = 3$.  Let $m = 6$ and $\ell = 4$. A state from $s_{8^4}(\xx, \bt, t_{\ell})$ is
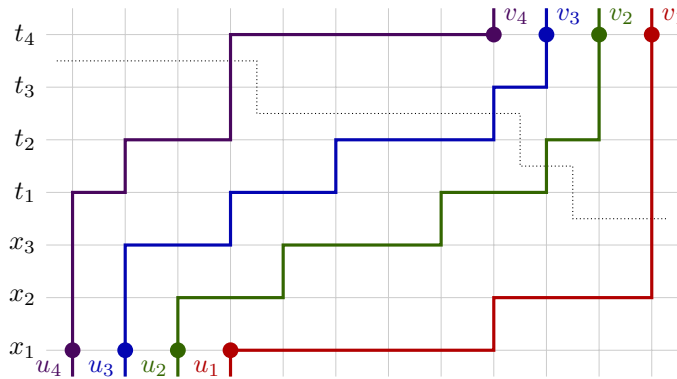
\[
\begin{tikzpicture}[>=latex,baseline=2.7cm,scale=0.7]
\foreach \y in {1,2,3} {
  \draw node[anchor=east] at (-0.5,\y)  {\small $x_{\y}$};
}
\foreach \y in {1,2,3,4}
  \draw node[anchor=east] at (-0.5,\y+3)  {\small $t_{\y}$};
\draw[very thin, black!20] (-0.5, 1-0.5) grid (11.5, 7.5);
\draw[densely dotted] (-0.3,6.5) -- (3.5,6.5) -- (3.5,5.5) -- (8.5,5.5) -- (8.5,4.5) -- (9.5,4.5) -- (9.5,3.5) -- (11.3,3.5);
\draw[-, very thick, darkred] (3,0.5) -- (3,1) -- (8,1) -- (8,2) -- (9,2) -- (11,2) -- (11,7.5);
\fill[darkred] (3,1) circle (0.15) node[below left] {\small $u_1$};
\fill[darkred] (11,7) circle (0.15) node[above right] {\small $v_1$};
\draw[-, very thick, dgreencolor] (2,0.5) -- (2,2) -- (4,2) -- (4,3) -- (7,3) -- (7,4) -- (9,4) -- (9,5) -- (10,5) -- (10,7.5);
\fill[dgreencolor] (2,1) circle (0.15) node[below left] {\small $u_2$};
\fill[dgreencolor] (10,7) circle (0.15) node[above right] {\small $v_2$};
\draw[-, very thick, dbluecolor] (1,0.5) -- (1,3) -- (3,3) -- (3,4) -- (5,4) -- (5,5) -- (8,5) -- (8,6) -- (9,6) -- (9,7.5);
\fill[dbluecolor] (1,1) circle (0.15) node[below left] {\small $u_3$};
\fill[dbluecolor] (9,7) circle (0.15) node[above right] {\small $v_3$};
\draw[-, very thick, UQpurple] (0,0.5) -- (0,4) -- (1,4) -- (1,5) -- (3,5)  -- (3,7) -- (8,7) -- (8,7.5);
\fill[UQpurple] (0,1) circle (0.15) node[below left] {\small $u_4$};
\fill[UQpurple] (8,7) circle (0.15) node[above right] {\small $v_4$};
\end{tikzpicture}
\]
where the cut is given by the dotted line. Thus we have $\lambda = (8, 7, 7, 3)$.
\end{ex}

In the previous proof, if we remove the topmost row (but keeping the same boundary conditions), then not only does the upper-right region become fixed, there is only one allowable partition of $\lambda$ being the rectangle allowed by the pigeonhole principle.
This is exactly analogous to how we are able to fix the lower-left portion (and why the model is equivalent to the NILPs).
Therefore, we obtain the following, where the $t_1 = \cdots = t_n = 1$ specialized version was first given in~\cite[Lemma~3.4]{Yel19II} (the proof is essentially equivalent by the LGV lemma).

\begin{cor}
\label{cor:refined_coincidence_lemma}
We have
\[
s_{m^{\ell}}(\xx, \bt) = \dG_{m^{\ell}}(\xx; \bt).
\]
\end{cor}

Note that Corollary~\ref{cor:refined_coincidence_lemma} implies that $\dG_{\lambda}(\xx; \bt)$ is symmetric in $\bt$ when $\lambda$ is a rectangle (which can be seen by the standard train argument). Generalizing these fixed regions for when $\lambda$ is a more general shape, we obtain the following.

\begin{cor}
\label{cor:refined_dual_symmetries}
Suppose $\lambda_i = \lambda_{i+1}$, then $\dG_{\lambda}(\xx; \bt)$ is symmetric in $t_{i-1}$ and $t_i$, where we can take $t_0 = x_j$ for any $1 \leq j < \ell$.
\end{cor}

\begin{ex}
Consider the vertex model with $\lambda = 4422$ and $n = 5$. Then one such state in this extended model and the corresponding reverse plane partition is
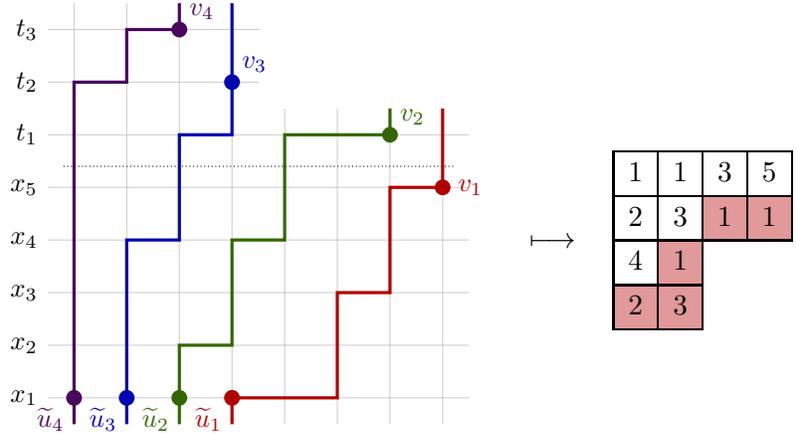
\[
\begin{tikzpicture}[>=latex,baseline=2.7cm,scale=0.7]
\foreach \y in {1,2,3,4,5}
  \draw node[anchor=east] at (-0.5,\y)  {\small $x_{\y}$};
\foreach \y in {1,2,3}
  \draw node[anchor=east] at (-0.5,\y+5)  {\small $t_{\y}$};
\draw[very thin, black!20] (-0.5, 1-0.5) grid (7.5, 8.5);
\draw[densely dotted] (-0.2,5.4) -- (7.2,5.4);
\fill[white] (3.5,6.5) rectangle (7.6, 8.6);
\draw[-, very thick, darkred] (3,0.5) -- (3,1) -- (5,1) -- (5,3) -- (6,3) -- (6,5) -- (7,5) -- (7,6.5);
\fill[darkred] (3,1) circle (0.15) node[below left] {\small $\widetilde{u}_1$};
\fill[darkred] (7,5) circle (0.15) node[right=2pt] {\small $v_1$};
\draw[-, very thick, dgreencolor] (2,0.5) -- (2,2) -- (3,2) -- (3,4) -- (4,4) -- (4,6) -- (6,6) -- (6,6.5);
\fill[dgreencolor] (2,1) circle (0.15) node[below left] {\small $\widetilde{u}_2$};
\fill[dgreencolor] (6,6) circle (0.15) node[above right] {\small $v_2$};
\draw[-, very thick, dbluecolor] (1,0.5) -- (1,4) -- (2,4) -- (2,6) -- (3,6) -- (3,8.5);
\fill[dbluecolor] (1,1) circle (0.15) node[below left] {\small $\widetilde{u}_3$};
\fill[dbluecolor] (3,7) circle (0.15) node[above right] {\small $v_3$};
\draw[-, very thick, UQpurple] (0,0.5) -- (0,7) -- (1,7) -- (1,8) -- (2,8) -- (2,8.5);
\fill[UQpurple] (0,1) circle (0.15) node[below left] {\small $\widetilde{u}_4$};
\fill[UQpurple] (2,8) circle (0.15) node[above right] {\small $v_4$};
\end{tikzpicture}
\quad \longmapsto \quad
\ytableaushort{1135,2311,41,23}  *[*(white)]{4,2,1} *[*(darkred!40)]{4,4,2,2}\,.
\]
Note that we can apply the $R$-matrix to any pair of (adjacent) rows except the second and the third.
We compute
\begin{align*}
\dG_{\lambda}(\xx; \bt) & = s_{4422}(\xx, t_1) + (t_2 + t_3) s_{4421}(\xx, t_1) + (t_2^2 + t_2 t_3 + t_3^2) s_{442}(\xx, t_1)
\\ & \hspace{20pt} + t_2 t_3 s_{4411}(\xx, t_1) + (t_2^2 t_3 + t_2 t_3^2) s_{441}(\xx, t_1) + t_2^2 t_3^2 s_{44}(\xx, t_1),
\end{align*}
which is clearly symmetric in $t_2$ and $t_3$.
\end{ex}

We also have a refined version of the branching rule of~\cite[Thm.~8.6]{Yel17} at $\alpha = 0$.

\begin{cor}[Branching rule]
\label{cor:branching_rule}
We have
\[
\dG_{\lambda}(\xx, \gamma; \bt) = \sum_{\mu \subseteq \lambda} \gamma^{\lambda_1 - \mu_1} t_1^{\lambda_2 - \mu_2} t_2^{\lambda_3 - \mu_3} \dotsm t_{\ell-1}^{\lambda_{\ell} - \mu_{\ell}} \dG_{\mu}(\xx; \gamma, \bt).
\]
\end{cor}

\begin{proof}
By setting $x_{n+1} = \gamma$, we can consider that as part of the $\bt$ variables.
Thus, the result can be seen by superimposing the model for $\dG_{\mu}(\xx; \gamma, \bt)$ onto that of $\dG_{\lambda}(\xx, \gamma; \bt)$ and noting that the vertices between the two boundaries are fixed and have weight $\gamma^{\lambda_1 - \mu_1} t_1^{\lambda_2 - \mu_2} \dotsm t_{\ell-1}^{\lambda_{\ell} - \mu_{\ell}}$.
In particular, note that since $\ell(\mu) \leq \ell(\lambda)$, we do not use the variable $t_{\ell-1}$ in $\dG_{\mu}(\xx; \gamma, \bt)$.
\end{proof}

\begin{ex}
Consider $\dG_{4322}(\xx, \gamma; \bt)$, where $\xx = (x_1, x_2, x_3, x_4)$, and the state
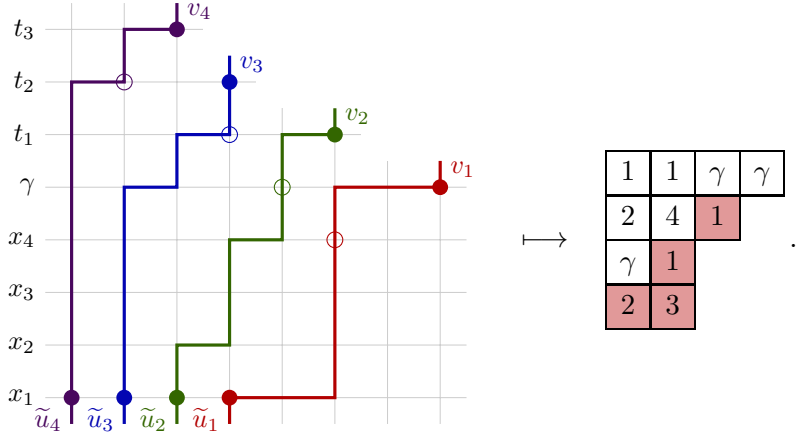
\[
\begin{tikzpicture}[>=latex,baseline=2.7cm,scale=0.7]
\foreach \y in {1,2,3,4}
  \draw node[anchor=east] at (-0.5,\y)  {\small $x_{\y}$};
\draw node[anchor=east] at (-0.5,5)  {\small $\gamma$};
\foreach \y in {1,2,3}
  \draw node[anchor=east] at (-0.5,\y+5)  {\small $t_{\y}$};
\draw[very thin, black!20] (-0.5, 1-0.5) grid (7.5, 8.5);
\fill[white] (2.5,7.5) rectangle (7.6,8.6);
\fill[white] (3.5,6.5) rectangle (7.6,7.6);
\fill[white] (5.5,5.5) rectangle (7.6,6.6);
\draw[-, very thick, darkred] (3,0.5) -- (3,1) -- (5,1) -- (5,5) -- (7,5) -- (7,5.5);
\fill[darkred] (3,1) circle (0.15) node[below left] {\small $\widetilde{u}_1$};
\fill[darkred] (7,5) circle (0.15) node[above right] {\small $v_1$};
\draw[-, very thick, dgreencolor] (2,0.5) -- (2,2) -- (3,2) -- (3,4) -- (4,4) -- (4,6) -- (5,6) -- (5,6.5);
\fill[dgreencolor] (2,1) circle (0.15) node[below left] {\small $\widetilde{u}_2$};
\fill[dgreencolor] (5,6) circle (0.15) node[above right] {\small $v_2$};
\draw[-, very thick, dbluecolor] (1,0.5) -- (1,5) -- (2,5) -- (2,6) -- (3,6) -- (3,7.5);
\fill[dbluecolor] (1,1) circle (0.15) node[below left] {\small $\widetilde{u}_3$};
\fill[dbluecolor] (3,7) circle (0.15) node[above right] {\small $v_3$};
\draw[-, very thick, UQpurple] (0,0.5) -- (0,7) -- (1,7) -- (1,8) -- (2,8) -- (2,8.5);
\fill[UQpurple] (0,1) circle (0.15) node[below left] {\small $\widetilde{u}_4$};
\fill[UQpurple] (2,8) circle (0.15) node[above right] {\small $v_4$};
\draw[darkred] (5,4) circle (0.15);
\draw[dgreencolor] (4,5) circle (0.15);
\draw[dbluecolor] (3,6) circle (0.15);
\draw[UQpurple] (1,7) circle (0.15);
\end{tikzpicture}
\quad \longmapsto \quad
\ytableaushort{11{\gamma}{\gamma},241,{\gamma}1,23}  *[*(white)]{4,2,1} *[*(darkred!40)]{4,3,2,2}\,.
\]
To perform the branching in Corollary~\ref{cor:branching_rule} for this state, we compute that $\mu = 2221$ by considering the last vertical step in each path.
We have indicated the endpoints for $\dG_{2221}(\xx; \gamma, \bt)$ as the open circles.
Therefore, the coefficient in the branching rule is $\gamma^2 t_1 t_3$.
\end{ex}

We remark that Corollary~\ref{cor:refined_coincidence_lemma} and Corollary~\ref{cor:branching_rule} imply Corollary~\ref{cor:littlewood_identity}.

We can also generalize Corollary~\ref{cor:refined_coincidence_lemma} by considering a more general shape $\lambda$.

\begin{cor}
\label{cor:generalized_coincidence}
Let $\nu$ be a partition and $\ell = \ell(\nu)$.
Let $\widetilde{\bt} = (t_1, \dotsc, t_m)$ for some $m \geq \ell - 1$.
Then we have
\[
s_{\nu}(\xx, \widetilde{\bt}) = \sum_{\lambda \subseteq \nu} p_{\nu}^{\lambda}(\widetilde{\bt}) \dG_{\lambda}(\xx; \bt),
\]
where $p_{\nu}^{\lambda}(\widetilde{\bt}) = \det \bigl[ h_{\nu_i-\lambda_j-i+j}(t_m, \dotsc, t_j) \bigr]_{i,j=1}^{\ell}$ and for $m = \ell - 1$, we consider $h_k(t_m, \dotsc, t_{\ell}) = \delta_{k0}$.
Furthermore, we have
\[
p_{\nu}^{\lambda}(\widetilde{\bt}) = \sum_{T} \widetilde{\bt}^{T},
\]
where the sum is over all semistandard skew tableaux $T$ of shape $\nu / \lambda$ with max entry $m$ and \emph{lower} flagging $f = (0,1,\dotsc,\ell-1)$, that is the smallest entry in row $i$ is strictly greater than $i$.
\end{cor}

\begin{proof}
The results follow from utilizing the same cut idea utilized in the proof of Corollary~\ref{cor:refined_Cauchy} and from the LGV lemma.
\end{proof}

Combinatorially, what we are doing is taking a semistandard tableau $T$ of shape $\nu$, taking the subtableau $T_{\lambda}$ of shape $\lambda$, and saying $T_{\lambda}$ contributes to $\dG_{\lambda}(\xx; \bt)$ if and only if $T_{\lambda}$ contains only the letters contributing to $\xx$ and $\bt$.
The remaining part of $T \setminus T_{\lambda}$ contributes to $p_{\nu}^{\lambda}(\widetilde{\bt})$.

For our last identity of this section, we need to introduce some additional notation for the Yang--Baxter algebra and list some properties that we require.
For the remainder of this section, let us denote the $i$-th quantum space from the left in the lattice model pictures as $V_i$, and denote the dual space as $V_i^*$.
We denote the basis vectors of $V_i$ and $V_i^*$ as $(e_0)_i,(e_1)_i$ and $(e_0)_i^*,(e_1)_i^*$, respectively.
We denote the $X$-operator acting on $V_1 \otimes V_2 \otimes \cdots \otimes V_m$ sometimes as $X_m(z)$ to emphasize the number of tensor factors.
Let $\binom{[n]}{k}$ be the set of $k$-element subsets of $\{1, \dotsc, n\}$, and for $S \in \binom{[n]}{k}$, let $\overline{S} := \{1,2,\dotsc,n \} \setminus S$.

We define $\delta_{i,\lambda}$ for a positive integer $i$ and a partition $\lambda$ to be 1 if there exists an integer $j$ such that $i=\lambda_j+\ell(\lambda)-j+1$,
and 0 otherwise.
Define $(e_{\lambda})_m^*$ as
\[
(e_{\lambda})_m^* := \bigotimes_{i=1}^m \bigl( (1-\delta_{i,\lambda}) (e_0)_i^*+\delta_{i,\lambda}(e_1)_i^* \bigr),
\]
where the tensor product is done with the natural order. Similarly, define
\[
(e)_{\ell} :=
\bigotimes_{i=1}^{\ell} (e_1)_i
\otimes
\bigotimes_{i=\ell+1}^{\ell+\lambda_1} (e_0)_i,
\qquad\qquad
(e)_0 := \bigotimes_{i=1}^{\ell+\lambda_1} (e_0)_i.
\]

From the $RLL$ relation (Proposition~\ref{prop:refined_dual_integrable}) with the standard train argument (Equation~\eqref{eq:standard_train}), we get the following commutation relations between the $A$- and $B$-operators.

\begin{cor}
\label{cor:AB_ops_commute}
We have
\begin{subequations}
\begin{align}
A(z_i)B(z_j)&=\frac{z_j}{z_j-z_i}B(z_j)A(z_i)-
\frac{z_j}{z_j-z_i}B(z_i)A(z_j), \label{commrelone} \\
A(z_j)B(z_i)&=\frac{z_i}{z_j}A(z_i)B(z_j), \label{commreltwo} \\
B(z_j)B(z_i)&=\frac{z_i}{z_j}B(z_i)B(z_j), \label{commrelthree} \\
A(z_j)A(z_i)&=A(z_i)A(z_j). \label{commrelfour}
\end{align}
\end{subequations}
\end{cor}

Define  $\widetilde{A}(z_i) := z_i A(z_i)$.
Then~\eqref{commreltwo} and~\eqref{commrelfour} can be rewritten as
\begin{subequations}
\label{eq:commrelsprime}
\begin{align}
\widetilde{A}(z_j)B(z_i)&=\widetilde{A}(z_i)B(z_j), \label{commreltwoprime} \\
\widetilde{A}(z_j)\widetilde{A}(z_i)&=\widetilde{A}(z_i)\widetilde{A}(z_j). \label{commrelfourprime}
\end{align}
\end{subequations}

Using~\eqref{commrelthree} and~\eqref{eq:commrelsprime},
we can easily show that the following product of $A$- and $B$-operators
with some overall factor
\begin{align}
\displaystyle \prod_{i=1}^{\ell-1} z_i^{i-\ell}
\widetilde{A}(z_{\ell+m}) \cdots \widetilde{A}(z_{\ell+2}) \widetilde{A}(z_{\ell+1})
B(z_\ell) \cdots B(z_2) B(z_1),
\end{align}
is symmetric with respect to $z_1,z_2,\dotsc,z_{\ell+m}$.
Using the original $A$-operators, this means the following.

\begin{lemma} \label{symmetryproductsofAB}
The matrix-valued function
\[
\prod_{j={\ell+1}}^{\ell+m}
z_j \prod_{i=1}^{\ell-1} z_i^{i-\ell}
A(z_{\ell+m}) \cdots A(z_{\ell+2}) A(z_{\ell+1})
B(z_\ell) \cdots B(z_2) B(z_1)
\]
is symmetric with respect to $z_1,z_2,\dotsc,z_{\ell+m}$.
\end{lemma}

We show an analog of identities for Schur and Grothendieck polynomials by Feh\'er--N\'emethi--Rim\'anyi \cite{FNR12} and Guo--Sun \cite{GS19}.
We remark that a proof of the Guo--Sun identity based on quantum integrability was given in~\cite{Motegi20}, and we apply a similar argument to derive our identity.

\begin{thm}
\label{thm:fnrgstype}
Let $\lambda$ be a partition $\lambda=(\lambda_1,\lambda_2,\dotsc,\lambda_{\ell}>0,0^{n-\ell})$
such that $\lambda_1 = \lambda_2 = \cdots = \lambda_k$.
We define $x_{n+j}:=t_j$ for $j = 1, \dotsc, k-1$.
The following identity holds:
\begin{align*}
\dG_\lambda(\xx; \bt)
& = \sum_{S=\{j_1 < j_2 < \cdots < j_\ell \}
\in \binom{[n+k-1]}{\ell}}
\prod_{i \in \overline{S}}
\prod_{j \in S} \frac{x_j}{x_j-x_i} \\
& \hspace{20pt} \times
g_{\lambda}(x_{j_1},\dotsc,x_{j_{\ell-k+1}};x_{j_{\ell-k+2}},\dotsc,x_{j_\ell},t_k,\dotsc,t_{\ell-1}).
\end{align*}
\end{thm} 

\begin{proof}
Using the Yang--Baxter algebra, the partition functions $Z(\fN; \xx \sqcup \bt)$ can be written as
\begin{equation}
Z(\fN; \xx \sqcup \bt)
= (e_\lambda)_{\ell+\lambda_1}^*
A_{\lambda_\ell+1}(t_{\ell-1}) \cdots A_{\lambda_3+\ell-2}(t_2)
A_{\lambda_2+\ell-1}(t_1) A_{\lambda_1+\ell}(x_n) \cdots
A_{\lambda_1+\ell}(x_1) (e)_{\ell}.
\label{partitionfunctionabrepresentationone}
\end{equation}
We note from the nonintersecting lattice path interpretation that $Z(\fN_{\lambda}; \xx \sqcup \bt)$ is essentially the same as the partition function with the starting points of the paths are on the left side instead of the bottom.
We denote this lattice model as $\widetilde{\fN}_{\lambda}$.
Comparing the product of the Boltzmann weights in the southwest corner of the two partition functions
$Z(\fN_{\lambda}; \xx \sqcup \bt)$ and $Z(\widetilde{\fN}_{\lambda}; \xx \sqcup \bt)$, we obtain
\begin{align}
Z(\widetilde{\fN}_{\lambda}; \xx \sqcup \bt) =\displaystyle \prod_{i=1}^{\ell} x_i^{\ell+1-i} Z(\fN_{\lambda}; \xx \sqcup \bt)
= \displaystyle \prod_{i=1}^{\ell} x_i^{\ell+1-i} g_\lambda(\xx;\bt). \label{fnrgsone}
\end{align}

In terms of the Yang--Baxter algebra, what we consider is
\begin{align*}
Z(\widetilde{\fN}_{\lambda}; \xx \sqcup \bt)
& = (e_\lambda)_{\ell+\lambda_1}^*
A_{\lambda_\ell+1}(t_{\ell-1}) \cdots A_{\lambda_3+\ell-2}(t_2)
A_{\lambda_2+\ell-1}(t_1) \\
& \hspace{20pt} \times A(x_n) \cdots A(x_{\ell+1})B(x_\ell) \cdots B(x_1)
(e)_0,
\end{align*}
where we denote $A_{\lambda_1+\ell}(x)$ and $B_{\lambda_1+\ell}(x)$ as $A(x)$ and $B(x)$ for simplicity.

Now consider the case when $\lambda_1=\lambda_2=\cdots=\lambda_k$.
In this case, we note by using $e_1^* \otimes e_1^* L e_1 \otimes e_1=0$
that we can replace the $A$-operators $A_{\lambda_{i+1}+\ell-i}(t_i)$, $i=1,\dotsc,k-1$ in $\widetilde{Z}(\fN; \xx \sqcup \bt)$ by $A(t_i)$.
In terms of NILPs, this means that we can extend the ending points of the last $k-1$ paths as we did for Corollary~\ref{cor:refined_dual_symmetries}.
So we have the following relation
\begin{equation}
\begin{aligned}
Z(\widetilde{\fN}_{\lambda}; \xx \sqcup \bt)
& = (e_\lambda)_{\ell+\lambda_1}^*
A_{\lambda_\ell+1}(t_{\ell-1}) \cdots 
A_{\lambda_{k+1}+\ell-k}(t_k) \\
& \hspace{20pt} \times A(t_{k-1}) \cdots A(t_1)
A(x_n) \cdots A(x_{\ell+1})B(x_\ell) \cdots B(x_1)
(e)_0.
\label{partitionfunctionabrepresentationthree}
\end{aligned}
\end{equation}
Let us examine the operators
$
A(t_{k-1}) \cdots A(t_1)
A(x_n) \cdots A(x_{\ell+1})B(x_\ell) \cdots B(x_1)$.
We claim the following multiple commutation relation
\begin{align}
A(t_{k-1}) \cdots A(t_1) & 
A(x_n) \cdots A(x_{\ell+1})B(x_\ell) \cdots B(x_1) \nonumber \\
& = \sum_{S=\{j_1 < j_2 < \cdots < j_\ell \}
\in \binom{[n+k-1]}{\ell}}
\frac{\displaystyle \prod_{i \in \overline{S}} x_i}{\displaystyle \prod_{i=\ell+1}^{n+k-1} x_i}
\frac{\displaystyle \prod_{i=1}^{\ell-1} x_i^{\ell-i}}
{\displaystyle \prod_{i=1}^{\ell-1} x_{j_i}^{\ell-i}}
\prod_{i \in \overline{S}}
\prod_{j \in S} \frac{x_j}{x_j-x_i} \nonumber \\
& \hspace{20pt} \times B(x_{j_\ell}) \cdots B(x_{j_2}) B(x_{j_1})
\prod_{i \in \overline{S}} A(x_i),
\label{multiplecommutation}
\end{align}
where $x_{n+j}:=t_j$ for all $j = 1, \dotsc, k-1$.
An argument to derive multiple commutation relations for five-vertex type models can be found in~\cite{SU05} for example, which we apply to derive~\eqref{multiplecommutation} as follows.
We first use Lemma~\ref{symmetryproductsofAB} to rewrite the left hand side of~\eqref{multiplecommutation} as
\begin{align}
\frac{\displaystyle \prod_{i \in \overline{S}} x_i \prod_{i=1}^{\ell-1} x_i^{\ell-i}}
{\displaystyle \prod_{i=\ell+1}^{n+k-1} x_i \prod_{i=1}^{\ell-1} x_{j_i}^{\ell-i}}
\prod_{i \in \overline{S}} A(x_i)
B(x_{j_\ell}) \cdots B(x_{j_2}) B(x_{j_1}).
\label{beforecommutingAB}
\end{align}
Applying the commutation relation~\eqref{commrelone} to~\eqref{beforecommutingAB} repeatedly, we obtain the coefficient for
\[
B(x_{j_\ell}) \cdots B(x_{j_2}) B(x_{j_1}) \prod_{i \in \overline{S}} A(x_i)
\]
by noting that we always have to choose the first term in the right hand side of~\eqref{commrelone} when commuting $A$- and $B$-operators.

Inserting~\eqref{multiplecommutation} into~\eqref{partitionfunctionabrepresentationthree} and using the action of the $A$-operators
\[
\prod_{i \in \overline{S}} A(x_i) (e)_0 = (e)_0,
\]
we get
\begin{align}
Z(\widetilde{\fN}_{\lambda}; \xx \sqcup \bt)
=&
\sum_{S=\{j_1 < j_2 < \cdots < j_\ell \}
\in \binom{[n+k-1]}{\ell}}
\frac{\displaystyle \prod_{i \in \overline{S}} x_i}{\displaystyle \prod_{i=\ell+1}^{n+k-1} x_i}
\frac{\displaystyle \prod_{i=1}^{\ell-1} x_i^{\ell-i}}
{\displaystyle \prod_{i=1}^{\ell-1} x_{j_i}^{\ell-i}}
\prod_{i \in \overline{S}}
\prod_{j \in S} \frac{x_j}{x_j-x_i} \nonumber \\
\times&(e_\lambda)_{\ell+\lambda_1}^*
A_{\lambda_\ell+1}(t_{\ell-1}) \cdots 
A_{\lambda_{k+1}+\ell-k}(t_k) B(x_{j_\ell}) \cdots B(x_{j_2}) B(x_{j_1})
(e)_0. \label{beforeinsertingexpression}
\end{align}
We next use the NILP interpretation and change the starting points of the lattice paths to the bottom side and moving down the ending points of the last $k-1$ paths to get
\begin{align}
(e_\lambda)_{\ell+\lambda_1}^* &
A_{\lambda_\ell+1}(t_{\ell-1}) \cdots 
A_{\lambda_{k+1}+\ell-k}(t_k) B(x_{j_\ell}) \cdots B(x_{j_2}) B(x_{j_1})
(e)_0 \nonumber \\
& =
\prod_{i=1}^{\ell} x_{j_i}^{\ell+1-i}
(e_\lambda)_{\ell+\lambda_1}^*
A_{\lambda_\ell+1}(t_{\ell-1}) \cdots 
A_{\lambda_{k+1}+\ell-k}(t_k) A(x_{j_\ell}) \cdots A(x_{j_2}) A(x_{j_1})
(e)_{\ell}
\nonumber \\
& =
\prod_{i=1}^{\ell} x_{j_i}^{\ell+1-i}
(e_\lambda)_{\ell+\lambda_1}^*
A_{\lambda_\ell+1}(t_{\ell-1}) \cdots 
A_{\lambda_{k+1}+\ell-k}(t_k) 
A_{\lambda_k+\ell-k+1}(x_{j_\ell}) \cdots
A_{\lambda_2+\ell-1}(x_{j_{\ell-k+2}})
\nonumber \\
& \hspace{20pt} \times A(x_{j_{\ell-k+1}}) \cdots A(x_{j_1}) (e)_{\ell} \nonumber \\
& = \prod_{i=1}^{\ell} x_{j_i}^{\ell+1-i}
\dG_{\lambda}(x_{j_1},\dotsc,x_{j_{\ell-k+1}};x_{j_{\ell-k+2}},\dotsc,x_{j_\ell},t_k,\dotsc,t_{\ell-1}). \label{refinedexpression}
\end{align}
Inserting~\eqref{refinedexpression} into~\eqref{beforeinsertingexpression} and using
\begin{align}
\frac{\displaystyle \prod_{i \in \overline{S}} x_i \prod_{i=1}^{\ell} x_{j_i}^{\ell+1-i}}{\displaystyle \prod_{i=\ell+1}^{n+k-1} x_i \prod_{i=1}^{\ell-1} x_{j_i}^{\ell-i}}
= \prod_{i=1}^\ell x_i,
\end{align}
we get
\begin{align}
Z(\widetilde{\fN}_{\lambda}; \xx \sqcup \bt)
& =\displaystyle \prod_{i=1}^{\ell} x_i^{\ell+1-i}
\sum_{S=\{j_1 < j_2 < \cdots < j_\ell \}
\in \binom{[n+k-1]}{\ell}}
\prod_{i \in \overline{S}}
\prod_{j \in S} \frac{x_j}{x_j-x_i} \nonumber \\
& \hspace{20pt} \times \dG_{\lambda}(x_{j_1},\dotsc,x_{j_{\ell-k+1}};x_{j_{\ell-k+2}},\dotsc,x_{j_\ell},t_k,\dotsc,t_{\ell-1}). \label{fnrgstwo}
\end{align}
Comparing~\eqref{fnrgsone} and~\eqref{fnrgstwo}, we obtain our result.
\end{proof}

\begin{ex}
Let $n = 3$ and $\ell = 3$. Then from Theorem~\ref{thm:fnrgstype}, we have
\begin{align*}
\dG_{221}(x_1,x_2,x_3;t_1,t_2)&=
\frac{x_1 x_2 x_3}{(x_1-t_1)(x_2-t_1)(x_3-t_1)} g_{221}(x_1,x_2;x_3,t_2)
\\
& \hspace{20pt} +\frac{x_1 x_2 t_1}{(x_1-x_3)(x_2-x_3)(t_1-x_3)} g_{221}(x_1,x_2;t_1,t_2)
\\
& \hspace{20pt} +\frac{x_1 x_3 t_1}{(x_1-x_2)(x_3-x_2)(t_1-x_2)} g_{221}(x_1,x_3;t_1,t_2)
\\
& \hspace{20pt} +\frac{x_2 x_3 t_1}{(x_2-x_1)(x_3-x_1)(t_1-x_1)} g_{221}(x_2,x_3;t_1,t_2)
\end{align*}
for the refined dual Grothendieck polynomials
\begin{align*}
\dG_{221}(x_1,x_2,x_3;t_1,t_2) & = x_1^2 x_2^2 x_3+x_1^2 x_2 x_3^2+x_1 x_2^2 x_3^2
\\
& \hspace{20pt} +(t_1+t_2)(x_1^2 x_2^2+x_1^2 x_3^2+x_2^2 x_3^2+x_1^2 x_2 x_3+x_1 x_2^2 x_3
+x_1 x_2 x_3^2) \\
& \hspace{20pt}  + t_1(x_1^2 x_2 x_3+x_1 x_2^2 x_3+x_1 x_2 x_3^2)
\\
& \hspace{20pt}  + t_1(t_1+t_2)(x_1^2 x_2+x_1^2 x_3+x_2^2 x_3+x_1 x_2^2+x_1 x_3^2+x_2 x_3^2+2x_1 x_2 x_3) \\
& \hspace{20pt}  + t_1^2 t_2 (x_1^2+x_2^2+x_3^2+x_1 x_2+x_1 x_3+x_2 x_3),
\\
\dG_{221}(x_i,x_j;t_1,t_2)
& = (t_1+t_2)x_i^2 x_j^2+t_1(t_1+t_2)(x_i x_j^2+x_i^2 x_j) + t_1^2 t_2 (x_i^2+x_i x_j+x_j^2).
\end{align*}
\end{ex}

\subsection{Alternative Fermionic Vertex Model}

We can apply the same combinatorics used in the proof of Theorem~\ref{thm:refined_dual_model} to devise an equivalent statement using the $L$-matrix from Table~\ref{table:L_matrix_Grothendieck} at $\beta = 0$. Recall that there is a natural bijection between states of the model and Gelfand--Tsetlin patterns, which are equivalent to semistandard tableaux. Moreover, this model is also well-behaved for looking at skew Schur functions; see~\cite[Sec.~3]{MS14} for a generalization of this classical approach. The flagging condition on a semistandard skew tableau means that a particular particle does not move a certain distance beyond its initial position under a sequence of $A$-operators. Since our flagging is $(0, 1, 2, \dotsc, \ell)$, we can instead represent the elegant tableau portion as a sequence of $C$ operators, but we have to take into account the extra horizontal steps into the weight. We also are no longer able to control the final shape, but this is determined by the final horizontal steps for the corresponding $C$ operator. (Equivalently, the shape $\lambda$ is determined by the positions of the final vertical step.) This is exactly the same as the jagged part in the model above. Hence, we obtain the following partition function of the model given by $\langle 0^{m+n} | C^{\ell} A^{n-\ell} B^{\ell} | 0^{m+n} \rangle$ (note that $n \geq \ell$ necessarily).

\begin{prop}
The partition function of this model is
\[
Z(n, m+n) = \sum_{\lambda_1 \leq m} \prod_{i=1}^{\ell} t_i^{k-\lambda_i} \dG_{\lambda}(\xx; \bt).
\]
\end{prop}

\begin{ex}
\label{ex:5V_model_to_dualG}
Consider $\lambda = 6433$ with $\ell = 4$ and $m = 8$. Then the $01$-sequence of $\lambda$ is $000110100100$.
Let $n = 5$.
The following is a state in $Z(5, 13)$ and the corresponding element from $\dG_{\lambda}(\xx; \bt)$:
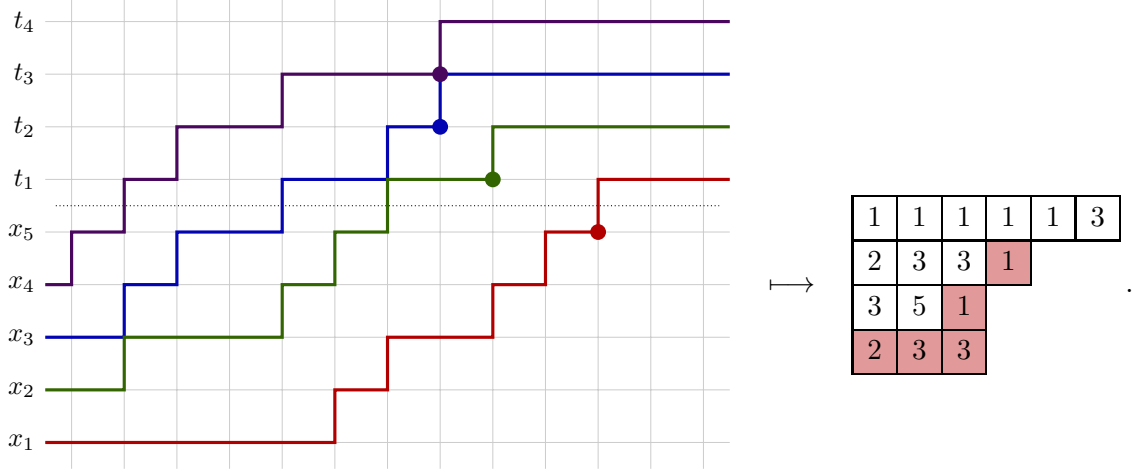
\[
\begin{tikzpicture}[>=latex,baseline=2.7cm,scale=0.7]
\draw[very thin, black!20] (0.5, 0.5) grid (13.5, 9.5);
\foreach \y in {1,2,3,4,5}
  \draw node[anchor=east] at (0.5,\y)  {\small $x_{\y}$};
\foreach \y in {1,2,3,4}
  \draw node[anchor=east] at (0.5,\y+5)  {\small $t_{\y}$};
\draw[densely dotted] (0.7,5.5) -- (13.3,5.5);
\draw[-, very thick, UQpurple] (0.5,4) -- (1,4) -- (1,5) -- (2,5) -- (2,6) -- (3,6) -- (3,7) -- (5,7) -- (5,8) -- (8,8) -- (8,9) -- (13.5,9);
\draw[-, very thick, dbluecolor] (0.5,3) -- (2,3) -- (2,4) -- (3,4) -- (3,5) -- (5,5) -- (5,6) -- (7,6) -- (7,7) -- (8,7) -- (8,8) -- (13.5,8);
\draw[-, very thick, dgreencolor] (0.5,2) -- (2,2) -- (2,3) -- (5,3) -- (5,4) -- (6,4) -- (6,5) -- (7,5) -- (7,6) -- (9,6) -- (9,7) -- (13.5,7);
\draw[-, very thick, darkred] (0.5,1) -- (6,1) -- (6,2) -- (7,2) -- (7,3) -- (9,3) -- (9,4) -- (10,4) -- (10,5) -- (11,5) -- (11,6) -- (13.5,6);
\fill[UQpurple] (8,8) circle (0.15);
\fill[dbluecolor] (8,7) circle (0.15);
\fill[dgreencolor] (9,6) circle (0.15);
\fill[darkred] (11,5) circle (0.15);
\end{tikzpicture}
\quad \longmapsto \quad
\ytableaushort{111113,2331,351,233}  *[*(white)]{6,3,2} *[*(darkred!40)]{4,4,3,3}\,.
\]
\end{ex}

Next, we can write the partition function $Z(n, m+n)$ as the classical Cauchy identity (see, \textit{e.g.},~\cite[Cor.~3.6, Cor.~5.4]{MS13} at $\beta = 0$) and then perform the substitution
\begin{equation}
\label{eq:schur_complement}
s_{\lambda}(\bt) = \prod_{i=1}^{\ell} t_i^m s_{\lambda^{\dagger}}(\bt^{-1})
\end{equation}
since we are taking $C$ operators instead of $B$ operators reflected vertically.
Note that~\eqref{eq:schur_complement} can be checked using
$\lambda_j^{\dagger}=m-\lambda_{\ell+1-j}$, $j=1,\dotsc,\ell$
and the Weyl character formula
for the Schur functions
\begin{align}
\displaystyle s_\lambda(\bt)
=\frac{\det \big[t_i^{\lambda_j+\ell-j} \big]_{i,j=1}^\ell}
{\det \big[ t_i^{\ell-j} \big]_{i,j=1}^\ell}.
\end{align}
Hence, we obtain a refined version of a Cauchy--Littlewood identity with the partitions $\lambda \subseteq m^{\ell}$ inside an $\ell \times m$ rectangle.

\begin{cor}[Cauchy--Littlewood identity]
\label{cor:dual_Grothendiecks_in_box}
We have
\begin{align*}
& \sum_{\lambda \subseteq m^{\ell}} \prod_{i=1}^{\ell} t_i^{m-\lambda_i} \dG_{\lambda}(\xx; \bt)
\\ & \hspace{40pt} =
\prod_{i=1}^{\ell} t_i^m \prod_{1 \leq i < j \leq n} \dfrac{1}{(x_i - x_j)(t_i^{-1} - t_j^{-1})} \det \left[ \dfrac{(x_i t_j^{-1})^{m+n} - 1}{x_i t_j^{-1} - 1} \right]_{i,j=1}^n \Biggr\rvert_{t_{\ell+1} = \cdots = t_n = \infty}.
\end{align*}
\end{cor}

By multiplying both sides by $\prod_{i=1}^{\ell} t_i^{-m}$ and then taking the limit of $m \to \infty$, we recover the Cauchy--Littlewood identity~\cite[Thm~5.2(iv)]{Yel19III} (after substituting $t_i \mapsto t_i^{-1}$ and using~\cite[Rem.~5.6]{Yel19III}; see also~\cite[Lemma~3.6]{Yel19II} for the $\bt = 1$ version).

\begin{cor}[{\cite[Thm~5.2(iv)]{Yel19III}}]
\label{cor:dual_Grothendiecks_bounded}
We have
\[
\sum_{\ell(\lambda) \leq \ell} \prod_{i=1}^{\ell} t_i^{-\lambda_i} \dG_{\lambda}(\xx; \bt) = \prod_{i=1}^n \prod_{j=1}^{\ell}  \dfrac{1}{1 - t_j^{-1}x_i} = \prod_{i=1}^n \prod_{j=1}^{\ell} \dfrac{t_j}{t_j - x_i}.
\]
\end{cor}

We also remark that we can move between states of this model and those using the $L$-matrix from Table~\ref{table:L_matrix_nilp} sliding the $j$-th path from the left $j$ steps to up, staggering the vertical steps, slightly stretching the horizontal steps.

\begin{ex}
If we apply the above equivalence to the state from Example~\ref{ex:NILP_tableaux} (see also Example~\ref{ex:NILP_5V_model}), we obtain
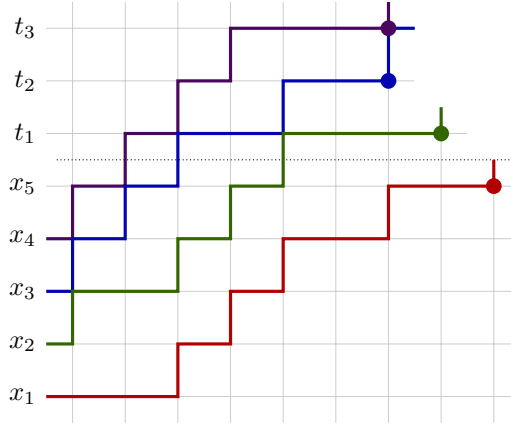
\[
\begin{tikzpicture}[>=latex,baseline=2.7cm,scale=0.7]
\draw[very thin, black!20] (0.5, 0.5) grid (9.5, 8.5);
\fill[white] (7.5,6.5) rectangle (9.5,8.5);
\fill[white] (8.5,5.5) rectangle (9.5,7.5);
\foreach \y in {1,2,3,4,5} {
  \draw node[anchor=east] at (0.5,\y)  {\small $x_{\y}$};
}
\foreach \y in {1,2,3} {
  \draw node[anchor=east] at (0.5,\y+5)  {\small $t_{\y}$};
}
\draw[densely dotted] (0.7,5.5) -- (9.3,5.5);
\draw[-, very thick, UQpurple] (0.5,4) -- (1,4) -- (1,5) -- (2,5) -- (2,6) -- (3,6) -- (3,7) -- (4,7) -- (4,8) -- (7,8) -- (7,8.5);
\draw[-, very thick, dbluecolor] (0.5,3) -- (1,3) -- (1,4) -- (2,4) -- (2,5) -- (3,5) -- (3,6) -- (5,6) -- (5,7) -- (7,7) -- (7,8) -- (7.5,8);
\draw[-, very thick, dgreencolor] (0.5,2) -- (1,2) -- (1,3) -- (3,3) -- (3,4) -- (4,4) -- (4,5) -- (5,5) -- (5,6) -- (8,6) -- (8,6.5);
\draw[-, very thick, darkred] (0.5,1) -- (3,1) -- (3,2) -- (4,2) -- (4,3) -- (5,3) -- (5,4) -- (7,4) -- (7,5) -- (9,5) -- (9,5.5);
\fill[UQpurple] (7,8) circle (0.15);
\fill[dbluecolor] (7,7) circle (0.15);
\fill[dgreencolor] (8,6) circle (0.15);
\fill[darkred] (9,5) circle (0.15);
\end{tikzpicture}
\]
Note that instead of using the operator $A B^4$ for the semistandard tableau portion, we could have equivalently used $B^5$. In this case, the paths for the additional particle would be trivial (\textit{i.e.}, they do not contribute to the weight).
\end{ex}

We note that Corollary~\ref{cor:dual_Grothendiecks_in_box} is the same as Corollary~\ref{cor:littlewood_identity}
from Equation~\eqref{eq:schur_complement} but expressed in two different models under the identification given above.

One other advantage to using this model is that we can use the generic $\beta$ Boltzmann weights for the $L$-matrix from Table~\ref{table:L_matrix_Grothendieck}. This gives us a new symmetric function that expands positively in the Grothendieck polynomials where the coefficients are the number of set-valued elegant tableaux. Furthermore, this expansion is finite even as $n \to \infty$ since the Schur expansion of $\dG_{\lambda}(\xx; \bt)$ is finite and there are only finitely many set-valued elegant tableaux of a fixed shape $\lambda/\mu$, which we denote by $\set(\lambda/\mu)$.

\begin{thm}
We have
\[
Z(\lambda; \xx; \bt; \beta) = \sum_{\mu \subseteq \lambda} \sum_{T \in \set(\lambda/\mu)} \bt^T \G_{\mu}(\xx; \beta).
\]
\end{thm}

As a corollary, we see that the partition function is Schur positive (up to a sign by degree) by~\cite[Thm.~4.6]{HS20}.

\section{Probability theory}
\label{sec:probability}

In this section, we will discuss how our refined dual Grothendieck polynomials relates with probability theory through the last passage percolation (LPP) stochastic model, the totally asymmetric simple exclusion process (TASEP), and the Schur measure.
We first show the relation between the LPP joint probability and the refined dual Grothendieck polynomials by using the RSK correspondence, refining the result by Yeliussizov~\cite{Yel20}.
We then take this probabilistic interpretation to be our definition of a refined dual Grothendieck polynomial, where the skew version is given by the LPP transition probability. Using this, we refine the algebraic approach given in~\cite{Johansson10} to give a (dual) Jacobi--Trudi formula for refined dual Grothendieck polynomials. Finally, we describe the Cauchy identity for Schur functions with refined dual Grothendieck polynomials to give a new proof of a result of Baik and Rains~\cite{BR01} and relating to the Schur measure~\cite{Okounkov00,Okounkov01}.
For this section, we use the indeterminates $\bt = (t_1, t_2, \dotsc, t_\ell)$.

\subsection{Last Passage Percolation}

We begin by defining the last-passage percolation (LPP) stochastic model and then recall the bijection from~\cite[Thm.~1]{Yel20} (see also~\cite[Thm~8.1]{Yel19II} for the specialization $\xx = q$) from reverse\footnote{The bijection given in~\cite{Yel20} was originally described for usual plane partitions, but we give it by swapping columns $i \leftrightarrow n + 1 -  i$ (\textit{i.e.}, reflecting the matrix across a vertical line), which corresponds to going between plane partitions and reverse plane partitions by swapping entries $i \leftrightarrow n + 1 -  i$.} plane partitions to the random matrices in the LPP model.
We note that the LPP model is a corner growth model for partitions.

\begin{remark}
In this section, we index our matrices following the natural indexing on $\ZZ^2$, so the bottom-left corner is the $(1,1)$ entry.
\end{remark}

Consider a random matrix $(w_{ij})_{i,j \geq 1}$ with independent entries $w_{ij}$ with the geometric distribution
\[
P(w_{ij} = k) = (1 - t_i x_j) (t_i x_j)^k,
\]
where $t_i, x_j \in (0,1)$ and $k \in \ZZ_{\geq 0}$. The \defn{last passage time} for this random matrix is
\[
G(m,n) = \max_{\Pi} \sum_{(i,j) \in \Pi} w_{ij},
\]
where the maximum is taken over all paths from $(1,1)$ to $(m,n)$ with unit steps to the right or up, \textit{i.e.}, for $(i_a,j_a)$, we either have $(i_{a+1},j_{a+1}) = (i_a, j_a+1), (i_a+1, j_a)$.

Let us define $\bG(i)$ as $\bG(i) := (G(\ell,i), \dotsc, G(1,i))$.
We denote the transition probability $P(\bG(i)=\lambda | \bG(j)=\mu)$, for $\lambda = (\lambda_1, \dotsc, \lambda_{\ell})$ and $\mu = (\mu_1, \dotsc, \mu_{\ell})$, of the LPP process. This was introduced by Johansson~\cite{Johansson10} with the specializations $x_i = q$ and $t_i = 1$.
We extend this in the natural way to allow $\bG(0) = \mu$ by instead using $G_{\mu}(k, i) = G(k, i) + \mu_{\ell+1-k}$.
Thus we have
\[
P(\bG(n) = \lambda) = P(\bG(n)=\lambda|\bG(0)=\emptyset).
\]

\begin{thm}[{\cite[Thm.~8.1]{Yel19II}}]
\label{thm:last_passage}
We have
\[
P(\bG(n)=\lambda) \biggr\rvert_{\bt=1} = \prod_{j=1}^n (1 - x_j)^{\ell} \dG_{\lambda}(\xx; 1).
\]
\end{thm}

In~\cite[Thm.~1]{Yel20}, Yeliussizov gave a direct combinatorial proof of Theorem~\ref{thm:last_passage} by constructing a bijection $\Phi$ from~\cite{Yel19II} between reverse plane partitions and the matrices that records the number of times a box is not equal to the box below it.
The remainder of the proof is given by conditioning the probabilities on each entry in the resulting matrix, which results in the leading factor.

\begin{ex}
Consider $\lambda = 210$, $\ell = 3$, and $n = 2$. Then under the bijection $\Phi$ given above, we have
\[
\begin{array}{c*{4}{@{\hspace{35pt}}c}}
\\\toprule
\ytableaushort{11,1} & \ytableaushort{11,2} & \ytableaushort{12,1} & \ytableaushort{12,2} & \ytableaushort{22,2}
\\\midrule
\begin{bmatrix} 1 & 0 \\ 1 & 0 \\ 0 & 0 \end{bmatrix}
&
\begin{bmatrix} 2 & 0 \\ 0 & 1 \\ 0 & 0  \end{bmatrix}
&
\begin{bmatrix}  0 & 1 \\ 1 & 0 \\ 0 & 0  \end{bmatrix}
&
\begin{bmatrix} 1 & 1 \\ 0 & 1 \\ 0 & 0  \end{bmatrix}
&
\begin{bmatrix} 0 & 1 \\ 0 & 1 \\ 0 & 0  \end{bmatrix}
\\\midrule
t_1 x_1^2 & x_1^2 x_2 & t_1 x_1 x_2 & x_1 x_2^2 & t_1 x_2^2
\\\bottomrule
\end{array}
\]
where we have written the weights of the RPP at the bottom. Note that in $P(\bG(n) = \lambda)$, we have a factor of $\prod_{i=1}^3 \prod_{j=1}^2 (1 - t_i x_j)$ from the fact we take the conditional probability on a $3 \times 2$ matrix, with one factor for each entry.
\end{ex}

\begin{ex}
\label{ex:RPP_Matrix_path}
The maximal path that defines $G(\ell+1-i, n) = \lambda_i$ in $\Phi(T)$ does not always come from the $i$-th row of $T$. Indeed, consider
\[
T = \ytableaushort{114,134,33} \longmapsto (b_{ij})_{i,j} = \begin{bmatrix}
1 & 0 & 0 & 0
\\ 1 & 0 & \mathbf{0} & \mathbf{1}
\\ \mathbf{0} & \mathbf{0} & \mathbf{2} & 0
\end{bmatrix},
\]
and the path defining $G(2,4) = 3$ (indicated by the bolded letters) does not pass through the entry $b_{2,1}$ but instead $b_{3,3}$.
We also note that reflection or rotating matrix results in a different shape~$\lambda$.
\end{ex}

Note that assuming Theorem~\ref{thm:last_passage}, then~\cite[Cor.~8.4]{Yel19II} is equivalent to Corollary~\ref{cor:littlewood_identity} with $\bt = 1$ by considering conditioning on the shape $\lambda$ with bounded $\lambda_1$.

Next, we consider the Robinson--Schensted--Knuth (RSK) bijection (see, \textit{e.g.},~\cite{Fulton,ECII}) between $\ZZ_{\geq 0}$-matrices and pairs of semistandard tableaux. It is known that such a matrix $M$ is equivalent to a two-line array (also known as generalized permutations) with $M_{ij}$ denoting the number of bi-letters $\begin{bmatrix} i \\ j \end{bmatrix}$, which are sorted by lex order.
Let $M \mapsto (P, Q)$ under RSK, where $P$ (resp.~$Q$) is called the insertion (resp.\ recording) tableaux. Let $\lambda$ be the shape of $P$ (and $Q$).
A well-known important property of the RSK bijection is the length of a longest increasing subsequence corresponds to $\lambda_1$ (see, \textit{e.g.},~\cite{Johansson00,ECII}).

\begin{ex}
We have
\[
\ytableaushort{113,134,33}
\overset{\Phi}{\longmapsto}
\begin{bmatrix}
1 & 0 & 1 & 0 \\
1 & 0 & 0 & 1 \\
0 & 0 & 2 & 0 \\
\end{bmatrix}
\longleftrightarrow
\begin{bmatrix}
1 & 1 & 2 & 2 & 3 & 3 \\
3 & 3 & 1 & 4 & 1 & 3
\end{bmatrix}
\overset{\RSK}{\longmapsto}
\left(
\ytableaushort{113,334}\,,\,
\ytableaushort{112,233}
\right).
\]
\end{ex}

Recall that the length of the longest increasing subsequence of the bottom row of a two-line array when we restrict to $\ell \times n$ matrices corresponds to $\lambda_1 = G(\ell,n)$ under RSK. Hence, we can compute $G(\ell-k, n)$ by setting the top $k$ rows (using the convention the lower-left corner is the $(1,1)$ entry) of the matrix $M$ to $0$, which we denote this matrix by $M^{(k)}$. Let $M^{(k)} \mapsto (P^{(k)}, Q^{(k)})$ under RSK, which have shape $\lambda^{(k)}$. Hence, we have $G(\ell-k, n) = \lambda^{(k)}_1$. If we look at the recording tableau, we have that $Q^{(k)}$ is equal to removing all entries at least $n - k$ from $Q = Q^{(0)}$. Thus, for any matrix $M$ such that $\bigl(G(\ell-k,n)\bigr)_{k=0}^{\ell-1} = \lambda$, we require that the left side of the GT pattern representation of $Q$ must be $\lambda$. Therefore, we obtain
\begin{equation}
\label{eq:prob_schur_expansion}
P(\bG(n)=\lambda) = \prod_{i=1}^{\ell} \prod_{j=1}^n (1 - t_i x_j) \bt^{\lambda} \sum_{\mu \subseteq \lambda} \sum_{\substack{Q \in \ssyt^{\ell}(\mu) \\ \lf(Q) = \lambda}} \bt^{-Q} s_{\mu}(\xx),
\end{equation}
where $\lf(Q)$ is the leftmost entries of the GT pattern corresponding to $Q$.

\begin{lemma}
\label{lemma:GT_ET_bijection}
There exists a weight-preserving bijection
\[
\psi \colon \{Q \in \ssyt^{\ell}(\mu) \mid \lf(Q) = \lambda \} \to \et(\lambda/\mu).
\]
\end{lemma}

\begin{proof}
To see that the corresponding GT patterns correspond to the elegant tableaux of shape $\lambda / \mu$, we simply reflect the GT pattern vertically and consider it inside of a larger skew GT pattern, which we then biject with semistandard tableaux. Pictorially, we have:
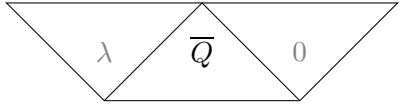
\[
\begin{tikzpicture}[scale=0.65]
\draw[-] (0,0) -- (2,-2) -- (6,-2) -- (8,0) -- cycle;
\draw[-] (2,-2) -- (4,0) -- (6, -2);
\draw[color=gray] (2,-1) node {$\lambda$};
\draw (4,-1) node {$\overline{Q}$};
\draw[color=gray] (6,-1) node {$0$};
\end{tikzpicture}
\]
with the bottom row of $\overline{Q}$ being $\mu$.
\end{proof}

\begin{ex}
For $\ell = 4$, under the correspondence mentioned above, we have
\[
\begin{array}{ccccccc}
\lambda_1 & & \mu_2 && \mu_3 && \mu_4 \\
& \lambda_2 && x && y \\
&& \lambda_3 && z \\
&&& \lambda_4
\end{array}
\longrightarrow
\begin{array}{ccccccccccccc}
{\color{gray}\lambda_1} && {\color{gray}\lambda_2} && {\color{gray}\lambda_3} && \lambda_4 && {\color{gray}0} && {\color{gray}0} && {\color{gray}0} \\
& {\color{gray}\lambda_1} && {\color{gray}\lambda_2} && \lambda_3 && z && {\color{gray}0} && {\color{gray}0}\\ 
&& {\color{gray}\lambda_1} && \lambda_2 && x && y && {\color{gray}0} \\
&&& \lambda_1 && \mu_2 && \mu_3 && \mu_4 &
\end{array}
\]
So for the elegant tableaux, the number of $1$'s on the second row is $\lambda_2 - \mu_2$, on the third row is $x - \mu_3$, and on the fourth row is $y - \mu_4$; the number of $2$'s on third row is $\lambda_3 - x$ and the fourth row is $z - y$; and the number of $3$'s on the third row is $\lambda_4 - z$.
\end{ex}

From Equation~\eqref{eq:prob_schur_expansion}, Lemma~\ref{lemma:GT_ET_bijection}, and Theorem~\ref{thm:dualG_schur_decomposition}, we have the refined version of Theorem~\ref{thm:last_passage}.

\begin{thm}
\label{thm:refined_last_passage}
We have
\[
P(\bG(n)=\lambda) = \prod_{i=1}^{\ell} \prod_{j=1}^n (1 - t_i x_j) \bt^{\lambda} \dG_{\lambda}(\xx; \bt^{-1}).
\]
\end{thm}

We could have also shown Theorem~\ref{thm:refined_last_passage} by showing the bijection $\Phi$ preserves the weight under the refinement.
Thus by also using Equation~\eqref{eq:prob_schur_expansion} and Lemma~\ref{lemma:GT_ET_bijection}, we obtain a new bijective proof of Theorem~\ref{thm:dualG_schur_decomposition} as we impose no conditions on the insertion tableau under RSK.
Note that $\bt^{\lambda} \bt^{-Q}$ gives the number of entries at height $i$ that have an $i$ directly below it in the RPP by the definition of $\Phi$.
Furthermore, the bijection $\Phi$ can be thought of as describing the positions where the RSK insertion of the inflation map acts.

Theorem~\ref{thm:refined_last_passage} also gives a new probabilistic proof of Corollary~\ref{cor:dual_Grothendiecks_bounded} (with substituting $t_i \mapsto t_i^{-1}$) by taking a sum over all $\lambda$ with $\ell(\lambda) \leq \ell$, which means we have considered all possible states. So we have
\begin{align*}
1 & = \sum_{\ell(\lambda) \leq \ell} P(\bG(n)=\lambda)
 = \prod_{i=1}^{\ell} \prod_{j=1}^n (1 - t_i x_j) \sum_{\ell(\lambda) \leq \ell} \bt^{\lambda} \dG_{\lambda}(\xx; \bt^{-1}),
\end{align*}
and then divide by the leading factor to obtain Corollary~\ref{cor:dual_Grothendiecks_bounded}.

We can take the idea of fixing the left part of a GT pattern a bit further obtain a similar-but-distinct lattice model for dual Grothendieck polynomials to those in Section~\ref{sec:refined_models}. Indeed, recall that GT patterns are in bijection with states of a trivially colored lattice model and NILPs. The condition that the left side of the GT pattern is $\lambda$ means we have fixed the bottom path, which we can then remove and transform it into a jagged model with specific boundary conditions.
If we then reflect this model over the horizontal axis, we can stack it on top of the corresponding lattice model for the insertion tableaux.

\begin{ex}
Consider the elegant tableaux from Example~\ref{ex:NILP_tableaux}, and its corresponding GT pattern and semistandard tableau are
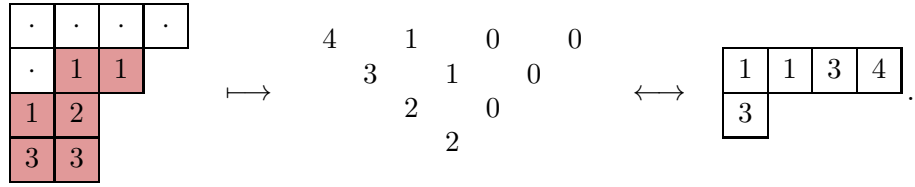
\[
\ytableaushort{{\cdot}{\cdot}{\cdot}{\cdot},{\cdot}11,12,33}  *[*(white)]{4,1} *[*(darkred!40)]{4,3,2,2}
\quad \longmapsto \quad
\begin{array}{ccccccc}
4 && 1 && 0 && 0 \\
& 3 && 1 && 0 \\
&& 2 && 0 \\
&&& 2
\end{array}
\quad \longleftrightarrow \quad
\ytableaushort{1134,3}\,.
\]
Next, we take the corresponding NILP model and 5-vertex models:
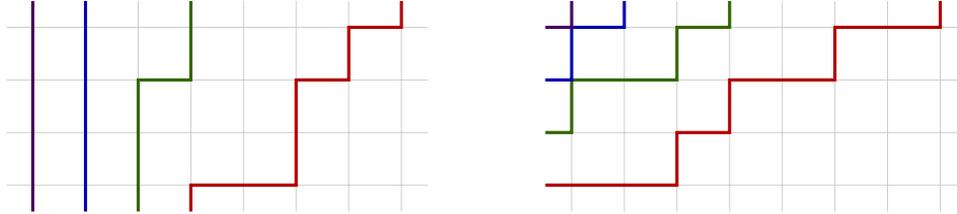
\[
\begin{tikzpicture}[>=latex,baseline=1.7cm,scale=0.7]
\draw[very thin, black!20] (-0.5, 1-0.5) grid (7.5, 4.5);
\draw[-, very thick, darkred] (3,0.5) -- (3,1) -- (5,1) -- (5,3) -- (6,3) -- (6,4) -- (7,4) -- (7,4.5);
\draw[-, very thick, dgreencolor] (2,0.5) -- (2,3) -- (3,3) -- (3,4.5);
\draw[-, very thick, dbluecolor] (1,0.5) -- (1,4.5);
\draw[-, very thick, UQpurple] (0,0.5) -- (0,4.5);
\end{tikzpicture}
\qquad\qquad
\begin{tikzpicture}[>=latex,baseline=1.7cm,scale=0.7]
\draw[very thin, black!20] (-0.5, 1-0.5) grid (7.5, 4.5);
\draw[-, very thick, darkred] (-0.5,1) -- (2,1) -- (2,2) -- (3,2) -- (3,3) -- (5,3) -- (5,4) -- (7,4) -- (7,4.5);
\draw[-, very thick, dgreencolor] (-0.5,2) -- (0,2) -- (0,3) -- (2,3) -- (2,4) -- (3,4) -- (3,4.5);
\draw[-, very thick, dbluecolor] (-0.5,3) -- (0,3) -- (0,4) -- (1,4) -- (1,4.5);
\draw[-, very thick, UQpurple] (-0.5,4) -- (0,4) -- (0,4.5);
\end{tikzpicture}\,.
\]
Then the corresponding jagged portions of the models are given by
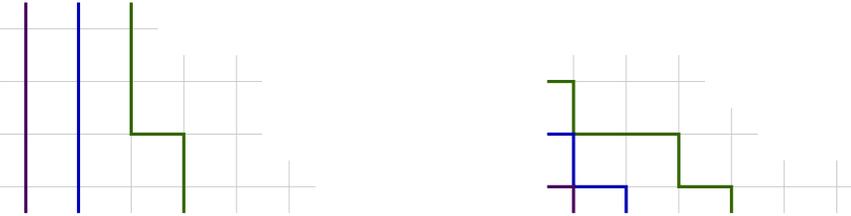
\[
\begin{tikzpicture}[>=latex,baseline=-1.7cm,xscale=0.7,yscale=-0.7]
\draw[very thin, black!20] (-0.5, 1-0.5) grid (5.5, 4.5);
\draw[-, very thick, dgreencolor] (2,0.5) -- (2,3) -- (3,3) -- (3,4.5);
\draw[-, very thick, dbluecolor] (1,0.5) -- (1,4.5);
\draw[-, very thick, UQpurple] (0,0.5) -- (0,4.5);
\fill[white] (2.5,0.3) rectangle (7.6,1.5);
\fill[white] (4.5,0.3) rectangle (7.6,3.5);
\end{tikzpicture}
\qquad\qquad
\begin{tikzpicture}[>=latex,baseline=-1.7cm,xscale=0.7,yscale=-0.7]
\draw[very thin, black!20] (-0.5, 1-0.5) grid (5.5, 4.5);
\draw[-, very thick, dgreencolor] (-0.5,2) -- (0,2) -- (0,3) -- (2,3) -- (2,4) -- (3,4) -- (3,4.5);
\draw[-, very thick, dbluecolor] (-0.5,3) -- (0,3) -- (0,4) -- (1,4) -- (1,4.5);
\draw[-, very thick, UQpurple] (-0.5,4) -- (0,4) -- (0,4.5);
\fill[white] (-0.6,0.3) rectangle (7.6,1.5);
\fill[white] (2.5,0.3) rectangle (7.6,2.5);
\fill[white] (3.5,0.3) rectangle (7.6,3.5);
\end{tikzpicture}\,.
\]
Note that we are still allowed to touch the rightmost line in the trivially colored model, so we do not remove those vertices. Furthermore, we could have removed the top row from the jagged NILP model since those paths are all fixed.
\end{ex}

Lastly, we describe the insertion tableau under the composition of RSK and $\Phi$.

\begin{prop}
The insertion tableau under RSK of the result of $\Phi$ is given by just removing any $i$ such that there is an $i$ directly below it and using jeu-de-taquin to rectify the remaining boxes to a straight shape $\mu$.
\end{prop}

\begin{proof}
This is a consequence of the definition of $\Phi$ and that jeu-de-taquin is equivalent to performing RSK insertion (see, \textit{e.g.},~\cite{ECII}).
This could also be seen using the crystal structure on reverse plane partitions given in~\cite{Galashin17}.
\end{proof}

\subsection{Transition probability and skew refined dual Grothendieck polynomials}
\label{sec:transition_JT_formula}

For the remainder of this section, we will take Theorem~\ref{thm:refined_last_passage} as the \emph{definition} of the refined dual Grothendieck polynomials.
This leads to a natural probabilistic definition of the skew refined dual Grothendieck polynomials using transition probabilities for the LPP.
In this subsection, we first show that this is equivalent to the combinatorial definition from~\cite{GGL16}.
Then we derive a Jacobi--Trudi determinant form for the skew refined dual Grothendieck polynomials by refining the difference operator technique by Johannson \cite{Johansson10}.

We first introduce notations.
We define the discrete Heaviside step function $H(\nu)$ for $\nu \in \ZZ$ as
\[
H(\nu) = \begin{cases}
0 & \text{if } \nu <0, \\
1 & \text{if } \nu \ge 0.
\end{cases}
\]
For two functions $f, g \colon \ZZ \to \CC$, we define the convolution product $f*g$ as
\[
f*g(\nu) = \sum_{\xi=-\infty}^\infty f(\nu-\xi)g(\xi).
\]

Recall that we have shown the following relation between the joint probability and the refined dual Grothendieck polynomials in Theorem~\ref{thm:refined_last_passage}:
\[
P(\bG(n)=\lambda)
=\prod_{i=1}^\ell \prod_{j=1}^n (1-t_i x_j) \bt^{\lambda} g_\lambda(\xx;{\bf t^{-1}}).
\]
As previously stated, we now take this to be our \emph{definition} of the refined dual Grothendieck polynomials; that is
\[
g_\lambda(\xx;{\bf t^{-1}}) := \bt^{-\lambda} \prod_{i=1}^\ell \prod_{j=1}^n (1-t_i x_j)^{-1} P(\bG(n)=\lambda).
\]
From this, a natural definition of skew refined dual Grothendieck polynomials for a single variable $x_n$ is using the transition (or conditional) probability by
\begin{equation}
g_{\lambda/\mu}(x_n;\bt^{-1})
:=\prod_{i=1}^\ell (1-t_i x_n)^{-1} \bt^{\mu-\lambda}
P(\bG(n)=\lambda|\bG(n-1)=\mu).
\label{eq:refinedskewdualbylpp}
\end{equation}
We then introduce the skew refined dual Grothendieck polynomials as
\begin{equation}
g_{\lambda/\mu}(\xx;\bt^{-1})
:=\prod_{i=1}^\ell \prod_{j=1}^n (1-t_i x_j)^{-1}
\bt^{\mu-\lambda}
P({\bf G}(n)=\lambda|{\bf G}(0)=\mu).
\label{eq:refineddefbylpp}
\end{equation}

\begin{prop}
\label{prop:refined_branching}
The following branching formulas hold:
\begin{subequations}
\begin{align}
g_\lambda(\xx;\bt^{-1}) & = \sum_{\mu} g_{\lambda/\mu}(x_n;\bt^{-1}) g_\mu(x_1,\dotsc,x_{n-1};\bt^{-1}),
\label{eq:branching_single_var}
\\
g_{\lambda/\nu}(\xx \sqcup \yy;\bt) & =\sum_\mu g_{\lambda/\mu}(\xx;\bt)g_{\mu/\nu}(\yy;\bt),
\\
g_{\lambda/\mu}(\xx;\bt) & = \sum_{\emptyset = \lambda^{(0)} \subseteq \lambda^{(1)} \subseteq \cdots \subseteq \lambda^{(n)} = \lambda}
\prod_{i=1}^n g_{\lambda^{(i)}/\lambda^{(i-1)}}(x_i;\bt).
\label{refineddefmultskew}
\end{align}
\end{subequations}
\end{prop}

\begin{proof}
Equation~\eqref{eq:branching_single_var} follows from applying~\eqref{eq:refinedskewdualbylpp} to the basic property
\[
P(\bG(n)=\lambda) = \sum_{\mu} P(\bG(n)=\lambda|\bG(n-1)=\mu)P(\bG(n-1)=\mu)
\]
of conditional probabilities.
The remaining equations follow from~\eqref{eq:branching_single_var} and properties of transition probabilities.
\end{proof}

Next, one can compute the transition probability explicitly as
\[
P(\bG(n)=\lambda|\bG(n-1)=\mu)
=\prod_{j=1}^\ell (1-t_j x_n)(t_j x_n)^{\lambda_j-\max(\mu_j,\lambda_{j+1})}
H(\lambda_j-\max(\mu_j,\lambda_{j+1})),
\]
from which we find the factorized form
\begin{equation}
g_{\lambda/\mu}(x_n;\bt)
=\prod_{j=1}^{\ell-1}t_j^{\max(\mu_j,\lambda_{j+1})-\mu_j}
\prod_{j=1}^{\ell} x_n^{\lambda_j-\max(\mu_j,\lambda_{j+1})}
H(\lambda_j-\max(\mu_j,\lambda_{j+1})). \label{eq:factorized_form_one_var}
\end{equation}

We can see that~\eqref{eq:factorized_form_one_var} is precisely the generating function for the number of reverse plane partitions of shape $\lambda/\mu$ with a single entry $n$ by reading the reverse plane partition row-by-row.
As a consequence, it is a simple recursive combinatorial argument to show this is equivalent to the definition in~\cite{GGL16} by removing the boxes in a reverse plane partition containing an $n$.

\begin{cor}
\label{cor:really_equivalent}
We have
\[
\dG_{\lambda/\mu}(\xx; \bt) = \sum_{T \in \rpp^n(\lambda/\mu)} \bt^{b(T)} \xx^{a(T)}.
\]
\end{cor}

Next we convert the factorized form~\eqref{eq:factorized_form_one_var} to a determinant form.
To do this, we introduce the weighted difference operators
$\Delta_t$ and $\Delta_t^{-1}$ acting on a function
$f \colon \ZZ \to \CC$ as
\[
\Delta_t f(\nu) = f(\nu+1)-t f(\nu),
\qquad\qquad
\Delta_t^{-1} f(\nu)=\sum_{\mu=-\infty}^{\nu-1} t^{\nu-1-\mu} f(\mu).
\]
Note that $(\Delta_t^{-1} \Delta_t) f(\nu) = (\Delta_t \Delta_t^{-1}) f(\nu) = f(\nu)$ and that $\Delta_{t'} \Delta_t f = \Delta_t \Delta_{t'} f$.
We define the multiple difference operators $\Delta_{\bt}^{j-i}$ as
\begin{equation}
\Delta_{\bt}^{j-i} := \begin{cases}
\Delta_{t_i} \cdots \Delta_{t_{j-1}} & \text{if } j \geq i, \\
\Delta_{t_j}^{-1} \cdots \Delta_{t_{i-1}}^{-1} & \text{if } j < i.
\end{cases}
\label{eq:defrefinedmultipledifference}
\end{equation}
Note that $\Delta_{\bt}^{j-i}$ does depend on the values $i$ and $j$ and not simply the difference.
Let us also define $v(\nu) := x^\nu H(\nu)$.

\begin{lemma} \label{lemma:onevariablelemma}
The following identity holds:
\begin{align}
\det \bigl[ \Delta_\bt^{j-i} v(\lambda_i-\mu_j) \bigr]_{i,j=1}^{\ell}
=\prod_{j=1}^\ell t_j^{-\mu_j+\max(\mu_j,\lambda_{j+1})}
v(\lambda_j-\max(\mu_j,\lambda_{j+1})).
\label{determinantformonevariable}
\end{align}
\end{lemma}

Lemma~\ref{lemma:onevariablelemma} is essentially a refined version of~\cite[Lemma~3.1]{Johansson10}, and it can be shown using the same argument based on induction on $\ell$.
Let us show an example of proving the case $\ell=3$ of Lemma~\ref{lemma:onevariablelemma} by assuming that the identity holds for the case $\ell=2$.
We consider the $\ell=3$ case of the left hand side of~\eqref{determinantformonevariable}.
Expanding the determinant along the first column, we get
\begin{align}
\det\bigl[ \Delta_\bt^{j-i} v(\lambda_i-\mu_j) \bigr]_{i,j=1}^{3}
& = v(\lambda_1-\mu_1)
\begin{vmatrix}
v(\lambda_2-\mu_2) & \Delta_{t_2}v(\lambda_2-\mu_3) \\
\Delta_{t_2}^{-1}v(\lambda_3-\mu_2) & v(\lambda_3-\mu_3) \\
\end{vmatrix} \nonumber \\
& \hspace{20pt} - \Delta_{t_1}^{-1} v(\lambda_2-\mu_1)
\begin{vmatrix}
\Delta_{t_1} v(\lambda_1-\mu_2) & \Delta_{t_1} \Delta_{t_2}v(\lambda_1-\mu_3) \\
\Delta_{t_2}^{-1}v(\lambda_3-\mu_2) & v(\lambda_3-\mu_3) \\
\end{vmatrix} \nonumber \\
& \hspace{20pt} + \Delta_{t_1}^{-1} \Delta_{t_2}^{-1} v(\lambda_3-\mu_1)
\begin{vmatrix}
\Delta_{t_1} v(\lambda_1-\mu_2) & \Delta_{t_1} \Delta_{t_2}v(\lambda_1-\mu_3) \\
v(\lambda_2-\mu_2) & \Delta_{t_2} v(\lambda_2-\mu_3) \\
\end{vmatrix}. \label{columnexpansion}
\end{align}
We can show that the third term of the right hand side of~\eqref{columnexpansion} vanishes.
It is easy to see this when $\lambda_3 < \mu_1$ since $\Delta_{t_1}^{-1} \Delta_{t_2}^{-1} v(\lambda_3-\mu_1)=0$.
Now we consider the case $\lambda_3 \ge \mu_1$ and show the determinant vanishes.
Let $\Delta_{t,\eta}$ denote the $t$-weighted difference operator with respect to the variable $\eta$; that is, it acts on the functions (in the $x$ indeterminate) such that $\eta$ appears as an argument.
Therefore, we can rewrite the determinant as
\[
\begin{vmatrix}
\Delta_{t_1} v(\lambda_1-\mu_2) & \Delta_{t_1} \Delta_{t_2}v(\lambda_1-\mu_3) \\
v(\lambda_2-\mu_2) & \Delta_{t_2} v(\lambda_2-\mu_3) \\
\end{vmatrix}
=\Delta_{t_1,-\mu_2} \Delta_{t_2,-\mu_3}
\begin{vmatrix}
v(\lambda_1-\mu_2) & \Delta_{t_1} v(\lambda_1-\mu_3) \\
\Delta_{t_1}^{-1} v(\lambda_2-\mu_2) & v(\lambda_2-\mu_3) \\
\end{vmatrix}.
\]
We apply the assumption of the identity~\eqref{determinantformonevariable} for the case $\ell=2$ and get
\begin{align*}
&
\begin{vmatrix}
\Delta_{t_1} v(\lambda_1-\mu_2) & \Delta_{t_1} \Delta_{t_2}v(\lambda_1-\mu_3) \\
v(\lambda_2-\mu_2) & \Delta_{t_2} v(\lambda_2-\mu_3) \\
\end{vmatrix}
\\
& = \Delta_{t_1,-\mu_2} \Delta_{t_2,-\mu_3}
t_1^{-\mu_2+\max(\mu_2,\lambda_2)}v(\lambda_1-\max(\mu_2,\lambda_2))
t_2^{-\mu_3+\max(\mu_3,\lambda_3)}v(\lambda_2-\max(\mu_3,\lambda_3)).
\end{align*}
Using $\lambda_1 \ge \lambda_2 \ge \lambda_3 \ge \mu_1 \ge \mu_2 \ge \mu_3$,
we can simplify the right hand side to get
\[
\begin{vmatrix}
\Delta_{t_1} v(\lambda_1-\mu_2) & \Delta_{t_1} \Delta_{t_2}v(\lambda_1-\mu_3) \\
v(\lambda_2-\mu_2) & \Delta_{t_2} v(\lambda_2-\mu_3) \\
\end{vmatrix}
= t_1^{-\mu_2+\lambda_2}t_2^{-\mu_3+\lambda_3} \Delta_{t_1,-\mu_2} \Delta_{t_2,-\mu_3} x^{\lambda_1-\lambda_3}.
\]
Since $\Delta_{t_2,-\mu_3} x^{\lambda_1-\lambda_3} = 0$ as there is no $\mu_3$ dependence, we conclude that
\[
\begin{vmatrix}
\Delta_{t_1} v(\lambda_1-\mu_2) & \Delta_{t_1} \Delta_{t_2}v(\lambda_1-\mu_3) \\
v(\lambda_2-\mu_2) & \Delta_{t_2} v(\lambda_2-\mu_3) \\
\end{vmatrix} = 0.
\]

In any case, the third term of the right hand side of~\eqref{columnexpansion} vanishes.
Now let us consider the case $\lambda_2<\mu_1$.
In this case, the second term also vanishes since $\Delta_{t_1}^{-1}v(\lambda_2-\mu_1)=0$, and we have
\[
\det \bigl[\Delta_\bt^{j-i} v(\lambda_i-\mu_j) \bigr]_{i,j=1}^{3}
=v(\lambda_1-\mu_1)
\begin{vmatrix}
v(\lambda_2-\mu_2) & \Delta_{t_2}v(\lambda_2-\mu_3) \\
\Delta_{t_2}^{-1}v(\lambda_3-\mu_2) & v(\lambda_3-\mu_3) \\
\end{vmatrix}.
\]
Applying the induction hypothesis on the determinant
on the right hand side, we get
\[
\det \bigl[ \Delta_\bt^{j-i} v(\lambda_i-\mu_j) \bigr]_{i,j=1}^{3}
=v(\lambda_1-\mu_1)
\prod_{j=2}^3 t_j^{-\mu_j+\max(\mu_j,\lambda_{j+1})}
v(\lambda_j-\max(\mu_j,\lambda_{j+1})).
\]
Since $\lambda_2<\mu_1$, we can rewrite $v(\lambda_1-\mu_1)$ as
$
v(\lambda_1-\mu_1) = t_1^{-\mu_1+\max(\mu_1,\lambda_2)}v(\lambda_1-\max(\mu_1,\lambda_2)),
$
and we get
\[
\det \bigl[ \Delta_\bt^{j-i} v(\lambda_i-\mu_j) \bigr]_{i,j=1}^{3}
=\prod_{j=1}^3 t_j^{-\mu_j+\max(\mu_j,\lambda_{j+1})}
v(\lambda_j-\max(\mu_j,\lambda_{j+1})),
\]
which is nothing but~\eqref{determinantformonevariable} for the case $\ell=3$.

Finally, let us consider the case $\lambda_2 \ge \mu_1$.
Recall the column expansion
\begin{align}
&\det \bigl[ \Delta_\bt^{j-i} v(\lambda_i-\mu_j) \bigr]_{i,j=1}^{3}
\nonumber \\
&=v(\lambda_1-\mu_1)
\begin{vmatrix}
v(\lambda_2-\mu_2) & \Delta_{t_2}v(\lambda_2-\mu_3) \\
\Delta_{t_2}^{-1}v(\lambda_3-\mu_2) & v(\lambda_3-\mu_3) \\
\end{vmatrix} \nonumber \\
& \hspace{20pt} - \Delta_{t_1}^{-1} v(\lambda_2-\mu_1)
\begin{vmatrix}
\Delta_{t_1} v(\lambda_1-\mu_2) & \Delta_{t_1} \Delta_{t_2}v(\lambda_1-\mu_3) \\
\Delta_{t_2}^{-1}v(\lambda_3-\mu_2) & v(\lambda_3-\mu_3) \\
\end{vmatrix}. \label{columnexpansiontwoterms}
\end{align}
We again apply the induction assumption on the first term of the right hand side of~\eqref{columnexpansiontwoterms} to get
\begin{align}
&v(\lambda_1-\mu_1)
\begin{vmatrix}
v(\lambda_2-\mu_2) & \Delta_{t_2}v(\lambda_2-\mu_3) \\
\Delta_{t_2}^{-1}v(\lambda_3-\mu_2) & v(\lambda_3-\mu_3) \\
\end{vmatrix} \nonumber \\
& = v(\lambda_1-\mu_1)
\prod_{j=2}^3 t_j^{-\mu_j+\max(\mu_j,\lambda_{j+1})}
v(\lambda_j-\max(\mu_j,\lambda_{j+1})).
\label{useone}
\end{align}
Let us look at the second term.
First, it is easy to see
\begin{align}
\Delta_{t_1}^{-1}v(\lambda_2-\mu_1)
=t_1^{\lambda_2-\mu_1-1} \frac{1-t_1^{\mu_1-\lambda_2}x^{\lambda_2-\mu_1}}{1-t_1^{-1}x}. \label{usetwo}
\end{align}
We next rewrite the determinant as
\begin{align}
&\begin{vmatrix}
\Delta_{t_1} v(\lambda_1-\mu_2) & \Delta_{t_1} \Delta_{t_2}v(\lambda_1-\mu_3) \\
\Delta_{t_2}^{-1}v(\lambda_3-\mu_2) & v(\lambda_3-\mu_3) \\
\end{vmatrix}
=\Delta_{t_1,\lambda_1}
\begin{vmatrix}
v(\lambda_1-\mu_2) & \Delta_{t_2}v(\lambda_1-\mu_3) \\
\Delta_{t_2}^{-1}v(\lambda_3-\mu_2) & v(\lambda_3-\mu_3) \\
\end{vmatrix},
\end{align}
and apply the induction assumption to get
\begin{align*}
& \begin{vmatrix}
\Delta_{t_1} v(\lambda_1-\mu_2) & \Delta_{t_1} \Delta_{t_2}v(\lambda_1-\mu_3) \\
\Delta_{t_2}^{-1}v(\lambda_3-\mu_2) & v(\lambda_3-\mu_3) \\
\end{vmatrix} \\
& \hspace{20pt} = \Delta_{t_1,\lambda_1}
t_2^{-\mu_2+\max(\mu_2,\lambda_3)}v(\lambda_1-\max(\mu_2,\lambda_3))
t_3^{-\mu_3+\max(\mu_3,\lambda_4)}v(\lambda_3-\max(\mu_3,\lambda_4)).
\end{align*}
It is easy to see $\Delta_{t_1,\lambda_1}v(\lambda_1-\max(\mu_2,\lambda_3))
=(x-t_1)v(\lambda_1-\max(\mu_2,\lambda_3))$,
and we have
\begin{align}
& \begin{vmatrix}
\Delta_{t_1} v(\lambda_1-\mu_2) & \Delta_{t_1} \Delta_{t_2}v(\lambda_1-\mu_3) \\
\Delta_{t_2}^{-1}v(\lambda_3-\mu_2) & v(\lambda_3-\mu_3) \\
\end{vmatrix} \nonumber \\
& \hspace{20pt} = (x-t_1)
t_2^{-\mu_2+\max(\mu_2,\lambda_3)}v(\lambda_1-\max(\mu_2,\lambda_3))
t_3^{-\mu_3+\max(\mu_3,\lambda_4)}v(\lambda_3-\max(\mu_3,\lambda_4)).
\label{usethree}
\end{align}
Inserting~\eqref{useone},~\eqref{usetwo}, and~\eqref{usethree} into~\eqref{columnexpansiontwoterms} yields
\begin{align*}
\det \bigl[ \Delta_\bt^{j-i} v(\lambda_i-\mu_j) \bigr]_{i,j=1}^{3}
& = v(\lambda_1-\mu_1)
\prod_{j=2}^3 t_j^{-\mu_j+\max(\mu_j,\lambda_{j+1})}
v(\lambda_j-\max(\mu_j,\lambda_{j+1})) \\
& \hspace{20pt} - (x^{\lambda_2-\mu_1}-t_1^{\lambda_2-\mu_1})
t_2^{-\mu_2+\max(\mu_2,\lambda_3)}v(\lambda_1-\max(\mu_2,\lambda_3))
\\ & \hspace{40pt} \times t_3^{-\mu_3+\max(\mu_3,\lambda_4)}v(\lambda_3-\max(\mu_3,\lambda_4)).
\end{align*}
We can further rewrite this using
\begin{align*}
v(\lambda_1-\mu_1)v(\lambda_2-\max(\mu_2,\lambda_3) & = x^{\lambda_2-\mu_1}v(\lambda_1-\max(\mu_2,\lambda_3)), \\
v(\lambda_1-\max(\mu_2,\lambda_3)) & = v(\lambda_1-\lambda_2) v(\lambda_2-\max(\mu_2,\lambda_3)),
\end{align*}
to obtain
\[
\det \bigl[ \Delta_\bt^{j-i} v(\lambda_i-\mu_j) \bigr]_{i,j=1}^{3}
=t_1^{\lambda_2-\mu_1}v(\lambda_1-\lambda_2)
\prod_{j=2}^3 t_j^{-\mu_j+\max(\mu_j,\lambda_{j+1})}
v(\lambda_j-\max(\mu_j,\lambda_{j+1})).
\]
Finally, since we are dealing the case $\lambda_2 \ge \mu_1$,
we can rewrite
\[
t_1^{\lambda_2-\mu_1}v(\lambda_1-\lambda_2) = t_1^{-\mu_1+\max(\mu_1,\lambda_2)}v(\lambda_1-\max(\mu_1,\lambda_2)),
\]
and hence we obtain
\[
\det \bigl[ \Delta_\bt^{j-i} v(\lambda_i-\mu_j) \bigr]_{i,j=1}^{3}
=\prod_{j=1}^3 t_j^{-\mu_j+\max(\mu_j,\lambda_{j+1})}
v(\lambda_j-\max(\mu_j,\lambda_{j+1})).
\]
This completes the induction step from $\ell=2$ to $\ell=3$.

From~\eqref{eq:factorized_form_one_var} and
Lemma~\ref{lemma:onevariablelemma}, we get the following determinant form
for $g_{\lambda/\mu}(x;\bt)$.
\begin{align}
g_{\lambda/\mu}(x;\bt)
=\det \bigl[ \Delta_\bt^{j-i} v(\lambda_i-\mu_j) \bigr]_{i,j=1}^{\ell}.
\label{refineddeterminantformonevariable}
\end{align}
To derive a determinant form for $g_{\lambda/\mu}(\xx;\bt)$, we combine~\eqref{refineddeterminantformonevariable} with the following formula, which is a refined version of~\cite[Lemma~3.2]{Johansson10}.

\begin{prop}
Let $f$, $g$ be functions $f,g \colon \ZZ \to \CC$ such that for some $M$ and all $\nu < M$, we have $f(\nu)=g(\nu)=0$.
We have
\begin{align}
\sum_{\nu_\ell \le \cdots \le \nu_2 \le \nu_1} \det \bigl[ \Delta_\bt^{j-i}f(\nu_i-\mu_j) \bigr]_{i,j=1}^{\ell}
\det \bigl[ \Delta_\bt^{j-i}g(\lambda_i-\nu_j) \bigr]_{i,j=1}^{\ell}
= \det \bigl[ \bigl(\Delta_\bt^{j-i}( f*g) \bigr)(\lambda_i-\mu_j) \bigr]_{i,j=1}^{\ell}.
\label{eq:refined_convolution}
\end{align}
\end{prop}

\begin{proof}
The proof given in~\cite[Lemma~3.2]{Johansson10} can be applied to
show~\eqref{eq:refined_convolution} as well.
We first rewrite the left hand side of~\eqref{eq:refined_convolution}
using the summation by parts formula
\[
\sum_{\mu=a}^b \Delta_t u(\mu-\lambda)v(\nu-\mu)
=\sum_{\mu=a}^b u(\mu-\lambda) \Delta_t v(\nu-\mu)
+u(b+1-\lambda) v(\nu-b)-u(a-\lambda)v(\nu+1-a)
\]
as
\begin{align}
\displaystyle
\sum_{\nu_\ell \le \cdots \le \nu_2 \le \nu_1}
\det\bigl[ \Delta_{t_j}^{-1} \cdots \Delta_{t_{\ell-1}}^{-1}f(\nu_i-\mu_j) \bigr]_{i,j=1}^{ \ell}
\det\bigl[ \Delta_{t_j} \cdots \Delta_{t_{\ell-1}}g(\lambda_j-\nu_i) \bigr]_{i,j=1}^{\ell}. \label{stepone}
\end{align}
Note that the sum in~\eqref{stepone} can be replaced by the sum with the order is removed and multiplying by the overall factor $(n!)^{-1}$.
Then we apply the generalized Cauchy--Binet formula
\begin{align}
\det \Bigg[
\int_X \phi_i(x)\psi_j(x)d \mu(x)
\Bigg]_{i,j=1}^n=\frac{1}{n!} \int_{X^n}
\det[\phi_i(x_j)]_{i,j=1}^n \det[\psi_i(x_j)]_{i,j=1}^n
\prod_{j=1}^n d \mu (x_j)
\label{generalizedCauchyBinet}
\end{align}
to~\eqref{stepone}, which yields
\begin{align}
\det \Bigg[ \sum_{\nu \in \mathbb{Z}}
\Delta_{t_j}^{-1} \cdots \Delta_{t_{\ell-1}}^{-1} f(\nu-\mu_j)
\Delta_{t_i} \cdots \Delta_{t_{\ell-1}} g(\lambda_i-\nu) \Bigg]_{i,j=1}^\ell.
\label{applyingCauchyBinet}
\end{align}
Finally, using $\Delta_{t}^{-1}f(\nu)=(j_t * f)(\nu)$ for $j_t(\nu)
=t^{\nu-1}H(\nu-1)$ one can show
\begin{align}
\sum_{\nu \in \mathbb{Z}}
\Delta_{t_j}^{-1} \cdots \Delta_{t_{\ell-1}}^{-1} f(\nu-\mu_j)
\Delta_{t_i} \cdots \Delta_{t_{\ell-1}} g(\lambda_i-\nu)
= \bigl( \Delta_{\bt}^{j-i}(f*g)\bigr)(\lambda_i-\mu_j). \label{equalitytoshowconvolution}
\end{align}
Combining~\eqref{applyingCauchyBinet} and~\eqref{equalitytoshowconvolution},
we get the right hand side of~\eqref{eq:refined_convolution}.
\end{proof}

From~\eqref{refineddeterminantformonevariable}, \eqref{refineddefmultskew}, and~\eqref{eq:refined_convolution}, we find that $g_{\lambda/\mu}(\xx;\bt)$ can be expressed as a single determinant
\begin{align}
g_{\lambda/\mu}(\xx;\bt)=\det \bigl[ \bigl(\Delta_\bt^{j-i} (f_1*f_2*\cdots*f_n) \bigr)(\lambda_i-\mu_j) \bigr]_{i,j=1}^{\ell},
\label{singledeterminantexpression}
\end{align}
where $f_j(\nu)=h_\nu(x_j)=x_j^\nu H(\nu)$.
It is easy to see that
\begin{align}
f_1*f_2*\cdots*f_n(\nu)=h_\nu(\xx), \label{fundamentalconvolution}
\end{align}
where $\xx=\{x_1,\dotsc,x_n \}$.
Using
\eqref{fundamentalconvolution}, we can rewrite
\eqref{singledeterminantexpression} as
\begin{align}
g_{\lambda/\mu}(\xx;\bt)=\det \bigl[ (\Delta_\bt^{j-i} h_\nu(\xx))|_{\nu=\lambda_i-\mu_j} \bigr]_{i,j=1}^{\ell}.
\label{refineddeterminantformula}
\end{align}
Note that the difference operators act on the subscript $\nu$
of $h_\nu(\xx)$.
To further rewrite the determinant form~\eqref{refineddeterminantformula},
we examine the action of the multiple weighted difference operators.

One can show
\begin{align}
\Delta_{t_k} \cdots \Delta_{t_2} \Delta_{t_1} f(\nu)
=\sum_{m=0}^{k} e_m(-t_1,-t_2,\dotsc,-t_k) f(\nu+k-m),
\label{refinedmultipleaction}
\end{align}
for any positive integer $k$, which can be shown by induction on $k$.
Note that~\eqref{refinedmultipleaction} can also be written as
\[
\Delta_{t_k} \cdots \Delta_{t_2} \Delta_{t_1} f(\nu)
=\sum_{m=0}^{\infty} e_m(-t_1,-t_2,\dotsc,-t_k) f(\nu+k-m),
\]
since $\displaystyle e_m(-t_1,-t_2,\dotsc,-t_k)=0$ for $m > k$.

One can also show
\begin{align}
\Delta_{t_k}^{-1} \cdots \Delta_{t_2}^{-1} \Delta_{t_1}^{-1} f(\nu)
=\sum_{m=0}^{\infty} h_m(t_1,t_2,\dotsc,t_k) f(\nu-k-m),
\label{tdeformedinversemultiple}
\end{align}
for any positive integer $k$, which can also be shown by
induction on $k$ in the following way.
The case $k=1$ is nothing but the definition of the action
of $\Delta_t^{-1}$ on $f(\nu)$.
Assume that~\eqref{tdeformedinversemultiple} holds for $k$.
Applying $\Delta_{t_{k+1}}^{-1}$ to~\eqref{tdeformedinversemultiple},
we get
\begin{align}
\Delta_{t_{k+1}}^{-1} \cdots \Delta_{t_2}^{-1} \Delta_{t_1}^{-1} f(\nu)&=\sum_{m=0}^\infty \sum_{\ell=-\infty}^{\nu-k-m-1}
h_m(t_1,t_2,\dotsc,t_{k}) t_{k+1}^{\nu-k-m-1-\ell} f(\ell) \nonumber \\
&=\sum_{m=0}^\infty \sum_{\ell=0}^m h_\ell(t_1,t_2,\dotsc,t_{k})t_{k+1}^{m-\ell}
f(\nu-k-1-m). \label{beforethefinalform}
\end{align}
It is easy to show
\begin{align}
\sum_{\ell=0}^m h_\ell(t_1,t_2,\dotsc,t_{k})t_{k+1}^{m-\ell}
=h_m(t_1,t_2,\dotsc,t_{k+1}), \label{relationcomplete}
\end{align}
by using the generating function for the complete symmetric polynomials
\[
\sum_{\ell \ge 0}h_\ell(t_1,t_2,\dotsc,t_{k})z^\ell = \frac{1}{\prod_{j=1}^k (1-t_j z)}.
\]
Inserting
\eqref{relationcomplete} into~\eqref{beforethefinalform}, we get
\begin{align}
\Delta_{t_{k+1}}^{-1} \cdots \Delta_{t_2}^{-1} \Delta_{t_1}^{-1} f(\nu)
=\sum_{m=0}^{\infty} h_m(t_1,t_2,\dotsc,t_{k+1}) f(\nu-k-1-m),
\end{align}
which is~\eqref{tdeformedinversemultiple} with $k$ replaced by $k+1$.

Using~\eqref{refinedmultipleaction} and~\eqref{tdeformedinversemultiple}, Equation~\eqref{refineddeterminantformula} can be rewritten as the following determinant formula.

\begin{thm}
\label{thm:refined_Jacobi_Trudi_h}
The skew refined dual Grothendieck polynomials have a determinant formula of
\begin{align}
\displaystyle
g_{\lambda/\mu}(\xx;\bt)=\det
\Bigg[ \sum_{m=0}^\infty \alpha_m^{ij}(\bt)
 h_{\lambda_i-\mu_j-i+j-m}(\xx)  \Bigg]_{i,j=1}^{\ell},
\label{refineddetformafterconvolution}
\end{align}
where $\alpha_m^{ij}(\bt)=h_m(t_j,\dotsc,t_{i-1})$ for $i \ge j$
and $\alpha_m^{ij}(\bt)=e_m(-t_i,\dotsc,-t_{j-1})$ for $i<j$.
\end{thm}

\begin{remark}
\label{remark:proof_differences}
We note that Theorem~\ref{thm:refined_Jacobi_Trudi_h} is the refined version of the Jacobi--Trudi type determinant form recently derived by Iwao~\cite[Prop.~5.2]{Iwao20}\footnote{We need to take $\beta \mapsto -\beta$ in~\cite{Iwao20} to match our definitions.} using the boson-fermion correspondence and by Amanov and Yeliussizov~\cite[Thm.~14]{AY20} using a sign reversing involution.
Furthermore, our proof shows that at $\bt = 1$, this is equivalent to the derivation by using the transition probability formulas developed by Johansson (\textit{cf}.~\cite[Lemma 3.1]{Johansson10}).
Theorem~\ref{thm:refined_Jacobi_Trudi_h} was also recently proven independently by Kim~\cite{Kim20} using plethystic techniques.
\end{remark}

\begin{remark}
We cannot apply the same approach in~\cite[Thm.~14]{AY20} to prove the Jacobi--Trudi formula~\cite[Thm.~3]{AY20} is equivalent to Theorem~\ref{thm:refined_Jacobi_Trudi_h}.
Indeed, there is no analog of the the automorphism that sends $\dG_{\lambda/\mu}(\xx; \beta) \to \dG_{\lambda'/\mu'}(\xx; \beta)$ as in~\cite{Buch02,Yel19} for refined dual Grothendieck polynomials. Consider the following:
\begin{gather*}
\dG_1 \cdot \dG_2 = s_1 \cdot s_2 = s_3 + s_{21} = \dG_3 + \dG_{21} - t_1 \dG_2,
\\ \dG_1 \cdot \dG_{11} = s_1 \cdot (s_{11} + t_1 s_1) = s_{111} + s_{21} + t_1 (s_{11} + s_2) = \dG_{111} + \dG_{21} - t_1 \dG_{11}.
\end{gather*}
However, we note that
\begin{align*}
\dG_{111} + \dG_{21} - \gamma \dG_{11} & = (s_{111} + (t_1 + t_2) s_{11} + t_1 t_2 s_1) + (s_{21} + t_1 s_2) - \gamma (s_{11} + t_1 s_1)
\\ & = s_{111} + s_{21} + (t_1 + t_2 - \gamma) s_{11} + t_1 s_2 + (t_1 t_2 - \gamma t_1) s_1,
\end{align*}
and so we must have $\gamma = t_2$. However, we clearly cannot also interchange $t_1 \leftrightarrow t_2$.
Kim also independently proved~\cite[Sec.~3]{Kim20} that no such involution can exist using plethystic computations.
\end{remark}

From~\eqref{eq:refineddefbylpp}, we have
\begin{align}
\label{eq:probability_from_skew_dual}
P({\bf G}(n)=\lambda|{\bf G}(0)=\mu)
=\prod_{i=1}^\ell \prod_{j=1}^n (1-t_i x_j)
\bt^{\lambda-\mu}
g_{\lambda/\mu}(\xx;\bt^{-1}).
\end{align}
Combining~\eqref{eq:probability_from_skew_dual}, the determinant formula~\eqref{refineddeterminantformula}, and Theorem~\ref{thm:refined_Jacobi_Trudi_h}, we obtain the following refinement of~\cite[Thm.~2.1]{Johansson10}.

\begin{cor}
\label{cor:det_transition_prob_h}
We have the following determinant representations for the transition probability
\begin{align*}
P({\bf G}(n)=\lambda|{\bf G}(0)=\mu)
&=\prod_{i=1}^\ell \prod_{j=1}^n (1-t_i x_j)
\bt^{\lambda-\mu}
\det \bigl[ (\Delta_{\bt^{-1}}^{j-i}
h_\nu(\xx))|_{\nu=\lambda_i-\mu_j} \bigr]_{i,j=1}^{\ell}
\\
&=\prod_{i=1}^\ell \prod_{j=1}^n (1-t_i x_j)
\bt^{\lambda-\mu}
\det
\Bigg[ \sum_{m=0}^\infty \alpha_m^{ij}(\bt^{-1})
 h_{\lambda_i-\mu_j-i+j-m}(\xx)  \Bigg]_{i,j=1}^{\ell}.
\end{align*}
\end{cor}

Let us now remark on the difference between the branching rule and the skew refined dual Grothendieck polynomials.
This discrepancy can be seen in the single variable case by comparing Equation~\eqref{eq:factorized_form_one_var} and Corollary~\ref{cor:branching_rule} (by not considering the $\xx$ variables) and seeing the formulas are different.
So the approach taken in Section~\ref{sec:refined_models} is not amenable to the skew setting.

We can also prove a dual version of our Jacobi--Trudi type identity. We first consider a shorten version of the Heaviside step function defined by $E(\nu) = 1$ if $\nu = 0,1$ and otherwise $E(\nu) = 0$. So this is a discrete analog of the indicator function of $[0, 1]$. Define functions $\widetilde{f}_j(\nu) := x_j^{\nu} E(\nu)$ and
\[
\widetilde{\alpha}_m^{ij}(\bt) = \begin{cases}
e_m(t_{\mu_j+1}, \dotsc, t_{\lambda_i-1}) & \text{if } \mu_j \geq \lambda_i - 1, \\
h_m(-t_{\lambda_i}, \dotsc, -t_{\mu_j}) & \text{otherwise}.
\end{cases}
\]

\begin{thm}
\label{thm:refined_Jacobi_Trudi_e}
We have the following determinant representations for the skew refined dual Grothendieck polynomials:
\[
\dG_{\lambda/\mu}(\xx; \bt) = \det \Bigg[ \sum_{m=0}^\infty \widetilde{\alpha}_m^{ij}(\bt^{-1}) e_{\lambda_i-\mu_j-i+j-m}(\xx)  \Bigg]_{i,j=1}^{\ell}.
\]
\end{thm}

\begin{proof}
The proof is largely the same as for Theorem~\ref{thm:refined_Jacobi_Trudi_h}.
We first consider the single variable case and prove the following analog of Lemma~\ref{lemma:onevariablelemma}:
\begin{equation}
\label{eq:one_variable_e}
\det \bigl[ \Delta_{-\bt}^{\lambda_i-\mu_j-1} \widetilde{f}_n(j-i+1) \bigr]_{i,j=1}^{\ell}
= \prod_{i=1}^{\ell} \left( \prod_{j=\mu_i}^{\lambda_i - 1} t_j \right) \widetilde{f}_n\bigl( \min(1, \lambda_i - \mu_i) \bigr).
\end{equation}
We prove Equation~\eqref{eq:one_variable_e} by induction. We first note that
\[
\Delta_{-\bt}^{\lambda_1-\mu_1-1} \widetilde{f}_n(1) = \prod_{j=\mu_i}^{\lambda_i-1} t_j \widetilde{f}_n\bigl( \min(1, \lambda_i - \mu_i) \bigr)
\]
by splitting this into the case when $\lambda_1 = \mu_1$ and when $\lambda_1 > \mu_1$ (since $\mu \subseteq \lambda$, we cannot have $\lambda_1 < \mu_i$ for all $i$).
Next we split the induction step into two cases depending on if $\lambda_1 > \mu_2$ or $\lambda_1 = \mu_2$.
We first consider $\lambda_1 > \mu_2$. Then we expand the determinant along the first row (so $i = 1$), and we have
\[
\Delta_{-\bt}^{\lambda_1 - \mu_j - 1} \widetilde{f}_n(j) = 0
\]
for all $j > 1$ since $\lambda_1 - \mu_j - 1 \geq 0$ and Equation~\eqref{refinedmultipleaction}. So Equation~\eqref{eq:one_variable_e} holds in this case. Now we assume $\lambda_1 = \mu_2$, and so $\lambda_i \leq \mu_1$ for all $i$. We expand the determinant along the first column (so $j = 1$), and we have
\[
\Delta_{-\bt}^{\lambda_i-\mu_1-1} \widetilde{f}_n(2-i) = 0
\] 
for all $i > 2$ since $\lambda_i-\mu_1-1 < 0$ and Equation~\eqref{tdeformedinversemultiple} with $2-i \leq 0$.

Now we apply Proposition~\ref{prop:refined_branching} and~\eqref{eq:refined_convolution}, using~\eqref{refinedmultipleaction} and~\eqref{tdeformedinversemultiple} with
\[
\widetilde{f}_1 \ast \widetilde{f}_2 \ast \cdots \ast \widetilde{f}_n(\nu) = e_{\nu}(\xx)
\]
to obtain the desired analog of Theorem~\ref{thm:refined_Jacobi_Trudi_h}.
\end{proof}

We remark that Theorem~\ref{thm:refined_Jacobi_Trudi_e} was also proven in~\cite{AY20,Kim20} using different techniques, as well as for $\bt = \beta$ in~\cite{Iwao20}, in parallel to Remark~\ref{remark:proof_differences}.

Similar to Corollary~\ref{cor:det_transition_prob_h}, we have the following corollary of Theorem~\ref{thm:refined_Jacobi_Trudi_e}.

\begin{cor}
We have the following determinant representations for the transition probability:
\begin{align*}
P({\bf G}(n)=\lambda'|{\bf G}(0)=\mu')
&=\prod_{i=1}^\ell \prod_{j=1}^n (1-t_i x_j) \bt^{\lambda'-\mu'}
\det \bigl[ (\Delta_{-\bt^{-1}}^{\lambda_i-\mu_j-1} e_\nu(\xx))|_{\nu=j-i+1} \bigr]_{i,j=1}^{\ell}
\\
&=\prod_{i=1}^\ell \prod_{j=1}^n (1-t_i x_j)
\bt^{\lambda-\mu}
\det
\Bigg[ \sum_{m=0}^\infty \widetilde{\alpha}_m^{ij}(\bt^{-1})
e_{\lambda_i-\mu_j-i+j-m}(\xx)  \Bigg]_{i,j=1}^{\ell}.
\end{align*}
\end{cor}

We expect our techniques can easily be modified to impose the flagging condition to provide an alternative proof of the results of~\cite{Kim20}.

\subsection{The Schur measure and further results}

In this subsection, we focus on the probability $P(G(\ell,n) \le m)$ and present more results and alternative proofs of some of our previous results.
We note that $\lambda \subseteq m^{\ell}$ is equivalent to $0 \leq \lambda_{\ell} \leq \cdots \leq \lambda_2 \leq \lambda_1 \leq m$.

The following probability measure called the Schur measure came out of this probability.

\begin{dfn}[{\cite{Okounkov00,Okounkov01}}]
Let $t_i, x_i \in (0,1)$ such that $\displaystyle 0 < \prod_{i,j}(1-t_i x_j) < \infty $.
The \defn{Schur measure} is the measure on the set of partitions defined by
\begin{align}
P_{\mathrm{Schur}}(\lambda) := \prod_{i,j} (1 - t_i x_j) s_\lambda(\bt) s_\lambda(\xx).
\end{align}
\end{dfn}

The following is well known.

\begin{thm}[{\cite{Johansson00,BR01}}]
\label{thm:Schur_measure_LPP}
We have
\begin{equation}
\label{eq:Schur_measure_expr}
P(G(\ell,n) \le m) = \sum_{\lambda:\lambda_1 \le m} P_{\mathrm{Schur}}(\lambda) =\prod_{i=1}^\ell \prod_{j=1}^n (1-t_i x_j) \sum_{\lambda:\lambda_1 \le m} s_{\lambda}(\xx) s_{\lambda}(\bt).
\end{equation}
\end{thm}

The standard way to show~\eqref{eq:Schur_measure_expr} is to use the RSK correspondence.
Additionally, we can obtain~\eqref{eq:Schur_measure_expr} as a consequence of Theorem~\ref{thm:refined_last_passage}, Corollary~\ref{cor:refined_coincidence_lemma}, and the finite Cauchy formula for Schur functions~\eqref{eq:finite_Cauchy_formula}.
We will give another proof of~\eqref{eq:Schur_measure_expr} by using the results for the refined dual Grothendieck polynomials from Section~\ref{sec:transition_JT_formula}, well-known formulas for Schur functions, and refining the argument in~\cite{Johansson10}.

We begin by deriving a single determinant expression for $P(G(\ell,n) \le m)$.
Using Theorem~\ref{thm:refined_last_passage} and~\eqref{refineddeterminantformula}, we have
\begin{align}
\displaystyle P(G(\ell,n) \le m)
&=\sum_{0 \le \lambda_\ell \le \cdots \le \lambda_2 \le \lambda_1 \le m}
P(G(\ell,n)=\lambda_1,G(\ell-1,n)=\lambda_2,\dotsc,G(1,n)=\lambda_\ell)
\nonumber \\
&=\sum_{0 \le \lambda_\ell \le \cdots \le \lambda_2 \le m}
\sum_{\lambda_1=\lambda_2}^m
P(G(\ell,n)=\lambda_1,G(\ell-1,n)=\lambda_2,\dotsc,G(1,n)=\lambda_\ell)
\nonumber \\
&=\sum_{0 \le \lambda_\ell \le \cdots \le \lambda_2 \le m}
\sum_{\lambda_1=\lambda_2}^m
\prod_{i=1}^\ell \prod_{j=1}^n (1-t_i x_j) \bt^\lambda
g_\lambda(\xx;\bt^{-1}) \nonumber \\
&=\sum_{0 \le \lambda_\ell \le \cdots \le \lambda_2 \le m}
\sum_{\lambda_1=\lambda_2}^m
\prod_{j=1}^n (1-t_i x_j)
\det \bigl[ t_i^{\lambda_i} (\Delta_{\bt^{-1}}^{j-i}
h_\nu(\xx))|_{\nu=\lambda_i} \bigr]_{i,j=1}^{\ell}.
\label{beforesummingup}
\end{align}
We move $\displaystyle \sum_{\lambda_1=\lambda_2}^m$ into the first row of the determinant in the right hand side of~\eqref{beforesummingup}.
Then the first row becomes
\begin{align}
\displaystyle &\sum_{\lambda_1=\lambda_2}^m
t_1^{\lambda_1} (\Delta_{\bt^{-1}}^{j-1}
h_\nu(\xx))|_{\nu=\lambda_1}
=\Delta_{\bt^{-1}}^{j-2}
\sum_{\lambda_1=\lambda_2}^m
t_1^{\lambda_1} \Delta_{t_1^{-1}} h_\nu(\xx)|_{\nu=\lambda_1} \nonumber \\
& = \Delta_{\bt^{-1}}^{j-2}
(t_1^m h_\nu(\xx)|_{\nu=m+1}-t_1^{\lambda_2-1}
h_\nu(\xx)|_{\nu=\lambda_2})
=t_1^m \Delta_{\bt^{-1}}^{j-2} h_\nu(\xx)|_{\nu=m+1}
-t_1^{\lambda_2-1} \Delta_{\bt^{-1}}^{j-2} h_\nu(\xx)|_{\nu=\lambda_2}.
\label{firstrowsumup}
\end{align}
Using~\eqref{firstrowsumup}, we find that
the right hand side of~\eqref{beforesummingup} can be rewritten as
\begin{align}
\displaystyle
&P(G(\ell,n) \le m) \nonumber \\
& = \prod_{i=1}^\ell \prod_{j=1}^n (1-t_i x_j)
\sum_{0 \le \lambda_\ell \le \cdots \le \lambda_2 \le m}
\det {}^{t}
\bigl[
t_1^m \Delta_{\bt^{-1}}^{j-2} h_\nu(\xx)|_{\nu=m+1}
-t_1^{\lambda_2-1} \Delta_{\bt^{-1}}^{j-2} h_\nu(\xx)|_{\nu=\lambda_2},
\nonumber \\
& \hspace{185pt} t_2^{\lambda_2} \Delta_{\bt^{-1}}^{j-2} h_\nu(\xx)|_{\nu=\lambda_2}, \dotsc,
t_\ell^{\lambda_\ell} \Delta_{\bt^{-1}}^{j-\ell}
h_\nu(\xx)|_{\nu=\lambda_\ell}
\bigr]_{j=1}^\ell \nonumber \\
& = \prod_{i=1}^\ell \prod_{j=1}^n (1-t_i x_j) \nonumber \\
& \times \sum_{0 \le \lambda_\ell \le \cdots \le \lambda_2 \le m}
\det {}^{t}
\bigl[
t_1^m \Delta_{\bt^{-1}}^{j-2} h_\nu(\xx)|_{\nu=m+1},
t_2^{\lambda_2} \Delta_{\bt^{-1}}^{j-2} h_\nu(\xx)|_{\nu=\lambda_2}, \dotsc,
t_\ell^{\lambda_\ell} \Delta_{\bt^{-1}}^{j-\ell}
h_\nu(\xx)|_{\nu=\lambda_\ell}
\bigr]_{j=1}^\ell.
\end{align}
Iterating this procedure, one gets the following determinant representation
for $P(G(\ell,n) \le m)$, which is a refinement of~\cite[Thm.~2.2]{Johansson10}.

\begin{lemma}
We have the following determinant form
\begin{align}
\displaystyle
P(G(\ell,n) \le m)
=\prod_{i=1}^\ell t_i^m \prod_{i=1}^\ell \prod_{j=1}^n (1-t_i x_j)
\det\bigl[
\Delta_{\bt^{-1}}^{j-i-1} h_\nu(\xx)|_{\nu=m+1}
\bigr]_{i,j=1}^{\ell}, \label{refineddeterminantforlpp}
\end{align}
where $\Delta_{\bt^{-1}}^{j-i-1} := \Delta_{{t_i}^{-1}}^{-1} \Delta_{\bt^{-1}}^{j-i}$.
\end{lemma}

Next, we transform~\eqref{refineddeterminantforlpp} to a refined version of~\cite[Prop.~2.4]{Johansson10}.

\begin{lemma}
We have
\begin{align}
\displaystyle
P(G(\ell,n) \le m)
& =\frac{1}{\ell !} \prod_{i=1}^\ell t_i^m \prod_{i=1}^\ell \prod_{j=1}^n (1-t_i x_j)
\nonumber \\
& \hspace{20pt} \times \sum_{\nu_1,\dotsc,\nu_\ell=-\ell+1}^{m}
\det \bigl[
h_{m+1-\nu_j-i}(t_1^{-1},\dotsc,t_i^{-1})
\bigr]_{i,j=1}^{\ell}
\det \bigl[
\Delta_{\bt^{-1}}^{j-1}h_{\nu_i}(\xx)
\bigr]_{i,j=1}^{\ell}
\nonumber \\
& =\frac{1}{\ell !} \prod_{i=1}^\ell t_i^m \prod_{i=1}^\ell \prod_{j=1}^n (1-t_i x_j)
\nonumber \\
& \hspace{20pt} \times \sum_{\nu_1,\dotsc,\nu_\ell=0}^{m+\ell-1}
\det \bigl[
h_{m+\ell-\nu_j-i}(t_1^{-1},\dotsc,t_i^{-1})
\bigr]_{i,j=1}^{\ell}
\det \bigl[
\Delta_{\bt^{-1}}^{j-1}h_{\nu_i-\ell+1}(\xx)
\bigr]_{i,j=1}^{\ell}.
\label{multiplesummationone}
\end{align}
\end{lemma}

\begin{proof}
Recall that $j_t(\nu) = t^{\nu-1}H(\nu-1) = h_{\nu-1}(t)$ satisfies
$\Delta_{t}^{-1}f(\nu) = j_t*f(\nu)$,
and so we have
\begin{align}
\Delta_{t_1}^{-1} \cdots \Delta_{t_k}^{-1}f(\nu) = (j_{t_1}*\cdots*j_{t_k}*f)(\nu).
\label{useforconvolutiontwo}
\end{align}
Next, we have
\begin{align}
(j_{t_1}*j_{t_2}*\cdots*j_{t_k})(\nu)=h_{\nu-k}(t_1,t_2,\dotsc,t_k),
\label{useforconvolutionone}
\end{align}
by induction on $k$. When $t_j = 1$ for $j = 1,\dotsc,k$, Equation~\eqref{useforconvolutionone} can be simplified as
\[
j_{t=1}^{*k}(\nu) = \frac{\prod_{j=1}^{k-1}(\nu-j)}{(k-1)!} H(\nu-k).
\]
Let $\Delta_{\bt^{-1}}^{-i}:=\Delta_{t_1^{-1}}^{-1} \cdots \Delta_{t_i^{-1}}^{-1}$ and $\Delta_{\bt^{-1}}^{j-1}:=\Delta_{t_1^{-1}} \cdots \Delta_{t_{j-1}^{-1}}$.
Using~\eqref{useforconvolutiontwo} and~\eqref{useforconvolutionone}, we get
\begin{align}
&\Delta_{\bt^{-1}}^{j-i-1}h_\nu(\xx)|_{\nu=m+1}
=\Delta_{\bt^{-1}}^{-i}
\Delta_{\bt^{-1}}^{j-1}h_\nu(\xx)|_{\nu=m+1}
=\sum_{\nu \in \mathbb{Z}} h_{m+1-\nu-i}(t_1^{-1},\dotsc,t_i^{-1})
\Delta_{\bt^{-1}}^{j-1}h_\nu(\xx).
\label{beforerestrictingsum}
\end{align}
Since
  $h_{m+1-\nu-i}(t_1^{-1},\dotsc,t_i^{-1}) = 0$ if $\nu>m$ for $1 \le i \le \ell$ and
  $\Delta_{\bt^{-1}}^{j-1}h_\nu(\xx) = 0$ if $\nu \le -\ell$ for $1 \le j \le \ell$,
the sum in~\eqref{beforerestrictingsum} can be restricted to yield
\begin{align}
\Delta_{\bt^{-1}}^{j-i-1}h_\nu(\xx)|_{\nu=m+1}
=\sum_{\nu=-\ell+1}^m h_{m+1-\nu-i}(t_1^{-1},\dotsc,t_i^{-1})
\Delta_{\bt^{-1}}^{j-1}h_\nu(\xx).
\label{afterrestrictingsum}
\end{align}
Inserting~\eqref{afterrestrictingsum} into the right hand side of~\eqref{refineddeterminantforlpp}, we obtain
\begin{align}
\displaystyle
&P(G(\ell,n) \le m) \nonumber \\
& \hspace{40pt} = \prod_{i=1}^\ell t_i^m \prod_{i=1}^\ell \prod_{j=1}^n (1-t_i x_j)
\det
\Bigg[
\sum_{\nu=-\ell+1}^m h_{m+1-\nu-i}(t_1^{-1},\dotsc,t_i^{-1})
\Delta_{\bt^{-1}}^{j-1}h_\nu(\xx)
\Bigg]_{i,j=1}^{\ell}.
\end{align}
Applying the generalized Cauchy--Binet formula~\eqref{generalizedCauchyBinet}, the result follows.
\end{proof}

We will now transform~\eqref{multiplesummationone} to a single Schur function.
To do this, for $\nu = (\nu_1, \dotsc, \nu_{\ell})$, we first let $\inv(\nu)$ denote the number of inversions of $\nu$, the pairs $(i < j)$ such that $\nu_i \geq \nu_j$.
We claim that
\begin{equation}
\label{eq:diff_ops_h_det_schur}
\det \bigl[
\Delta_{\bt^{-1}}^{j-1} h_{\nu_i-\ell+1}(\xx)
\bigr]_{i,j=1}^{\ell}
= \begin{cases}
0 & \text{if $\nu_i = \nu_j$ for some $i \neq j$}, \\
(-1)^{\inv(\nu)} s_{\overline{\nu} - \rho_{\ell}}(\xx) & \text{otherwise},
\end{cases}
\end{equation}
where $\overline{\nu}$ is the partition corresponding to sorting $\nu$ in weakly decreasing order.
To show~\eqref{eq:diff_ops_h_det_schur}, first notice that from the definition of the difference operators, by induction on $j$, we can subtract $e_{j-k}(-t_1^{-1}, \dotsc -t_{j-1}^{-1})$ times the $k$-th column from the $j$-th column for all $k=1,\dotsc,j-1$ to obtain the column $\bigl( h_{\nu_i-\ell+i}(\xx) \bigr)_{i=1}^{\ell}$. Thus the resulting determinant is the Jacobi--Trudi formula for the Schur function $s_{\overline{\nu} - \rho_{\ell}}(\xx)$ up to swapping rows, which corresponds to the inversions.

\begin{ex}
Let $\nu = (10,2,7)$, then we have
\begin{align*}
\det \begin{bmatrix}
h_{8} & h_{9} - t_{1}h_{8} & h_{10} - (t_{1} + t_{2})h_{9} + t_{1} t_{2}h_{8}\\
h_{0} & h_{1} - t_{1}h_0 & h_2 - (t_{1} + t_{2})h_{1} + t_{1} t_{2}h_0 \\
h_{5} & h_6 - t_{1}h_{5} &  h_{7} - (t_{1} + t_{2})h_{6} + t_{1} t_{2}h_{5}
\end{bmatrix}
& =
\det \begin{bmatrix}
h_{8} & h_{9} & h_{10} - (t_{1} + t_{2})h_{9} + t_{1} t_{2}h_{8}\\
h_{0} & h_{1} & h_2 - (t_{1} + t_{2})h_{1} + t_{1} t_{2}h_0 \\
h_{5} & h_6 &  h_{7} - (t_{1} + t_{2})h_{6} + t_{1} t_{2}h_{5}
\end{bmatrix}
\\ & =
\det \begin{bmatrix}
h_{8} & h_{9} & h_{10} \\
h_{0} & h_{1} & h_2 \\
h_{5} & h_6 &  h_{7}
\end{bmatrix}
\\ & = -s_{8,6,2}(\xx),
\end{align*}
where $h_i = h_i(\xx)$ and $h_0(\xx) = 1$.
\end{ex}

Similarly using the fact
$
h_k(t_1^{-1}, \dotsc, t_i^{-1}) = t_i^{-1} h_{k-1}(t_1^{-1}, \dotsc, t_i^{-1}) + h_k(t_1^{-1}, \dotsc, t_{i-1}^{-1}),
$
elementary row operations, reindexing, and the Jacobi--Trudi formula for skew Schur functions, we have for $\nu$ such that $\nu_i \neq \nu_j$ for all $i \neq j$
\begin{equation}
\label{eq:h_tvars_det_schur}
\det \bigl[
h_{m+\ell-\nu_j-i}(t_1^{-1},\dotsc,t_i^{-1})
\bigr]_{i,j=1}^{\ell}
= \det \bigl[
h_{m+\ell-\nu_j-i}(\bt^{-1})
\bigr]_{i,j=1}^{\ell}
= (-1)^{\inv(\nu)} s_{m^{\ell} / (\overline{\nu}-\rho)}(\bt^{-1}).
\end{equation}

By combining~\eqref{eq:diff_ops_h_det_schur} and~\eqref{eq:h_tvars_det_schur}, we can rewrite the right hand side of~\eqref{multiplesummationone} as
\[
\frac{1}{\ell !} \prod_{i=1}^\ell t_i^m \prod_{i=1}^\ell \prod_{j=1}^n (1-t_i x_j)
 \ell ! \sum_{\mu \subseteq m^{\ell}}
(-1)^{2\inv(\nu)} s_{\mu}(\xx)
s_{m^{\ell} / \mu}(\bt^{-1}).
\]
Then using the combinatorial description of the Schur functions (or the finite Cauchy formula for Schur functions; see~\eqref{eq:finite_Cauchy_formula} below), we obtain the following.

\begin{thm}
We have
\begin{align}
P(G(\ell,n) \le m)
=\prod_{i=1}^\ell t_i^m \prod_{i=1}^\ell \prod_{j=1}^n (1-t_i x_j)
s_{m^\ell}(x_1,\dotsc,x_n,t_1^{-1},\dotsc,t_\ell^{-1}).
\label{singleschurexpression}
\end{align}
\end{thm}

Finally, using~\eqref{eq:schur_complement} and the finite Cauchy formula for Schur functions~\eqref{eq:finite_Cauchy_formula}, one can rewrite~\eqref{singleschurexpression} to the expression~\eqref{eq:Schur_measure_expr} using the Schur measure:
\begin{align*}
P(G(\ell,n) \le m)
&=\prod_{i=1}^\ell t_i^m \prod_{i=1}^\ell \prod_{j=1}^n (1-t_i x_j)
s_{m^\ell}(x_1,\dotsc,x_n,t_1^{-1},\dotsc,t_\ell^{-1}) \\
&=\prod_{i=1}^\ell \prod_{j=1}^n (1-t_i x_j)
\sum_{\lambda \subseteq m^\ell}
s_\lambda(x_1,\dotsc,x_n)
\prod_{i=1}^\ell t_i^m s_{\lambda^\dagger}(t_1^{-1},\dotsc,t_\ell^{-1})
\nonumber \\
&=\prod_{i=1}^\ell \prod_{j=1}^n (1-t_i x_j)
\sum_{\lambda \subseteq m^\ell}
s_\lambda(x_1,\dotsc,x_n) s_\lambda(t_1,\dotsc,t_\ell) \\
&=\sum_{\lambda \subseteq m^\ell}
P_{\mathrm{Schur}}(\lambda).
\end{align*}
Therefore, we have completed our new proof of Theorem~\ref{thm:Schur_measure_LPP}.

The computations above additionally give probabilistic proofs of identities for Schur functions and the refined dual Grothendieck polynomials.
For example, combining~\eqref{refineddeterminantforlpp} and~\eqref{eq:Schur_measure_expr}, we get the following identity.

\begin{cor}
We have
\begin{align}
\sum_{\lambda:\lambda_1 \le m} s_\lambda(x_1,\dotsc,x_n)
s_\lambda(t_1,\dotsc,t_\ell)
=\prod_{i=1}^\ell t_i^m
\det\bigl[
\Delta_{\bt^{-1}}^{j-i-1} h_\nu(\xx)|_{\nu=m+1}
\bigr]_{i,j=1}^{\ell}. \label{asummationidentity}
\end{align}
\end{cor}

For example, $\ell=1$ of~\eqref{asummationidentity} is
\begin{align*}
\sum_{\lambda:\lambda_1 \le m} s_\lambda(x_1,\dotsc,x_n)
s_\lambda(t_1)
&=t_1^m \Delta_{t_1^{-1}}^{-1} h_\nu(\xx) \biggr|_{\nu=m+1}
=t_1^m \sum_{k=0}^\infty h_k(t_1^{-1})h_{m-k}(\xx)  \\
&=\sum_{k=0}^\infty t_1^{m-k} h_{m-k}(\xx)
=\sum_{k=0}^\infty h_{m-k}(t_1) h_{m-k}(\xx),
\end{align*}
which is usually proved using the Jacobi--Trudi identity for Schur polynomials.

Another example is the finite weighted Littlewood formula for the refined dual Grothendieck polynomials in Corollary~\ref{cor:littlewood_identity}, which we first proved using nonintersecting lattice paths.
Combining the expression for $P(G(\ell,n) \le m)$ using the refined dual Grothendieck polynomials
\[
P(G(\ell,n) \le m)
= \sum_{\lambda \subseteq m^{\ell}} P(\bG(n)=\lambda)
=\sum_{\lambda \subseteq m^\ell} \prod_{i=1}^\ell \prod_{j=1}^n (1-t_i x_j) \bt^\lambda g_\lambda(\xx; \bt^{-1})
\]
with the single Schur function expression~\eqref{singleschurexpression}, we get
\[
\sum_{\lambda \subseteq m^\ell}
\bt^\lambda g_\lambda(\xx;\bt^{-1})
= \prod_{i=1}^\ell t_i^m s_{m^\ell}(\xx, \bt^{-1}),
\]
which becomes nothing but Corollary~\ref{cor:littlewood_identity} after substituting $\bt^{-1} \mapsto \bt$.

We also give a probabilistic proof of the Cauchy identity of Corollary~\ref{cor:refined_Cauchy} by using two expressions for $P(G(\ell,2n) \le m)$.
First, we take $\xx = (x_1, \dotsc, x_{2n})$ and define the $\yy$-variables as $y_j := x_{n+j}$ for $j = 1, \dotsc, n$.
From~\eqref{singleschurexpression}, we have a single Schur function representation
\begin{equation}
P(G(\ell,2n) \le m)
=\prod_{i=1}^\ell t_i^m
\prod_{i=1}^\ell \prod_{j=1}^n (1-t_i x_j)
\prod_{i=1}^\ell \prod_{j=1}^n (1-t_i y_j)
s_{m^\ell}(\xx, \yy, \bt^{-1})
\label{firstexpressionforcauchy}
\end{equation}
on one hand.
On the other hand, let us expand $P(G(\ell,2n) \le m)$ as
\begin{align}
P(G(\ell,2n) \le m)=\sum_{\mu \subseteq m^\ell}
P({\bf G}(n)=\mu)P(G(\ell,2n) \le m|{\bf G}(n)=\mu),
\label{expansionlpp}
\end{align}
and derive another expression.
Let us examine $P(G(\ell,2n) \le m|{\bf G}(n)=\mu)$.
Using~\eqref{refineddeterminantformula}, we note that~\eqref{expansionlpp} can be written as
\begin{align}
P(G(\ell,2n) \le m|{\bf G}(n)=\mu)
& =\sum_{\lambda \subseteq m^{\ell}}
P({\bf G}(2n)=\lambda|{\bf G}(n)=\mu) \nonumber \\
& =\sum_{\lambda \subseteq m^{\ell}}
\prod_{i=1}^\ell \prod_{j=1}^n (1-t_i y_j) \bt^{\lambda-\mu}
\det \bigl[ \Delta_{\bt^{-1}}^{j-i} h_\nu(\yy)|_{\nu=\lambda_i-\mu_j}  \bigr]_{i,j=1}^\ell. \label{applythesamestep}
\end{align}
We can perform the same computation to~\eqref{applythesamestep} as when we transformed~\eqref{beforesummingup} to a single determinant~\eqref{refineddeterminantforlpp}.
Therefore, we obtain
\begin{align}
P(G(\ell,2n) \le m|{\bf G}(n)=\mu)
=\bt^\mu \prod_{i=1}^\ell t_i^m \prod_{i=1}^\ell \prod_{j=1}^n (1-t_i y_j)
\det \bigl[ \Delta_{\bt^{-1}}^{j-i-1} h_\nu(\yy)|_{\nu=m+1-\mu_j} \bigr]_{i,j=1}^\ell. \label{singledeterminantforcauchy}
\end{align}
Inserting Theorem~\ref{thm:refined_last_passage} and~\eqref{singledeterminantforcauchy} into~\eqref{expansionlpp} results in
\begin{align}
P(G(\ell,2n) \le m)
&=\prod_{i=1}^\ell t_i^m
\prod_{i=1}^\ell \prod_{j=1}^n (1-t_i x_j)
\prod_{i=1}^\ell \prod_{j=1}^n (1-t_i y_j) \nonumber \\
& \hspace{20pt} \times \sum_{\mu \subseteq m^\ell}
g_\mu(\xx;\bt^{-1})
\det \bigl[ \Delta_{\bt^{-1}}^{j-i-1} h_\nu(\yy)|_{\nu=m+1-\mu_j}  \bigr]_{i,j=1}^\ell.
\label{secondexpressionforcauchy}
\end{align}
Combining the two expressions~\eqref{firstexpressionforcauchy} and~\eqref{secondexpressionforcauchy} for $P(G(\ell,2n) \le m)$, we get
\begin{align}
s_{m^\ell}(\xx, \yy, t_1^{-1}, \dotsc, t_\ell^{-1})
= \sum_{\mu \subseteq m^\ell} g_\mu(\xx;\bt^{-1})
\det [ \Delta_{\bt^{-1}}^{j-i-1} h_\nu(\yy)|_{\nu=m+1-\mu_j}  ]_{i,j=1}^\ell.
\label{generalizedcauchy}
\end{align}

Now we take the $t_\ell \to \infty$ limit of~\eqref{generalizedcauchy}, where the left hand side becomes
\[
s_{m^\ell}(x_1,\dotsc,x_n,y_1,\dotsc,y_n,t_1^{-1},\dotsc,t_{\ell-1}^{-1}).
\]
Let us examine the $t_\ell \to \infty$ limit of the determinant
$\det \bigl[ \Delta_{\bt^{-1}}^{j-i-1} h_\nu(\yy)|_{\nu=m+1-\mu_j}  \bigr]_{i,j=1}^\ell$ in the right hand side.
Using
\begin{align*}
\Delta_{\bt^{-1}}^{j-i} h_\nu(\yy)|_{\nu=m-\mu_j}
&=\Delta_{\bt^{-1}}^{j-i-1}(
h_\nu(\yy)|_{\nu=m+1-\mu_j}-t_i^{-1} h_\nu(\yy)|_{\nu=m-\mu_j}), \\
\lim_{t_\ell \to \infty}
\Delta_{\bt^{-1}}^{j-\ell-1} h_\nu(\yy)|_{\nu=m+1-\mu_j}
&=\Delta_{\bt^{-1}}^{j-\ell} h_\nu(\yy)|_{\nu=m-\mu_j},
\end{align*}
and row operations, we get
\begin{equation}
\lim_{t_\ell \to \infty}
\det \bigl [ \Delta_{\bt^{-1}}^{j-i-1} h_\nu(\yy)|_{\nu=m+1-\mu_j}  \bigr]_{i,j=1}^\ell
=
\det \bigl [ \Delta_{\bt^{-1}}^{j-i} h_\nu(\yy)|_{\nu=m-\mu_j} \bigr]_{i,j=1}^\ell.
\label{fromrowoperations}
\end{equation}
Next we claim that
\begin{align}
g_{\mu^{\dagger}}(\yy; (\bt^{\dagger})^{-1})
=\det \bigl[ \Delta_{\bt^{-1}}^{j-i} h_\nu(\yy)|_{\nu=m-\mu_j} \bigr]_{i,j=1}^\ell. \label{toshowtogetcauchy}
\end{align}
To show~\eqref{toshowtogetcauchy}, we first apply the determinant formula~\eqref{refineddeterminantformula} to $g_{\mu^{\dagger}}(\yy; (\bt^{\dagger})^{-1})$ to get
\begin{align}
g_{\mu^{\dagger}}(\yy; (\bt^{\dagger})^{-1})
=\det [ \Delta_{{\bt^\dagger}^{-1}}^{j-i} h_\nu(\yy)|_{\nu=\mu_i^\dagger}
]_{i,j=1}^\ell. \label{toshowtogetcauchytwo}
\end{align}
We consider the elements
$\Delta_{{\bt^\dagger}^{-1}}^{j-i} h_\nu(\yy)|_{\nu=\mu_i^\dagger}$
for the case $j \ge i$.
If one change the indices from $i,j$ to $i'=\ell+1-j'$,
$j' = \ell+1-i$, we can write it as
\begin{align*}
\Delta_{{\bt^\dagger}^{-1}}^{j-i} h_\nu(\yy)|_{\nu=\mu_i^\dagger}
&=\Delta_{(t_i^\dagger)^{-1}} \cdots \Delta_{(t_{j-1}^\dagger)^{-1}}
h_\nu(\yy)|_{\nu=\mu_i^\dagger}  \\
&=\Delta_{t_{\ell-i}^{-1}} \cdots \Delta_{t_{\ell+1-j}^{-1}}
h_\nu(\yy)|_{\nu=\mu_i^\dagger}  \\
&=\Delta_{t_{j'-1}^{-1}} \cdots \Delta_{t_{i'}^{-1}}
h_\nu(\yy)|_{\nu=\mu_{\ell+1-j'}^\dagger}  \\
&=\Delta_{\bt^{-1}}^{j'-i'} h_\nu(\yy)|_{\nu=\mu_{\ell+1-j'}^\dagger},
\end{align*}
and $j \ge i$ becomes $j' \ge i'$.

For the case $j < i$ which corresponds to $j' < i'$, we get
\begin{align*}
\Delta_{{\bt^\dagger}^{-1}}^{j-i} h_\nu(\yy)|_{\nu=\mu_i^\dagger}
&=\Delta_{(t_j^\dagger)^{-1}}^{-1} \cdots \Delta_{(t_{i-1}^\dagger)^{-1}}^{-1}
h_\nu(\yy)|_{\nu=\mu_i^\dagger}  \\
&=\Delta_{t_{\ell-j}^{-1}}^{-1} \cdots \Delta_{t_{\ell+1-i}^{-1}}^{-1}
h_\nu(\yy)|_{\nu=\mu_i^\dagger}  \\
&=\Delta_{t_{i'-1}^{-1}}^{-1} \cdots \Delta_{t_{j'}^{-1}}^{-1}
h_\nu(\yy)|_{\nu=\mu_{\ell+1-j'}^\dagger}  \\
&=\Delta_{\bt^{-1}}^{j'-i'} h_\nu(\yy)|_{\nu=\mu_{\ell+1-j'}^\dagger}.
\end{align*}
Hence we find~\eqref{toshowtogetcauchytwo} can be rewritten as
\[
g_{\mu^{\dagger}}(\yy; (\bt^{\dagger})^{-1})
=\det \left[ \Delta_{\bt^{-1}}^{j-i} h_\nu(\yy)
|_{\nu=\mu_{\ell+1-j}^\dagger} \right]_{i,j=1}^\ell,
\]
which is nothing but~\eqref{toshowtogetcauchy} by noting that $\mu_{\ell+1-j}^\dagger = m - \mu_j$ for $j=1, \dotsc, \ell$.

From~\eqref{fromrowoperations} and~\eqref{toshowtogetcauchy} we have
\begin{equation}
\label{eq:determinant_g_dual}
\lim_{t_\ell \to \infty}
\det \left[ \Delta_{\bt^{-1}}^{j-i-1} h_\nu(\yy)|_{\nu=m+1-\mu_j}  \right]_{i,j=1}^\ell
=g_{\mu^{\dagger}}(\yy; (\bt^{\dagger})^{-1}).
\end{equation}
Therefore, taking the $t_\ell \to \infty$ limit of~\eqref{generalizedcauchy} and using~\eqref{eq:determinant_g_dual}, we have
\[
s_{m^\ell}(\xx,\yy,\bt^{-1}) =\sum_{\mu \subseteq m^\ell} g_\mu(\xx; \bt^{-1}) g_{\mu^{\dagger}}(\yy; (\bt^{\dagger})^{-1}),
\]
which is nothing but the Cauchy identity in Corollary~\ref{cor:refined_Cauchy}.

\subsection{Integral representation}

One can also derive an integral representation of the refined skew dual Grothendieck polynomials from~\eqref{singledeterminantexpression} by refining the argument in~\cite{Johansson10}.
Let $z$ and $\zz = (z_1, z_2, \dotsc, z_{\ell})$ be commuting indeterminates.

First, by using
\[
\sum_{\nu \ge 0} f_j(\nu)z^\nu=\frac{1}{1-x_j z}
\]
and a straightforward induction on $n$, one can show
\begin{align}
\displaystyle (f_1*f_2*\cdots*f_n)(\nu)=
\frac{1}{2 \pi i} \int_{\gamma_r} \frac{dz}{\prod_{m=1}^n (1-x_m z)z^{\nu+1}},
\label{convolutionintegration}
\end{align}
where $\gamma_r$ is a circle centered at the origin
with radius $r$ satisfying $0 < r < x_m^{-1}$ for all $m=1,\dotsc,n$.

Next, one can show by acting the weighted difference operators on the integral representation~\eqref{convolutionintegration} that
\begin{subequations}
\label{eq:difference_integration}
\begin{align}
\Delta_{t_k} \cdots \Delta_{t_2} \Delta_{t_1}
(f_1*f_2*\cdots*f_n)(\nu)
=\frac{1}{2 \pi i} \int_{\gamma_r}
\frac{\prod_{m=1}^k (1-t_m z)}{\prod_{m=1}^n (1-x_m z)z^{\nu+k+1}} dz,
\label{eq:difference_integration_one}
\\
\Delta_{t_k}^{-1} \cdots \Delta_{t_2}^{-1} \Delta_{t_1}^{-1}
(f_1*f_2*\cdots*f_n)(\nu)
=\frac{1}{2 \pi i} \int_{\gamma_r}
\frac{dz}{\prod_{m=1}^n (1-x_m z)\prod_{m=1}^k (1-t_m z)z^{\nu-k+1}},
\label{eq:difference_integration_two}
\end{align}
\end{subequations}
for $k \ge 0$. Equations~\eqref{eq:difference_integration} can be shown by induction on $k$.
The integration circle $\gamma_r$ in~\eqref{eq:difference_integration_two} has to further satisfy $0<r<t_m^{-1}$ for $m=1,\dotsc,k$.

Applying~\eqref{eq:difference_integration} to the determinant~\eqref{singledeterminantexpression}, we get the following.

\begin{prop}
The refined skew dual Grothendieck polynomials
have the following determinant representation:
\begin{equation}
\dG_{\lambda/\mu}(\xx;\bt) = \det [F_{ij}]_{i,j=1}^{\ell},
\label{eq:integral_repr_dG}
\end{equation}
where
\[
F_{ij} = \begin{cases}
\displaystyle \frac{1}{2 \pi i} \int_{\gamma_r} \frac{\prod_{m=i}^{j-1} (1-t_m z)}{\prod_{m=1}^n (1-x_m z) z^{\lambda_i-\mu_j+j-i+1}} \, dz
 & \text{if } j \geq i, \\[12pt]
\displaystyle \frac{1}{2 \pi i} \int_{\gamma_r} \frac{1}{\prod_{m=1}^n (1-x_m z) \prod_{m=j}^{i-1} (1-t_m z) z^{\lambda_i-\mu_j+j-i+1}} \, dz
 & \text{if } j < i.
 \end{cases}
\]
\end{prop}

Note that if one sets $\mu = \emptyset$ and $t_j=0$, $j=1,\dotsc,\ell-1$, then $g_{\lambda/\mu}(\xx;\bt)$ is nothing but the Schur polynomials $s_\lambda(\xx)$.
In this case, we see that~\eqref{eq:integral_repr_dG} becomes the famous integral representation of the Schur polynomials (see, \textit{e.g.},~\cite{AOS96,DKJM83}):
\begin{align}
s_\lambda(\xx)&=\det
\Bigg[
\frac{1}{2 \pi i} \int_{\gamma_r} \frac{dz}
{\prod_{m=1}^n (1-x_m z) z^{\lambda_i-i+j+1}}
\Bigg]_{i,j=1}^{\ell} \nonumber \\
&=\det
\Bigg[
\frac{1}{2 \pi i} \oint \frac{z^{\lambda_i-i+j+n-1}}
{\prod_{m=1}^n (z-x_m)} dz \Bigg]_{i,j=1}^{\ell} \nonumber \\
&=\det
\Bigg[
\frac{1}{2 \pi i} \oint \frac{z_i^{\lambda_i-i+j+n-1}}
{\prod_{m=1}^n (z_i-x_m)} dz_i \Bigg]_{i,j=1}^{\ell} \nonumber \\
&=\frac{1}{(2 \pi i)^R\ell} \oint \cdots \oint
\frac{\prod_{i=1}^{\ell} z_i^{{\lambda_i}+n-i} \det [z_i^{j-1} ]_{i,j=1}^{\ell}}{\prod_{i=1}^\ell \prod_{m=1}^n (z_i-x_m)}
dz_1 \cdots dz_\ell \nonumber \\
&=
\frac{1}{(2 \pi i)^\ell} \oint \cdots \oint
\frac{\prod_{i=1}^{\ell} z_i^{{\lambda_i}+n-i}
\prod_{1 \le i < j \le \ell} (z_j-z_i)
}{\prod_{i=1}^\ell \prod_{m=1}^n (z_i-x_m)}
dz_1 \cdots dz_\ell. \label{Schurintegral}
\end{align}
Here, the integration contour in~\eqref{Schurintegral} surrounds $x_m$ for all $m=1,\dotsc,n$ and the origin.
Note also that~\eqref{Schurintegral} can be regarded as the $x_{n+1} = \cdots = x_\ell=0$ case of the following integral representation
\[
s_\lambda(\xx) =
\frac{1}{(2 \pi i)^\ell} \oint \cdots \oint
\frac{\prod_{i=1}^{\ell} z_i^{{\lambda_i}+\ell-i} \prod_{1 \le i < j \le \ell} (z_j-z_i)}{\prod_{i=1}^\ell \prod_{m=1}^\ell (z_i-x_m)} \, dz_1 \cdots dz_\ell.
\]

We can perform the same analysis in~\eqref{Schurintegral} for general $\bt$ to obtain an integral representation for the refined dual Grothendieck polynomials. The key step is determining a closed form for the corresponding determinant.

\begin{prop}
\label{prop:refined_expansion_det}
Define
\[
G_{ij} = \begin{cases}
\prod_{m=i}^{j-1} (1 - t_m z_i^{-1}) & \text{if } j \geq i, \\
\prod_{m=j}^{i-1} (1 - t_m z_i^{-1})^{-1} & \text{if } j < i.
\end{cases}
\]
We have
\[
\det \bigl[ z_i^{j-1} G_{ij} \bigr]_{i,j=1}^{\ell} = \prod_{i < j} \frac{z_j - z_i}{1- t_i z_j^{-1}} = \prod_{i < j} \frac{(z_j - z_i) z_j}{z_j - t_i}.
\]
Moreover, we have
\[
\dG_{\lambda}(\xx; \bt) =
\frac{1}{(2 \pi i)^\ell} \oint \cdots \oint
\frac{\prod_{i=1}^{\ell} z_i^{{\lambda_i}+\ell-i} \prod_{1\leq i < j \leq \ell} (z_j - z_i) z_j}{\prod_{i=1}^{\ell} \prod_{m=1}^{\ell} (z_i-x_m) \prod_{1\leq i < j \leq \ell} (z_j - t_i)} \, dz_1 \cdots dz_{\ell},
\]
where the contour surrounds $x_m$ for all $m=1,\dotsc,n$; $t_m$  for all $m = 1, \dotsc, \ell$; and the origin.
\end{prop}

\begin{proof}
We first multiply the $i$-th row by $\prod_{m=1}^{i-1} (1 - t_m z_i^{-1})$. Therefore, we have
\begin{align*}
\det \bigl[ z_i^{j-1} G_{ij} \bigr]_{i,j=1}^{\ell} & = \frac{\det \bigl[ z_i^{j-1} \prod_{m=1}^{j-1} (1 - t_m z_i^{-1}) \bigr]_{i,j=1}^{\ell}}{\prod_{i < j} (1- t_i z_j^{-1})}.
\end{align*}
To compute that the determinant on the right hand side is $\prod_{i<j} (z_j - z_i)$, we can follow many of the classical proofs of the Vandermonde determinant to show the claim.
Indeed, we apply factor exhaustion by noting that if $z_i = z_j$ for any $i \neq j$, then the determinant is $0$. Moreover, we note that both the determinant and claim have degree $\binom{n}{2}$ if we consider $\bt$ to be scalars, and the top-degree components correspond to the usual Vandermonde determinant.
\end{proof}


\section{Refined Grothendieck polynomials}
\label{sec:refined_grothendiecks}

We will construct a vertex model for refined Grothendieck functions by using a refined version of the Schur expansion of symmetric Grothendieck functions due to Lenart~\cite[Thm.~2.2]{Lenart00}.
We then relate our vertex model to nonintersecting lattice paths and explore some consequences.
For this section, we consider the indeterminates $\bt = (t_1, t_2, \dotsc, t_n)$.

\subsection{Algebraic expressions}

Here we compute a Schur function expansion for refined Grothendieck functions $\G_{\lambda}(\xx; \bt)$ and then describing $\G_{\lambda}(\xx; \bt)$ as a multi-Schur function.
The Schur function expansion is just a refined version of Lenart~\cite[Thm.~2.2]{Lenart00} using the uncrowding operation of Buch~\cite[Sec.~6]{Buch02} (see also~\cite{RTY18}).
The uncrowding map is given by applying RSK to the reading word of each row considered as a hook shape and recording the distance the number of rows a letter has moved.

\begin{prop}
\label{prop:refined_G_schur}
We have
\[
\G_{\lambda}(\xx; \bt) = \sum_{\mu \supseteq \lambda} E_{\lambda}^{\mu}(-\bt) s_{\mu}(\xx),
\]
where
\[
E_{\lambda}^{\mu}(\bt) := \sum_{\mcT} \prod_{\mcT_{ij}} t_{i-\mcT_{ij}},
\]
where the sum is over all increasing elegant tableaux $\mcT$ of shape $\mu / \lambda$ (\textit{i.e.}, rows are also strictly increasing) and the product is over all values $\mcT_{ij}$ in row $i$ and column $j$.
\end{prop}

\begin{proof}
The extra entries in the $i$-th row of a set-valued tableau add entries $k - i$ in the $k$-th row of the recording tableau under uncrowding.
\end{proof}

\begin{ex}
We have
\[
\G_{21}(\xx; \bt) = s_{21}(\xx) - t_1 s_{22}(\xx) - (t_1 + t_2) s_{211}(\xx) + (t_1^2 + t_1 t_2) s_{221}(\xx) - t_1^2 t_2 s_{222}(\xx).
\]
\end{ex}

Next, we will express $\G_{\lambda}(\xx; \bt)$ as a multi-Schur definition, giving a refined analog of~\cite[Thm.~2.1]{Lenart00} (see also~\cite[Eq.~(2)]{LN14}). In order to state it, we need to define the (Newton) divided difference operators
\[
\partial_i f := \frac{f - s_i f}{x_i - x_{i+1}},
\qquad\qquad
\partial_w f := \partial_{i_1} \dotsm \partial_{i_m} f,
\]
for any reduced expression $w = s_{i_1} \dotsm s_{i_m}$.
This is well-defined as the divided difference operators satisfy the braid relations.

\begin{thm}
\label{thm:refined_Grothendieck_multiSchur}
We have
\begin{align*}
\partial_{w_0} \! \left( x^{\lambda + \rho_n} \prod_{1 \leq i < j \leq n} (1 - t_i x_j) \right) \!
& = (-1)^{\frac{n(n-1)}{2}} \bt^{\rho_n} s_{\widetilde{\lambda}}\bigl(\xx, \xx - (t_1^{-1}), \dotsc, \xx - (t_1^{-1}, \dotsc, t_{n-1}^{-1}) \bigr)
= \G_{\lambda}(\xx; \bt).
\end{align*}
where $\widetilde{\lambda} = (\lambda_1, \lambda_2 + 1, \dotsc, \lambda_n + n - 1)$ and $w_0$ is the longest element of $S_n$.
\end{thm}

\begin{proof}
The first equality is just a modified proof of~\cite[Eq.~(2)]{LN14} (or~\cite[Thm.~2.1]{Lenart00}), and the second equality is just a modified proof of~\cite[Thm.~2.2]{Lenart00} with Proposition~\ref{prop:refined_G_schur} to obtain the final combinatorial interpretation. Indeed, the coefficient of $s_{\lambda}$ in the Schur expansion of $\dG_{\mu}$ is the determinant
\[
\det \left[ e_{\mu_j-j-\lambda_i+i}(t_1, \dotsc, t_{i-1}) \right]_{i,j=1}^n,
\]
and from there, the result follows by applying the LGV lemma.
\end{proof}

\begin{ex}
We verify Theorem~\ref{thm:refined_Grothendieck_multiSchur} for $\xx = (x_1, x_2, x_3)$ and $\lambda = 2$:
\[
\partial_{s_1 s_2 s_1} \bigl( x_1^4 x_2 (1 - t_1 x_2) (1 - t_1 x_3) (1 - t_2 x_3) \bigr)
= t_1^2 s_{211}(\xx) - t_1 s_{21}(\xx) + s_2(\xx)
= G_2(\xx; \bt).
\]
\end{ex}

By applying the well-known identity
$
\pi_{w_0} f = \partial_{w_0} (x^{\rho_n} f)
$
to Theorem~\ref{thm:refined_Grothendieck_multiSchur}, we can also write
\[
\G_{\lambda}(\xx; \bt) = \pi_{w_0} \! \left( x^{\lambda} \prod_{1 \leq i < j \leq n} (1 - t_i x_j) \right).
\]

\subsection{The vertex model}

From Proposition~\ref{prop:refined_G_schur}, the coefficient $E_{\lambda}^{\mu}(\bt)$ is the sum over \emph{reverse conjugate} semistandard tableaux of shape $\mu / \lambda$ with the entry in row $i$ strictly less than $i$ with the natural weight function.
To construct a reverse conjugate semistandard tableau using lattice paths, we need to allow for different paths to touch at corners (the weakly decreasing column condition) and not allow for repeated flat steps (the strictly decreasing row condition).
While this doesn't naturally work for NILPs, it does fit nicely into the concept of the five-vertex model, where we can associate a flat step to either the vertex on the right or left side. We choose the left side.

\begin{table}
\[
\begin{array}{ccccc}
\toprule
\ta_1^{\vee} & \ta_2^{\vee} & \tb_2^{\vee} & \tc_1^{\vee} & \tc_2^{\vee}
\\\midrule
\vertexF{0}{0}{0}{0}
&
\vertexF{1}{1}{1}{1}
&
\vertexF{1}{0}{1}{0}
&
\vertexF{0}{0}{1}{1}
&
\vertexF{1}{1}{0}{0}
\\\midrule
1 & z & 1 & 1 & z
\\\bottomrule
\end{array}
\]
\caption{The $L$-matrix for the Grothendieck jagged part NILP vertex model.}
\label{table:L_matrix_nilp_dual}
\end{table}

Next, consider the starting points $\uu = (u_i)_{i=1}^n$ and ending points $\vv = (v_i)_{i=1}^n$ for the NILP interpretation for a skew Schur function of shape $\mu / \lambda$, which are $u_i = (n - i + \lambda_i, 0)$ and $v_i = (n - i + \mu_i, \ell)$.
To account for the flagging condition, we take as our starting points $u_i = (n - i + \lambda_i, \ell+1-i)$.
However, the vertex model we want will actually look like this reflected over the horizontal axis, so we obtain the $L$-matrix given by Table~\ref{table:L_matrix_nilp_dual}.
Thus the jagged part of the vertex model we obtain has the top boundary given by $\lambda$, left and right boundaries are $0$, and the bottom part is $\mu$.
The $i$-th row from the bottom of the jagged part has $\lambda_i + \ell - i - 1$ vertices left-aligned and a spectral parameter of $t_i$.
We note the resulting lattice model by $\fG_{\lambda}$.
From the natural correspondence between the lattice paths (we can trivially colorize this model), we obtain the following.

\begin{thm}
\label{thm:partition_function_Grothendieck}
The partition function $Z(\fG_{\lambda}; \xx, \bt)$ is the refined Grothendieck function $\G_{\lambda}(\xx; \bt)$.
\end{thm}

For a set-valued tableau $T \in \svt^n(\lambda)$, we consider its image $(T, \mcE)$ under the uncrowding map.
Recall that $T$ is a semistandard tableau of shape $\lambda$ and $\mcE$ is a strict elegant tableau of shape $\mu / \lambda$.
An \defn{inelegant marking} of $T$ is a subscript on the letters in $\mu / \lambda$ whose value is $i - \mcE_{ij}$.

\begin{ex}
\label{ex:inelegant_bijection}
Suppose $n = 6$. Let $\lambda = 64211$.
Then a particular state in the model and the corresponding semistandard tableau with inelegant markings are
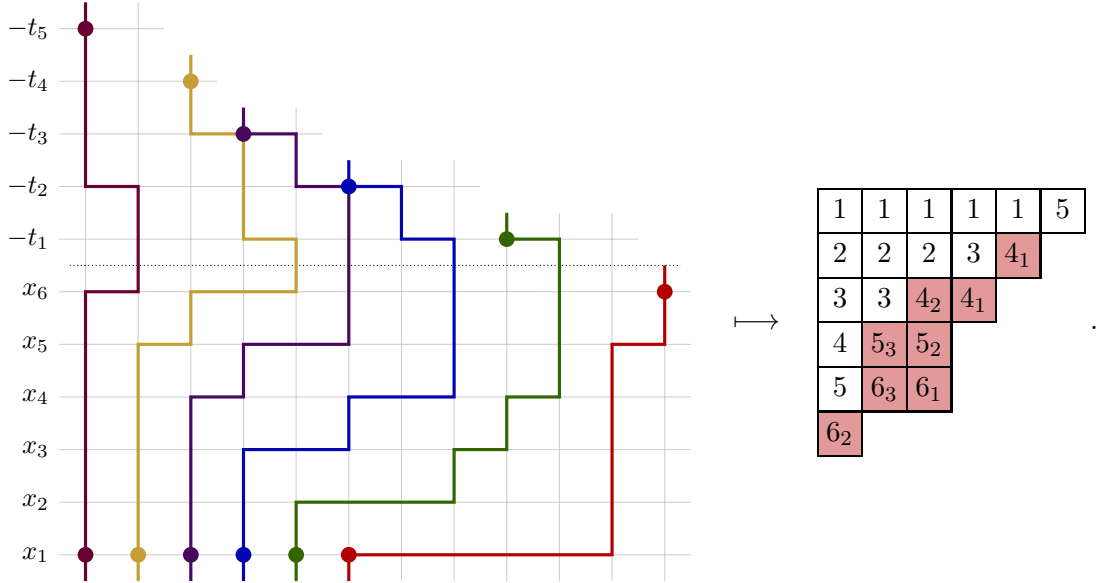
\[
\begin{tikzpicture}[>=latex,baseline=3.7cm,scale=0.7]
\draw[very thin, black!20] (-1.5, 0.5) grid (10.5, 11.5);
\fill[white] (0.5,10.5) rectangle (10.6,11.6);
\fill[white] (1.5,9.5) rectangle (10.6,10.6);
\fill[white] (3.5,8.5) rectangle (10.6,9.6);
\fill[white] (6.5,7.5) rectangle (10.6,8.6);
\fill[white] (9.5,6.5) rectangle (10.6,7.6);
\foreach \y in {1,2,3,4,5,6} {
  \draw node[anchor=east] at (-1.5,\y)  {\small $x_{\y}$};
}
\foreach \y in {1,2,3,4,5} {
  \draw node[anchor=east] at (-1.5,\y+6)  {\small $-t_{\y}$};
}
\draw[densely dotted] (-1.3,6.5) -- (10.3,6.5);
\draw[-, very thick, UMNmaroon] (-1,0.5) -- (-1,6) -- (0,6) -- (0,8) -- (-1,8) -- (-1,11.5);
\draw[-, very thick, UQgold] (0,0.5) -- (0,5) -- (1,5) -- (1,6) -- (3,6) -- (3,7) -- (2,7) -- (2,9) -- (1,9) -- (1,10.5);
\draw[-, very thick, UQpurple] (1,0.5) -- (1,4) -- (2,4) -- (2,5) -- (4,5) -- (4,8) -- (3,8) -- (3,9) -- (2,9) -- (2,9.5);
\draw[-, very thick, dbluecolor] (2,0.5) -- (2,3) -- (4,3) -- (4,4) -- (6,4) -- (6,7) -- (5,7) -- (5,8) -- (4,8) -- (4,8.5);
\draw[-, very thick, dgreencolor] (3,0.5) -- (3,2) -- (6,2) -- (6,3) -- (7,3) -- (7,4) -- (8,4) -- (8,7) -- (7,7) -- (7,7.5);
\draw[-, very thick, darkred] (4,0.5) -- (4,1) -- (9,1) -- (9,5) -- (10,5) -- (10,6.5);
\fill[UMNmaroon] (-1,1) circle (0.15);
\fill[UQgold] (0,1) circle (0.15);
\fill[UQpurple] (1,1) circle (0.15);
\fill[dbluecolor] (2,1) circle (0.15);
\fill[dgreencolor] (3,1) circle (0.15);
\fill[darkred] (4,1) circle (0.15);
\fill[UMNmaroon] (-1,11) circle (0.15);
\fill[UQgold] (1,10) circle (0.15);
\fill[UQpurple] (2,9) circle (0.15);
\fill[dbluecolor] (4,8) circle (0.15);
\fill[dgreencolor] (7,7) circle (0.15);
\fill[darkred] (10,6) circle (0.15);
\end{tikzpicture}
\quad \longmapsto \quad
\ytableaushort{111115,2223{4_1},33{4_2}{4_1},4{5_3}{5_2},5{6_3}{6_1},{6_2}}  *[*(white)]{6,4,2,1,1} *[*(darkred!40)]{6,5,4,3,3,1}\,.
\]
\end{ex}

\begin{prop}
The jagged part is integrable with the $R$-matrix
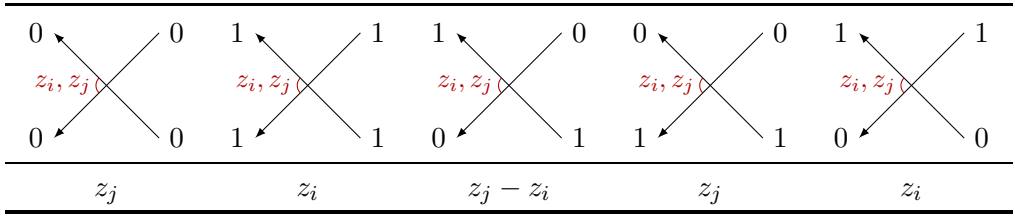
\[
\begin{array}{ccccc}
\toprule
\rmatrixF{0}{0}{0}{0}
&
\rmatrixF{1}{1}{1}{1}
&
\rmatrixF{0}{1}{0}{1}
&
\rmatrixF{1}{1}{0}{0}
&
\rmatrixF{0}{0}{1}{1}
\\ \midrule
z_j & z_i & z_j - z_i & z_j & z_i
\\ \bottomrule
\end{array}
\]
\end{prop}

We can then perform the same extension as Corollary~\ref{cor:refined_dual_symmetries} and show some partial symmetry in the $\bt$ variables by enlarging the model slightly with fixed vertices and applying the standard train argument (Equation~\eqref{eq:standard_train}).

\begin{cor}
Suppose $\lambda_i = \lambda_{i+1}$, then $\G_{\lambda}(\xx; \bt)$ is symmetric in $t_i$ and $t_{i+1}$.
\end{cor}

\begin{ex}
Let $n = 5$, then we have
\begin{align*}
\G_{211}(\xx; \bt) & = s_{211}(\xx) - s_1(t_1, t_2, t_3) s_{2111}(\xx) + s_2(t_1, t_2, t_3) s_{21111}(\xx)
\\ & \hspace{20pt} - t_1 \bigl( s_{221}(\xx) - s_1(t_1, t_2, t_3) s_{2211}(\xx) + s_2(t_1, t_2, t_3) s_{22111}(\xx) \bigr)
\\ & \hspace{20pt} + t_1^2 \bigl( s_{222}(\xx) - s_1(t_1, t_2, t_3) s_{2221}(\xx) + s_2(t_1, t_2, t_3) s_{22211}(\xx) \bigr)
\\ & \hspace{20pt} - t_1^3 \bigl( -s_1(t_2, t_3) s_{2222}(\xx) + s_2(t_2, t_3) s_{22221}(\xx) \bigr)
+ t_1^4 s_2(t_2, t_3) s_{22222}(\xx).
\end{align*}
Note that this is symmetric in $t_2 \leftrightarrow t_3$ and $t_4 \leftrightarrow t_5$.

Let us look at how we can see this using the lattice model. Here we have drawn a particular state in the extended lattice model:
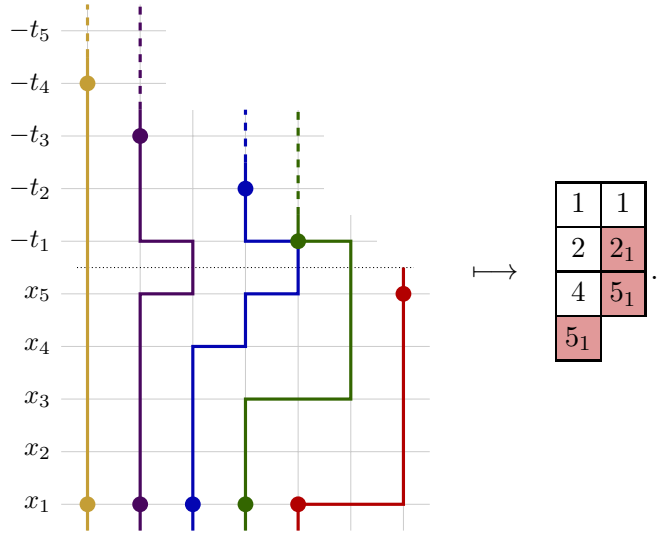
\[
\begin{tikzpicture}[>=latex,baseline=3.7cm,scale=0.7]
\foreach \y in {1,2,3,4,5}
  \draw node[anchor=east] at (-1.5,\y)  {\small $x_{\y}$};
\foreach \y in {1,2,3,4,5}
  \draw node[anchor=east] at (-1.5,\y+5)  {\small $-t_{\y}$};
\draw[very thin, black!20] (-1.5, 1-0.5) grid (5.5, 10.5);
\fill[white] (0.5,8.5) rectangle (5.6,10.6);
\fill[white] (3.5,6.5) rectangle (5.6,10.6);
\fill[white] (4.5,5.5) rectangle (5.6,7.6);
\draw[densely dotted] (-1.2,5.5) -- (5.2,5.5);
\draw[-, very thick, darkred] (3,0.5) -- (3,1) -- (5,1) -- (5,5.5);
\fill[darkred] (3,1) circle (0.15);
\draw[-, very thick, dgreencolor] (2,0.5) -- (2,3) -- (4,3) -- (4,6) -- (3,6) -- (3,6.5);
\draw[dashed, very thick, dgreencolor] (3,6.5) -- (3,8.5);
\fill[dgreencolor] (2,1) circle (0.15);
\draw[-, very thick, dbluecolor] (1,0.5) -- (1,4) -- (2,4) -- (2,5) -- (3,5) -- (3,6) -- (2,6) -- (2,7.5);
\draw[dashed, very thick, dbluecolor] (2,7.5) -- (2,8.5);
\fill[dbluecolor] (1,1) circle (0.15);
\draw[-, very thick, UQpurple] (0,0.5) -- (0,5) -- (1,5) -- (1,6) -- (0,6) -- (0,8.5);
\draw[dashed, very thick, UQpurple] (0,8.5) -- (0,10.5);
\fill[UQpurple] (0,1) circle (0.15);
\draw[-, very thick, UQgold] (-1,0.5) -- (-1,9.5);
\draw[dashed, very thick, UQgold] (-1,9.5) -- (-1,10.5);
\fill[UQgold] (-1,1) circle (0.15);
\fill[darkred] (5,5) circle (0.15);
\fill[dgreencolor] (3,6) circle (0.15);
\fill[dbluecolor] (2,7) circle (0.15);
\fill[UQpurple] (0,8) circle (0.15);
\fill[UQgold] (-1,9) circle (0.15);
\end{tikzpicture}
\quad \longmapsto \quad
\ytableaushort{11,2{2_1},4{5_1},{5_1}} *[*(white)]{2,1,1} *[*(darkred!40)]{2,2,2,1}\,.
\]
We have added an extra row with spectral parameter $-t_5$ for clarify, and the dashed lines correspond to the extended portion that are completely fixed by the vertex model. Notice that we can now apply the $R$-matrix to the rows corresponding to $(-t_2, -t_3)$ and $(-t_4, -t_5)$.
\end{ex}

\begin{prop}
\label{prop:jagged_rect_R_mat}
A pair of jagged and NILP vertices are integrable with the $R$-matrix
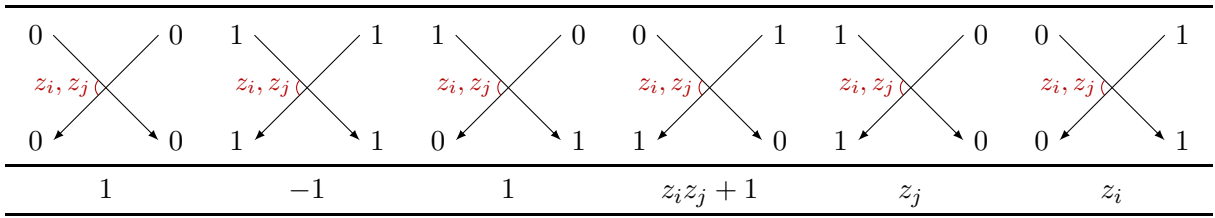
\[
\begin{array}{cccccc}
\toprule
\rmatrixD{0}{0}{0}{0}
&
\rmatrixD{1}{1}{1}{1}
&
\rmatrixD{1}{0}{1}{0}
&
\rmatrixD{0}{1}{0}{1}
&
\rmatrixD{1}{1}{0}{0}
&
\rmatrixD{0}{0}{1}{1}
\\ \midrule
1 & -1 & 1 & z_i z_j + 1 & z_j & z_i
\\ \bottomrule
\end{array}
\]
\end{prop}

We can also extend the model to a rectangle, which we denote by $\overline{\fG}$. However, in order to describe the partition function, we need to develop some additional tools as using the standard train argument introduces additional states with $1$'s on the left and right boundaries.

\subsection{Nonintersecting lattice paths}

We deform our lattice model to obtain a description of the partition functions as a determinant using the LGV lemma. How we do this is noting that in the rectangular part we already have NILPs, but we are allowing paths to touch in the jagged part as previously noted. Since each path can only take one horizontal step at a time, we can instead consider the jagged part as a graph with steps moving vertical or diagonally to the upper-left. This transforms any touch point into a pair of diagonally nonintersecting paths. Furthermore, we can put the weight $-t_i$ on the diagonal paths.
We note that this is a standard way to use NILPs to write a single column (fermionic behavior).

\begin{ex}
Consider the state of the vertex model from Example~\ref{ex:inelegant_bijection}. If we just examine the jagged part, we have
\[
\begin{tikzpicture}[>=latex,baseline=3.7cm,scale=0.7]
\foreach \i in {1,2,...,12}
  \draw[very thin, black!20] (\i,0) -- (\i,5);
\foreach \i in {1,2,...,5} {
  \draw[very thin, black!20] (\i+1,0) -- (1,\i);
  \draw[very thin, black!20] (6+\i,5) -- (12,\i-1);
  \draw[very thin, black!20] (6+\i,0) -- (\i+1,5);
}
\draw[-, very thick, UMNmaroon] (2,0) -- (2,1) -- (1,2) -- (1,5);
\draw[-, very thick, UQgold] (5,0) -- (4,1) -- (4,2) -- (3,3) -- (3,4);
\draw[-, very thick, UQpurple] (6,0) -- (6,1) -- (4,3);
\draw[-, very thick, dbluecolor] (8,0) -- (6,2);
\draw[-, very thick, dgreencolor] (10,0) -- (9,1);
\fill[UMNmaroon] (1,5) circle (0.15);
\fill[UQgold] (3,4) circle (0.15);
\fill[UQpurple] (4,3) circle (0.15);
\fill[dbluecolor] (6,2) circle (0.15);
\fill[dgreencolor] (9,1) circle (0.15);
\fill[darkred] (12,0) circle (0.15);
\end{tikzpicture}
\]
\end{ex}

By the LGV lemma going from vertex $u_j$ on the bottom to $v_i$ in the jagged part, we have the following Jacobi--Trudi formulas for refined Grothendieck polynomials.
The second formula is a refined version of~\cite[Thm.~1.10]{Kirillov16}.

\begin{thm}
\label{thm:refined_JT}
We have
\begin{align*}
\G_{\lambda}(\xx; \bt) & = \det \left[
\sum_{k=0}^{\infty} e_{k-\lambda_i+i-j}(-t_1, \dotsc, -t_{i-1}) h_k(\xx)
\right]_{i,j=1}^n
\\ & = \det \left[
\sum_{k=0}^{\infty} e_{k}(-t_1, \dotsc, -t_{i-1}) h_{k+\lambda_i-i+j}(\xx)
\right]_{i,j=1}^n.
\end{align*}
\end{thm}

We note that each of the sums in Theorem~\ref{thm:refined_JT} are finite as there are only finitely many variables~$\bt$.

Now we return to the rectangular model $\overline{\fG}$ and use that to compute a new identity for refined Grothendieck polynomials.
We take a cut using the same idea utilized in the proof of Corollary~\ref{cor:refined_Cauchy} to obtain the following.

\begin{thm}
\label{thm:extended_height_Grothendieck_model}
Let $\widetilde{\bt} = (t_1, \dotsc, t_{\ell})$.
We have
\begin{align*}
Z(\overline{\fG}; \xx, \widetilde{\bt}) & =
\det \left[
\sum_{k=0}^{\infty} e_{k+i-j}(-\widetilde{\bt}) h_k(\xx)
\right]_{i,j=1}^n
\\ & = \det \left[
\sum_{k=0}^{\infty} e_{k}(-\widetilde{\bt}) h_{k-i+j}(\xx)
\right]_{i,j=1}^n
\\ & = \sum_{\lambda \subseteq \ell^n} p_{\lambda}(-\widetilde{\bt}) \G_{\lambda}(\xx; \bt),
\end{align*}
where $p_{\lambda}(\widetilde{\bt}) = \det \bigl[ e_{\lambda_i-i+j}(t_{\ell}, \dotsc, t_i) \bigr]_{i,j=1}^n$.
\end{thm}

\begin{proof}
The determinant formula for the partition function comes from the interpretation as NILPs with the LGV lemma.

To see the right hand side, we take the model and make a cut for the $i$-th path in the $-t_{n-i}$ spectral weight line.
The number of ways to go from the top to the cut points is given by $p_{\lambda}(-\widetilde{\bt})$ via the LGV lemma.
The portion below the cut is $\G_{\lambda}(\xx; \bt)$.
\end{proof}

We can extend Theorem~\ref{thm:extended_height_Grothendieck_model} so that the ending points are not all the way on the left but at generic positions similar to how we generalized Corollary~\ref{cor:refined_coincidence_lemma} to Corollary~\ref{cor:generalized_coincidence}.
We leave the details for the interested reader.

\begin{ex}
Let $n = 5$ and $\ell = 6$. Consider the NILP corresponding to a state in the model $\overline{\fG}$
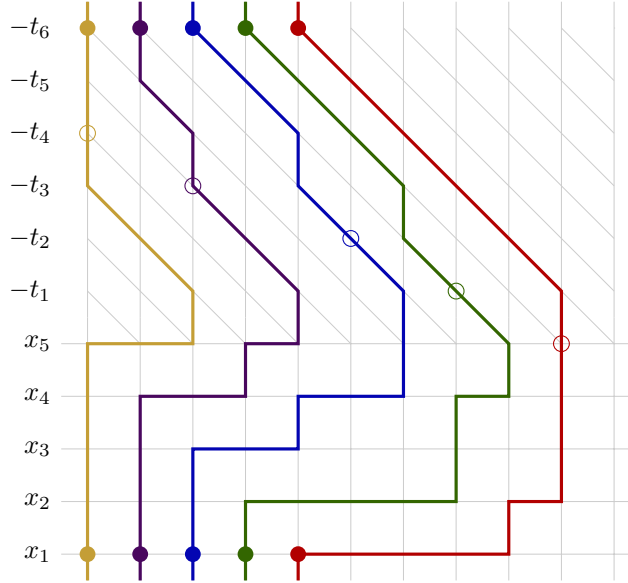
\[
\begin{tikzpicture}[>=latex,baseline=3.7cm,scale=0.7]
\draw[very thin, black!20] (-0.5, 0.5) grid (10.5, 5.5);
\foreach \i in {0,1,...,10}
  \draw[very thin, black!20] (\i,5.5) -- (\i,11.5);
\foreach \i in {1,2,...,6} {
  \draw[very thin, black!20] (\i,5) -- (0,\i+5);
  \draw[very thin, black!20] (4+\i,11) -- (10,\i+5);
}
\foreach \i in {1,2,3,4} {
  \draw[very thin, black!20] (6+\i,5) -- (\i,11);
}
\foreach \y in {1,2,3,4,5} {
  \draw node[anchor=east] at (-0.5,\y)  {\small $x_{\y}$};
}
\foreach \y in {1,2,3,4,5,6} {
  \draw node[anchor=east] at (-0.5,\y+5)  {\small $-t_{\y}$};
}
\draw[-, very thick, UQgold] (0,0.5) -- (0,5) -- (1,5) -- (2,5) -- (2,6) -- (0,8) -- (0,11.5);
\draw[-, very thick, UQpurple] (1,0.5) -- (1,4) -- (3,4) -- (3,5) -- (4,5) -- (4,6) -- (2,8) -- (2,9) -- (1,10) -- (1,11.5);
\draw[-, very thick, dbluecolor] (2,0.5) -- (2,3) -- (4,3) -- (4,4) -- (6,4) -- (6,6) -- (4,8) -- (4,9) -- (2,11) -- (2,11.5);
\draw[-, very thick, dgreencolor] (3,0.5) -- (3,2) -- (7,2) -- (7,4) -- (8,4) -- (8,5) -- (6,7) -- (6,8) -- (3,11) -- (3,11.5);
\draw[-, very thick, darkred] (4,0.5) -- (4,1) -- (8,1) -- (8,2) -- (9,2) -- (9,6) -- (4,11) -- (4,11.5);
\fill[UQgold] (0,1) circle (0.15);
\fill[UQpurple] (1,1) circle (0.15);
\fill[dbluecolor] (2,1) circle (0.15);
\fill[dgreencolor] (3,1) circle (0.15);
\fill[darkred] (4,1) circle (0.15);
\fill[UQgold] (0,11) circle (0.15);
\fill[UQpurple] (1,11) circle (0.15);
\fill[dbluecolor] (2,11) circle (0.15);
\fill[dgreencolor] (3,11) circle (0.15);
\fill[darkred] (4,11) circle (0.15);
\draw[UQgold] (0,9) circle (0.15);
\draw[UQpurple] (2,8) circle (0.15);
\draw[dbluecolor] (5,7) circle (0.15);
\draw[dgreencolor] (7,6) circle (0.15);
\draw[darkred] (9,5) circle (0.15);
\end{tikzpicture}
\]
The open circles drawn above are the cut points, and thus, we compute $\lambda = 54310$. The term in $p_{\lambda}(-\widetilde{\bt})$ corresponding to this NILP is $-t_2^2 t_3^2 t_4^2 t_5^4 t_6^3$.
\end{ex}

\subsection{Feh\'er--N\'emethi--Rim\'anyi-type identity}

In parallel to Theorem~\ref{thm:fnrgstype}, we derive a Feh\'er--N\'emethi--Rim\'anyi identity~\cite{FNR12} for refined Grothendieck functions.
We give both a simple proof using the Feh\'er--N\'emethi--Rim\'anyi identity and an alternative proof using the lattice model and the Yang--Baxter algebra similar to~\cite{Motegi20} (although with a different Yang--Baxter algebra).
We will show by giving an explicit example showing that our identity is a different than the one by Guo and Sun for factorial Grothendieck polynomials~\cite{GS19}.

Recall that for $S \in \binom{[n]}{k}$, we have $\overline{S} = \{1,2,\dotsc,n \} \setminus S$.
Recall that for sequences $\lambda, \mu$, we denote the concatenation by $\lambda \sqcup \mu$.

\begin{prop}
\label{prop:fnrgs_grothendieck}
Let $\lambda = (\lambda_1,\dots,\lambda_k)$ be a partition satisfying $\lambda_1 \leq m - k$,
and let $\mu = (m-k)^{n-k} \sqcup \lambda$.
Then we have
\begin{align}
\displaystyle
\G_{\mu}(\xx;\bt)
= \sum_{S \in \binom{[n]}{k}}
\prod_{i \in \overline{S}} x_i^m \prod_{i \in \overline{S}} \prod_{j \in S} \frac{1}{x_i-x_j} W_{\lambda}(\xx_S;\bt),
\label{refinedfnridentity}
\end{align}
where $\xx_S = \{x_{j_1},\dots,x_{j_k} \}$ for $\displaystyle S = \{j_1, \dotsc, j_k \}$ and
\begin{align}
W_{\lambda}(\xx_{S};\bt) = \sum_{\nu \subseteq (m-k)^k} E_{(m-k)^{n-k} \sqcup \lambda}^{(m-k)^{n-k} \sqcup \nu}(-\bt) s_\nu(\xx_{S}).
\label{smallerpartitionfunction}
\end{align}
\end{prop}

\begin{proof}
Combining~\eqref{eq:FNR} and the Schur expansion of refined Grothendieck polynomials~\eqref{eq:Grothendieck_defn}, we obtain
\begin{align*}
\G_{\mu}(\xx;\bt)
& = \sum_{\mu \subseteq \nu \subseteq (m-k)^n } E_\mu^\nu(-\bt) s_\nu
 = \sum_{\lambda \subseteq \nu \subseteq (m-k)^{k} }
 E_{(m-k)^{n-k} \sqcup \lambda}^{(m-k)^{n-k} \sqcup \nu}(-\bt) s_{((m-k)^{n-k},\nu)} \\
& = \sum_{\lambda \subseteq \nu \subseteq (m-k)^{k} } E_{(m-k)^{n-k} \sqcup \lambda}^{(m-k)^{n-k} \sqcup \nu}(-\bt)
\sum_{S \in \binom{[n]}{k}}
\prod_{i \in \overline{S}} x_i^m
\prod_{i \in \overline{S}}
\prod_{j \in S} \frac{1}{x_i-x_j}
s_{\nu}(\xx_{S}) \\
& = \sum_{S \in \binom{[n]}{k}}
\prod_{i \in \overline{S}} x_i^m
\prod_{i \in \overline{S}}
\prod_{j \in S} \frac{1}{x_i-x_j}
\sum_{\lambda \subseteq \nu \subseteq (m-k)^{k} } E_{(m-k)^{n-k} \sqcup \lambda}^{(m-k)^{n-k} \sqcup \nu}(-\bt) s_{\nu}(\xx_{S}) \\
& = \sum_{S \in \binom{[n]}{k}}
\prod_{i \in \overline{S}} x_i^m
\prod_{i \in \overline{S}}
\prod_{j \in S} \frac{1}{x_i-x_j}
W_{\lambda}(\xx_{S};\bt)
\end{align*}
as desired.
\end{proof}

We present another proof using integrability, which uses partition function descriptions and Yang--Baxter algebra, and which does not assume the identity~\eqref{eq:FNR}.
In addition to the $A$- and $B$-operators, we require the $D$-operator $D(z)$ and a list of commutation relations that follow from the $RLL$ relation in parallel to Corollary~\ref{cor:AB_ops_commute}.

\begin{cor}
\label{cor:BD_ops_commute}
We have
\begin{subequations}
\label{eq:BDcommrels}
\begin{align}
D(z_i)B(z_j)&=\frac{z_j}{z_i-z_j}B(z_j)D(z_i)-
\frac{z_j}{z_i-z_j}B(z_i)D(z_j), \label{BDcommrelone} \\
D(z_j)B(z_i)&=\frac{z_i}{z_j}D(z_i)B(z_j), \label{BDcommreltwo} \\
B(z_j)B(z_i)&=\frac{z_i}{z_j}B(z_i)B(z_j), \label{BDcommrelthree} \\
D(z_j)D(z_i)&=D(z_i)D(z_j). \label{BDcommrelfour}
\end{align}
\end{subequations}
\end{cor}

Now we give our integrability proof of Proposition~\ref{prop:fnrgs_grothendieck}.

\begin{proof}[Integrability proof]
Define
\[
(e)_n :=
\bigotimes_{i=1}^n (e_1)_i
\otimes
\bigotimes_{i=n+1}^{m+n-k} (e_0)_i,
\qquad\qquad
(e)_0 := \bigotimes_{i=1}^{m+n-k} (e_0)_i.
\]
We also denote the $A$-operator constructed from the $L$-operator for the jagged part and acts on $V_1 \otimes V_2 \otimes \cdots \otimes V_m$ as $\overline{A}_m(z)$.
We start from the partition function description of the refined Grothendieck polynomials
\begin{align*}
Z(\fG_{\mu}; \xx, \bt)
& := (e_\mu)_{m+n-k}^*
\overline{A}_{\lambda_k+1}(-t_{n-1}) \cdots
\overline{A}_{\lambda_1+k}(-t_{n-k})
\overline{A}_{m+1}(-t_{n-k-1}) \cdots 
\overline{A}_{m+n-k-1}(-t_1) \\
& \hspace{20pt} \times A_{m+n-k}(x_n) \cdots A_{m+n-k}(x_1)
(e)_n,
\end{align*}
which is equivalent to the nonintersecting lattice path description which the starting and ending points of the paths are on the bottom and top side.
One can see from its graphical description that we can push down the ending points in the northeast part and we get 
\begin{align*}
Z(\fG_{\mu}; \xx, \bt) & = (e_\mu)_{m+n-k}^*
\overline{A}_{\lambda_k+1}(-t_{n-1}) \cdots
\overline{A}_{\lambda_1+k}(-t_{n-k})
\overline{A}_{m}(-t_{n-k-1}) \cdots 
\overline{A}_{m}(-t_1) \\
& \hspace{20pt} \times A_{m+n-k}(x_n) \cdots A_{m+n-k}(x_1)
(e)_n.
\end{align*}
We next change the starting points to the left side and the ending points to the right side, and let $\widetilde{\fG}_{\mu}$ denote the corresponding lattice model.
Namely, we consider the following partition function
\begin{align}
Z(\widetilde{\fG}_{\mu}; \xx, \bt)
& := (e_\mu)_{m+n-k}^*
\overline{A}_{\lambda_k+1}(-t_{n-1}) \cdots
\overline{A}_{\lambda_1+k}(-t_{n-k})
\overline{A}_{m}(-t_{n-k-1}) \cdots 
\overline{A}_{m}(-t_1) \nonumber \\
& \hspace{20pt} \times D_{m+n-k}(x_n) \cdots D_{m+n-k}(x_{k+1})
B_{m+n-k}(x_k) \cdots B_{m+n-k}(x_{1}) (e)_0.
\label{changedpartitionfunction}
\end{align}
By comparing the Boltzmann weights of the partition functions in the southwest part and northeast part, we find the following relation
\begin{align}
Z(\widetilde{\fG}_{\mu}; \xx, \bt) = \prod_{i=k+1}^n x_i^{i-k-1} \prod_{i=1}^n x_i^{n+1-i} Z(\fG_{\mu}; \xx, \bt)
= \prod_{i=k+1}^n x_i^{i-k-1} \prod_{i=1}^n x_i^{n+1-i} G_{\mu}(\xx;\bt),
\label{relationwithrefinedgrothendieck}
\end{align}
where the last equality is using $Z(\fG_{\mu}; \xx, \bt) = \G_{\mu}(\xx;\bt)$ (Theorem~\ref{thm:partition_function_Grothendieck}).
Now we investigate~\eqref{changedpartitionfunction} using the Yang--Baxter algebra.
Using the commutation relations~\eqref{eq:BDcommrels} and applying a similar argument given in the section for refined dual Grothendieck polynomials,
we get the following multiple commutation relation between the $B$- and $D$-operators:
\begin{align}
&D(x_n) \cdots D(x_{k+1})B(x_k) \cdots B(x_1) \nonumber \\
& = \sum_{S=\{j_1 < j_2 < \cdots < j_k \}
\in \binom{[n]}{k}}
\frac{\displaystyle \prod_{i \in \overline{S}} x_i}
{\displaystyle \prod_{i=k+1}^{n} x_i}
\frac{\displaystyle \prod_{i=1}^{k-1} x_i^{k-i}}
{\displaystyle \prod_{i=1}^{k-1} x_{j_i}^{k-i}}
\prod_{i \in \overline{S}}
\prod_{j \in S} \frac{x_j}{x_i-x_j}  B(x_{j_k}) \cdots B(x_{j_2}) B(x_{j_1})
\prod_{i \in \overline{S}} D(x_i).
\label{BDmultiplecommutation}
\end{align}
Using~\eqref{BDmultiplecommutation}  and the action of the $D$-operators
\[
\prod_{i \in \overline{S}} D(x_i) (e)_0 = \prod_{i \in \overline{S}} x_i^{m+n-k} (e)_0,
\]
the partition function~\eqref{changedpartitionfunction} can be expressed as
\begin{align}
Z(\widetilde{\fG}_{\mu}; \xx, \bt)
& = \sum_{S=\{j_1 < j_2 < \cdots < j_k \}
\in \binom{[n]}{k}}
\frac{\displaystyle \prod_{i \in \overline{S}} x_i^{m+n+1-k}}
{\displaystyle \prod_{i=k+1}^{n} x_i}
\frac{\displaystyle \prod_{i=1}^{k-1} x_i^{k-i}}
{\displaystyle \prod_{i=1}^{k-1} x_{j_i}^{k-i}}
\prod_{i \in \overline{S}}
\prod_{j \in S} \frac{x_j}{x_i-x_j} \nonumber \\
& \hspace{20pt}  \times
(e_\mu)_{m+n-k}^*
\overline{A}_{\lambda_k+1}(-t_{n-1}) \cdots
\overline{A}_{\lambda_1+k}(-t_{n-k})
\overline{A}_{m}(-t_{n-k-1}) \cdots 
\overline{A}_{m}(-t_1) \nonumber \\
& \hspace{20pt} \times B_{m+n-k}(x_{j_k}) \cdots B_{m+n-k}(x_{j_1}) (e)_0.
\label{changedpartitionfunctiontwo}
\end{align}
Now let us examine the partition functions appearing in the summands in~\eqref{changedpartitionfunctiontwo}.
From its graphical description, we can see that the partition functions can be rewritten as follows.

\begin{align}
& (e_\mu)_{m+n-k}^*
\overline{A}_{\lambda_k+1}(-t_{n-1}) \cdots
\overline{A}_{\lambda_1+k}(-t_{n-k})
\overline{A}_{m}(-t_{n-k-1}) \cdots 
\overline{A}_{m}(-t_1) \nonumber \\
& \hspace{20pt} \times B_{m+n-k}(x_{j_k}) \cdots B_{m+n-k}(x_{j_1})
(e)_0 \nonumber \\
& = (e_\lambda)_{m}^*
\overline{A}_{\lambda_k+1}(-t_{n-1}) \cdots
\overline{A}_{\lambda_1+k}(-t_{n-k})
\overline{A}_{m}(-t_{n-k-1}) \cdots 
\overline{A}_{m}(-t_1) \nonumber \\
& \hspace{20pt} \times B_{m}(x_{j_k}) \cdots B_{m}(x_{j_1})
(e)_0 \nonumber \\
& = \prod_{i=1}^k x_{j_i}^{k+1-i} W_{\lambda}(\xx_{S};\bt),
\label{relatewithpartitionw}
\end{align}
where
\begin{align*}
W_{\lambda}(\xx_{S};\bt)
& = (e_\lambda)_{m}^*
\overline{A}_{\lambda_k+1}(-t_{n-1}) \cdots
\overline{A}_{\lambda_1+k}(-t_{n-k})
\overline{A}_{m}(-t_{n-k-1}) \cdots 
\overline{A}_m(-t_1) \\
& \hspace{20pt} \times \prod_{j \in S} A_m(x_j)
\bigotimes_{i=1}^{k} (e_1)_i \otimes
\bigotimes_{i=k+1}^{m} (e_0)_i,
\end{align*}
Decomposing the partition function as
\begin{align*}
W_{\lambda}(\xx_{S};\bt)
=&\sum_{\nu \subseteq (m-k)^k}
(e_\lambda)_{m}^*
\overline{A}_{\lambda_k+1}(-t_{n-1}) \cdots
\overline{A}_{\lambda_1+k}(-t_{n-k})
\overline{A}_{m}(-t_{n-k-1}) \cdots 
\overline{A}_m(-t_1)
(e_\nu)_{m}
 \\
& \hspace{20pt} \times (e_\nu)_m^*
\prod_{j \in S} A_m(x_j)
\bigotimes_{i=1}^{k} (e_1)_i \otimes
\bigotimes_{i=k+1}^{m} (e_0)_i,
\end{align*}
and using
$\displaystyle (e_\nu)_m^*
\prod_{j \in S} A_m(x_j)
\otimes_{i=1}^{k} (e_1)_i
\otimes_{i=k+1}^{m} (e_0)_i=s_\nu(\xx_{S})$
and
\begin{align*}
&
(e_\lambda)_{m}^*
\overline{A}_{\lambda_k+1}(-t_{n-1}) \cdots
\overline{A}_{\lambda_1+k}(-t_{n-k})
\overline{A}_{m}(-t_{n-k-1}) \cdots 
\overline{A}_m(-t_1)
(e_\nu)_{m} \\
& =
(e_{((m-k)^{n-k},\lambda)})_{m+n-k}^*
\overline{A}_{\lambda_k+1}(-t_{n-1}) \cdots
\overline{A}_{\lambda_1+k}(-t_{n-k}) \\
& \hspace{20pt} \times \overline{A}_{m+1}(-t_{n-k-1}) \cdots 
\overline{A}_{m+n-k-1}(-t_1)
(e_{((m-k)^{n-k},\nu)})_{m+n-k} \\
& = E_{(m-k)^{n-k} \sqcup \lambda}^{(m-k)^{n-k} \sqcup \nu}(-\bt),
\end{align*}
we note that $W_{\lambda}(\xx_{S};\bt)$ can be explicitly written as
\[
W_{\lambda}(\xx_{S};\bt) = \sum_{\nu \subseteq (m-k)^k} E_{(m-k)^{n-k} \sqcup \lambda}^{(m-k)^{n-k} \sqcup \nu}(-\bt) s_\nu(\xx_{S}).
\]
Inserting~\eqref{relatewithpartitionw} into~\eqref{changedpartitionfunctiontwo}, we get
\begin{equation}
Z(\widetilde{\fG}_{\mu}; \xx, \bt)
= \sum_{S \in \binom{[n]}{k}}
\frac{\displaystyle \prod_{i \in S} x_i \prod_{i \in \overline{S}} x_i^{m+n+1-k} \prod_{i=1}^{k-1} x_i^{k-i}}{\displaystyle \prod_{i=k+1}^{n} x_i}
\prod_{i \in \overline{S}} \prod_{j \in S} \frac{x_j}{x_i-x_j} W_{\lambda}(\xx_{S};\bt).
\label{changedpartitionfunctionthree}
\end{equation}
Comparing~\eqref{relationwithrefinedgrothendieck} and~\eqref{changedpartitionfunctionthree} and simplifying the expression for each summand, we obtain~\eqref{refinedfnridentity}.
\end{proof}

Let us make some remarks on Proposition~\ref{prop:fnrgs_grothendieck}.
We can see that when $\bt = 0$ in~\eqref{smallerpartitionfunction}, we have $W_{\lambda}(\xx_{S}; 0) = s_{\lambda}(\xx_{S})$ being the Schur polynomial since
$E_{(m-k)^{n-k} \sqcup \lambda}^{(m-k)^{n-k} \sqcup \nu}(0) = \delta_{\lambda,\nu}$
and~\eqref{refinedfnridentity} becomes
\begin{align}
\displaystyle
s_{\mu}(\xx)
=&\sum_{S
\in \binom{[n]}{k}}
\prod_{i \in \overline{S}} x_i^m
\prod_{i \in \overline{S}}
\prod_{j \in S} \frac{1}{x_i-x_j}
s_{\lambda}(\xx_{S}),
\label{eq:FNR}
\end{align}
which is the identity for Schur polynomials by Feh\'er, N\'emethi, and Rim\'anyi~\cite{FNR12}.
A generalization to the factorial Grothendieck polynomials is derived by Guo and Sun~\cite{GS19}, in which $s_{\lambda}(\xx_{S})$ in the summand in~\eqref{eq:FNR} is replaced by $\G_{\lambda}(\xx_{S};\beta)$ with some overall factor when expanding $\G_{\mu}(\xx;\beta)$.
However, Equation~\eqref{smallerpartitionfunction} cannot be written using a single Grothendieck polynomial in general (even at $\bt = \beta$), so the identity is different from the Guo--Sun identity.

\begin{ex}
For $n = 4$, Proposition~\ref{prop:fnrgs_grothendieck} states that
\begin{align*}
\G_{2210}(x_1,x_2,x_3,x_4; t_1,t_2,t_3)
& = \frac{x_3^4 x_4^4}{(x_3-x_1)(x_3-x_2)(x_4-x_1)(x_4-x_2)}W_{(1,0)}(x_1,x_2;t_1,t_2,t_3) \\
& + \frac{x_2^4 x_4^4}{(x_2-x_1)(x_2-x_3)(x_4-x_1)(x_4-x_3)}W_{(1,0)}(x_1,x_3;t_1,t_2,t_3) \\
& + \frac{x_2^4 x_3^4}{(x_2-x_1)(x_2-x_4)(x_3-x_1)(x_3-x_4)}W_{(1,0)}(x_1,x_4;t_1,t_2,t_3) \\
& + \frac{x_1^4 x_4^4}{(x_1-x_2)(x_1-x_3)(x_4-x_2)(x_4-x_3)}W_{(1,0)}(x_2,x_3;t_1,t_2,t_3) \\
& + \frac{x_1^4 x_3^4}{(x_1-x_2)(x_1-x_4)(x_3-x_2)(x_3-x_4)}W_{(1,0)}(x_2,x_4;t_1,t_2,t_3) \\
& + \frac{x_1^4 x_2^4}{(x_1-x_3)(x_1-x_4)(x_2-x_3)(x_2-x_4)}W_{(1,0)}(x_3,x_4;t_1,t_2,t_3),
\end{align*}
where
\begin{align*}
& \G_{2210}(x_1,x_2,x_3,x_4; t_1,t_2,t_3) \\
& = x_1^2 x_2^2 x_3+x_1^2 x_2 x_3^2+x_1 x_2^2 x_3^2
+x_1^2 x_2^2 x_4+x_1^2 x_2 x_4^2+x_1 x_2^2 x_4^2
+x_1^2 x_3^2 x_4+x_1^2 x_3 x_4^2+x_1 x_3^2 x_4^2 \\
&\hspace{20pt} 
+x_2^2 x_3^2 x_4+x_2^2 x_3 x_4^2+x_2 x_3^2 x_4^2
+2(x_1^2 x_2 x_3 x_4+x_1 x_2^2 x_3 x_4+x_1 x_2 x_3^2 x_4+x_1 x_2 x_3 x_4^2)
\\
&\hspace{20pt} 
-t_1(x_1^2 x_2^2 x_3^2 + x_1^2 x_2^2 x_4^2 + x_1^2 x_3^2 x_4^2 + x_2^2 x_3^2 x_4^2
+x_1^2 x_2^2 x_3 x_4+x_1^2 x_2 x_3^2 x_4+x_1^2 x_2 x_3 x_4^2
+x_1 x_2^2 x_3^2 x_4  \\
&\hspace{40pt}  
+ x_1 x_2^2 x_3 x_4^2+x_1 x_2 x_3^2 x_4^2) \\
& \hspace{20pt} 
- (t_1+t_2)(x_1^2 x_2^2 x_3 x_4 + x_1^2 x_2 x_3^2 x_4+x_1^2 x_2 x_3 x_4^2 + x_1 x_2^2 x_3^2 x_4 + x_1 x_2^2 x_3 x_4^2+x_1 x_2 x_3^2 x_4^2) \\
& \hspace{20pt} 
+(t_1^2+t_1 t_2)(x_1^2 x_2^2 x_3^2 x_4+x_1^2 x_2^2 x_3 x_4^2 + x_1^2 x_2 x_3^2 x_4^2+x_1 x_2^2 x_3^2 x_4^2)
 - t_1^2 t_2 x_1^2 x_2^2 x_3^2 x_4^2,
\end{align*}
and
\begin{align*}
W_{10}(x_1,x_2; t_1,t_2,t_3)
& = x_1 + x_2 - (t_1+t_2)x_1 x_2 - t_1 (x_1^2+x_1 x_2+x_2^2)
\\ & \hspace{20pt}
+ (t_1^2+t_1 t_2)(x_1^2 x_2+x_1 x_2^2) -t_1^2 t_2 x_1^2 x_2^2 \\
& = \G_{10}(x_1,x_2; t_1,t_2,t_3) - t_2 \G_{11}(x_1,x_2; t_1,t_2,t_3)
\\ & \hspace{20pt} - t_1 \G_{20}(x_1,x_2; t_1,t_2,t_3) + t_1 t_2 \G_{21}(x_1,x_2; t_1,t_2,t_3).
\end{align*}
\end{ex}

\appendix
\section{Combinatorial bijection between TASEP and LPP}

We give another combinatorial proof of Theorem~\ref{thm:last_passage} using the last passage times for the TASEP particle model from~\cite{Johansson00}. Our proof here is essentially equivalent to the proof from~\cite{Yel20}, just explained using a different stochastic model. Indeed, consider particles in the step initial condition on $\ZZ$, that is particle $p_i$ starts in position $-(i-1)$, with particle $p_i$ moving to the right with probability $(1-x_i)$ if there is no particle to its right. Let $G^*(\ell,n) = G(\ell, n) + \ell + n - 1 = t$, which measures the probability that the particle $p_n$ has moved $\ell$ steps at time $t$.
We note $G^*(\ell,n)$ is equivalent to computing the \emph{first time} the particle $p_n$ moved $\ell$ steps occurs at time $t$.

Let us first roughly describe how the proof works, which is based on particles $p_i$ for $i < n$ blocking the motion of $p_n$.
Consider a reverse plane partition $T$ and consider the $i$-th row $r_1, \dotsc, r_k$.
In this, we consider the conditional probability with the particles $p_{r_j}$ blocking the movement of $p_n$, which contributes $x_{r_j}$, and otherwise all particles move to the right when possible.
However, if $r_j$ equals its entry in the row below (in column $j$), then particle $p_{r_j}$ would already have blocked $p_n$ on an earlier step. Hence, we do not want to double-count it (hence, it contributes a $1$ to the weight).

Now we will be more precise.
We say a particle \defn{blocks} if it can move a step to the right, but it remains in place (which contributes $x_i$ to the conditional probability).
Particle $p_i$ blocks after moving $j$ steps (\textit{i.e.}, $p_i$ is at position $j + 1 - i$) at time $t$ if and only if there is an $i$ in row $j+1$ from the bottom and column $t-j-i+2$.
We claim there is a unique way to fill in the rest of the partition to obtain a reverse plane partition.
In the other direction, if there is an $i$ in row $r$ from the bottom and column $c$ with no $i$ immediately below it, then $p_i$ blocks at time $t = r + c + i - 3$.
To see this is a bijection, consider the reverse plane partition with all entries $1$, thus particle $1$ is blocking maximally.
Then add $1$ to entries repeatedly such that the result is again a reverse plane partition, and under the above map this moves the blocking particle one step to the right. It is easy to see that this is confluent (\textit{i.e.}, it satisfies the diamond lemma).
The weight of the reverse plane partition times $\prod_{i=1}^n (1 - x_i)^{\ell}$ is exactly the probability for this sequence of states.
We note that this bijection is just using the values $G^*(m,n)$, which equal $G(m, n)$ for the matrix $\Phi(T)$ but with all entries increased by $1$.
We leave the details to the interested reader.

\begin{ex}
Let us consider the reverse plane partition and matrix from Example~\ref{ex:RPP_Matrix_path} in order to compute one part of
\[
P(G(3, 4) = 3, G(2, 4) = 3, G(1, 4) = 2) = P(G^*(3, 4) = 9, G^*(2, 4) = 8, G^*(1, 4) = 6).
\]
Under the bijection described above, the evolution of the first four particles is given by
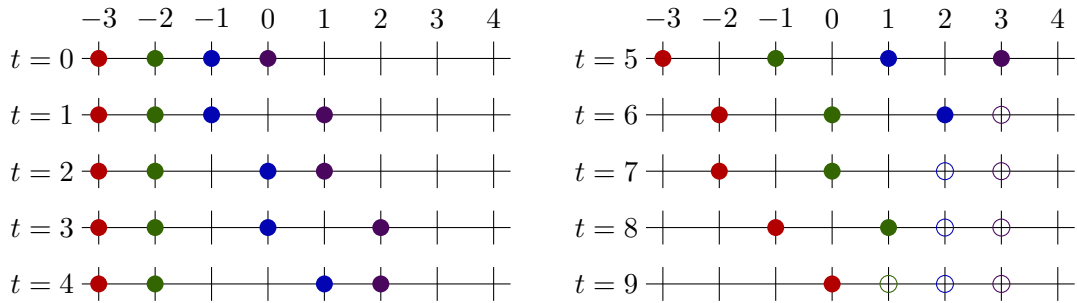
\[
\begin{tikzpicture}[scale=0.75]
\foreach \t in {0,...,4} {
  \draw (-4,-\t) node {$t = \t$};
  \draw[-] (-3.3,-\t) -- (4.3,-\t);
  \foreach \x in {-3,...,4} \draw [-] (\x,-\t-0.3) -- (\x,-\t+0.3);
}
\foreach \x in {-3,...,4} \draw (\x,0.7) node {$\x$};
\foreach \t/\x in {0/0,1/1,2/1,3/2,4/2}
    \fill[UQpurple] (\x,-\t) circle (0.15);
\foreach \t/\x in {0/-1,1/-1,2/0,3/0,4/1}
    \fill[dbluecolor] (\x,-\t) circle (0.15);
\foreach \t/\x in {0/-2,1/-2,2/-2,3/-2,4/-2}
    \fill[dgreencolor] (\x,-\t) circle (0.15);
\foreach \t/\x in {0/-3,1/-3,2/-3,3/-3,4/-3}
    \fill[darkred] (\x,-\t) circle (0.15);
\begin{scope}[xshift=10cm]
\foreach \t in {5,...,9} {
  \draw (-4,-\t+5) node {$t = \t$};
  \draw[-] (-3.3,-\t+5) -- (4.3,-\t+5);
  \foreach \x in {-3,...,4} \draw [-] (\x,-\t+5-0.3) -- (\x,-\t+5+0.3);
}
\foreach \x in {-3,...,4} \draw (\x,0.7) node {$\x$};
\fill[UQpurple] (3,0) circle (0.15);
\foreach \t in {1,2,3,4}
    \draw[UQpurple] (3,-\t) circle (0.15);
\foreach \t/\x in {0/1,1/2}
    \fill[dbluecolor] (\x,-\t) circle (0.15);
\foreach \t in {2,3,4}
    \draw[dbluecolor] (2,-\t) circle (0.15);
\foreach \t/\x in {0/-1,1/0,2/0,3/1}
    \fill[dgreencolor] (\x,-\t) circle (0.15);
\draw[dgreencolor] (1,-4) circle (0.15);
\foreach \t/\x in {0/-3,1/-2,2/-2,3/-1,4/0}
    \fill[darkred] (\x,-\t) circle (0.15);
\end{scope}
\end{tikzpicture}
\]
In the picture above, an open circle indicates that we do not impose any condition on its position.
\end{ex}

\bibliographystyle{alpha}
\bibliography{grothendiecks}{}

\begin{thebibliography}{KMO16b}

\bibitem[AOS96]{AOS96}
H.~Awata, S.~Odake, and J.~Shiraishi.
\newblock Integral representations of the {M}acdonald symmetric polynomials.
\newblock {\em Comm. Math. Phys.}, 179(3):647--666, 1996.

\bibitem[AY20]{AY20}
Alimzhan Amanov and Damir Yeliussizov.
\newblock Determinantal formulas for dual {G}rothendieck polynomials.
\newblock Preprint, \arxiv{2003.03907}, 2020.

\bibitem[BBF11]{BBF11b}
Ben Brubaker, Daniel Bump, and Solomon Friedberg.
\newblock Weyl group multiple {D}irichlet series, {E}isenstein series and
  crystal bases.
\newblock {\em Ann. of Math. (2)}, 173(2):1081--1120, 2011.

\bibitem[BBW20]{BBW20}
Alexei Borodin, Alexey Bufetov, and Michael Wheeler.
\newblock Between the stochastic six vertex model and {H}all--{L}ittlewood
  processes.
\newblock {\em J. Combin. Theory Ser. A}, 2020.
\newblock To appear.

\bibitem[BR01]{BR01}
Jinho Baik and Eric~M. Rains.
\newblock Algebraic aspects of increasing subsequences.
\newblock {\em Duke Math. J.}, 109(1):1--65, 2001.

\bibitem[BS17]{BS17}
Daniel Bump and Anne Schilling.
\newblock {\em Crystal bases}.
\newblock World Scientific Publishing Co. Pte. Ltd., Hackensack, NJ, 2017.
\newblock Representations and combinatorics.

\bibitem[BS20]{BS20}
Valentin Buciumas and Travis Scrimshaw.
\newblock Double {G}rothendieck polynomials and colored lattice models.
\newblock {\em Int. Math. Res. Not. IMRN}, 2020.
\newblock To appear, \arxiv{2007.04533}.

\bibitem[BSW20]{BSW20}
Valentin Buciumas, Travis Scrimshaw, and Katherine Weber.
\newblock Colored five-vertex models and {L}ascoux polynomials and atoms.
\newblock {\em J. Lond. Math. Soc.}, 102(3):1047--1066, 2020.

\bibitem[Buc02]{Buch02}
Anders~Skovsted Buch.
\newblock A {L}ittlewood-{R}ichardson rule for the {$K$}-theory of
  {G}rassmannians.
\newblock {\em Acta Math.}, 189(1):37--78, 2002.

\bibitem[CLL02]{CLL02}
William Y.~C. Chen, Bingqing Li, and J.~D. Louck.
\newblock The flagged double {S}chur function.
\newblock {\em J. Algebraic Combin.}, 15(1):7--26, 2002.

\bibitem[CP19]{CP19}
Melody Chan and Nathan Pflueger.
\newblock Combinatorial relations on skew {S}chur and skew stable
  {G}rothendieck polynomials.
\newblock Preprint, \arxiv{1909.12833}, 2019.

\bibitem[DJKM83]{DKJM83}
Etsur{\={o}} Date, Michio Jimbo, Masaki Kashiwara, and Tetsuji Miwa.
\newblock Transformation groups for soliton equations.
\newblock In {\em Nonlinear integrable systems---classical theory and quantum
  theory ({K}yoto, 1981)}, pages 39--119. World Sci. Publishing, Singapore,
  1983.

\bibitem[FK94]{FK94}
Sergey Fomin and Anatol~N. Kirillov.
\newblock Grothendieck polynomials and the {Y}ang-{B}axter equation.
\newblock In {\em Formal power series and algebraic combinatorics/{S}\'eries
  formelles et combinatoire alg\'ebrique}, pages 183--189. DIMACS, Piscataway,
  NJ, 1994.

\bibitem[FK96]{FK96}
Sergey Fomin and Anatol~N. Kirillov.
\newblock The {Y}ang-{B}axter equation, symmetric functions, and {S}chubert
  polynomials.
\newblock In {\em Proceedings of the 5th {C}onference on {F}ormal {P}ower
  {S}eries and {A}lgebraic {C}ombinatorics ({F}lorence, 1993)}, volume 153,
  pages 123--143, 1996.

\bibitem[FNR12]{FNR12}
L\'{a}szl\'{o}~M. Feh\'{e}r, Andr\'{a}s N\'{e}methi, and Rich\'{a}rd
  Rim\'{a}nyi.
\newblock Equivariant classes of matrix matroid varieties.
\newblock {\em Comment. Math. Helv.}, 87(4):861--889, 2012.

\bibitem[Ful97]{Fulton}
William Fulton.
\newblock {\em Young tableaux}, volume~35 of {\em London Mathematical Society
  Student Texts}.
\newblock Cambridge University Press, Cambridge, 1997.
\newblock With applications to representation theory and geometry.

\bibitem[Gal17]{Galashin17}
Pavel Galashin.
\newblock A {L}ittlewood-{R}ichardson rule for dual stable {G}rothendieck
  polynomials.
\newblock {\em J. Combin. Theory Ser. A}, 151:23--35, 2017.

\bibitem[GGL16]{GGL16}
Pavel Galashin, Darij Grinberg, and Gaku Liu.
\newblock Refined dual stable {G}rothendieck polynomials and generalized
  {B}ender--{K}nuth involutions.
\newblock {\em Electron. J. Combin.}, 23(3):Paper 3.14, 28, 2016.

\bibitem[GK17]{GK17}
Vassily Gorbounov and Christian Korff.
\newblock Quantum integrability and generalised quantum {S}chubert calculus.
\newblock {\em Adv. Math.}, 313:282--356, 2017.

\bibitem[GS19]{GS19}
Peter~L. Guo and Sophie C.~C. Sun.
\newblock Identities on factorial {G}rothendieck polynomials.
\newblock {\em Adv. in Appl. Math.}, 111:101933, 11, 2019.

\bibitem[GV85]{GV85}
Ira Gessel and G\'erard Viennot.
\newblock Binomial determinants, paths, and hook length formulae.
\newblock {\em Adv. in Math.}, 58(3):300--321, 1985.

\bibitem[GZJ20]{GZJ20}
Ajeeth Gunna and Paul Zinn-Justin.
\newblock Vertex models for canonical {G}rothendieck polynomials and their
  duals.
\newblock Preprint, \arxiv{2009.13172}, 2020.

\bibitem[HKZJ19]{HKZJ18}
Iva Halacheva, Allen Knutson, and Paul Zinn-Justin.
\newblock Restricting {S}chubert classes to symplectic {G}rassmannians using
  self-dual puzzles.
\newblock {\em S\'{e}m. Lothar. Combin.}, 82B:Art. 83, 12 pp., 2019.

\bibitem[HS20]{HS20}
Graham Hawkes and Travis Scrimshaw.
\newblock Crystal structures for canonical {G}rothendieck functions.
\newblock {\em Algebraic Combin.}, 3(3):727--755, 2020.

\bibitem[IN18]{IN18}
Shinsuke Iwao and Hidetomo Nagai.
\newblock The discrete {T}oda equation revisited: dual {$\beta$}-{G}rothendieck
  polynomials, ultradiscretization, and static solitons.
\newblock {\em J. Phys. A}, 51(13):134002, 16, 2018.

\bibitem[Iwa20a]{Iwao20}
Shinsuke Iwao.
\newblock Free-fermions and skew stable {G}rothendieck polynomials.
\newblock Preprint, \arxiv{2004.09499}, 2020.

\bibitem[Iwa20b]{Iwao19}
Shinsuke Iwao.
\newblock Grothendieck polynomials and the boson-fermion correspondence.
\newblock {\em Algebraic Combin.}, 3(5):1023--1040, 2020.

\bibitem[Joh00]{Johansson00}
Kurt Johansson.
\newblock Shape fluctuations and random matrices.
\newblock {\em Comm. Math. Phys.}, 209(2):437--476, 2000.

\bibitem[Joh10]{Johansson10}
Kurt Johansson.
\newblock A multi-dimensional {M}arkov chain and the {M}eixner ensemble.
\newblock {\em Ark. Mat.}, 48(1):79--95, 2010.

\bibitem[Kim20a]{Kim20}
Jang~Soo Kim.
\newblock Jacobi--{T}rudi formula for flagged refined dual stable
  {G}rothendieck polynomials.
\newblock Preprint, \arxiv{2008.12000}, 2020.

\bibitem[Kim20b]{Kim20II}
Jang~Soo Kim.
\newblock Jacobi--{T}rudi formula for refined dual stable {G}rothendieck
  polynomials.
\newblock Preprint, \arxiv{2003.00540}, 2020.

\bibitem[Kir16]{Kirillov16}
Anatol~N. Kirillov.
\newblock On some quadratic algebras {I} {$\frac{1}{2}$}: combinatorics of
  {D}unkl and {G}audin elements, {S}chubert, {G}rothendieck, {F}uss-{C}atalan,
  universal {T}utte and reduced polynomials.
\newblock {\em SIGMA Symmetry Integrability Geom. Methods Appl.}, 12:Paper No.
  002, 172, 2016.

\bibitem[KMO15]{KMO15}
Atsuo Kuniba, Shouya Maruyama, and Masato Okado.
\newblock Multispecies {TASEP} and combinatorial {$R$}.
\newblock {\em J. Phys. A}, 48(34):34FT02, 19, 2015.

\bibitem[KMO16a]{KMO16}
Atsuo Kuniba, Shouya Maruyama, and Masato Okado.
\newblock Inhomogeneous generalization of a multispecies totally asymmetric
  zero range process.
\newblock {\em J. Stat. Phys.}, 164(4):952--968, 2016.

\bibitem[KMO16b]{KMO16II}
Atsuo Kuniba, Shouya Maruyama, and Masato Okado.
\newblock Multispecies {TASEP} and the tetrahedron equation.
\newblock {\em J. Phys. A}, 49(11):114001, 22, 2016.

\bibitem[KZJ17]{KZJ17}
Allen Knutson and Paul Zinn-Justin.
\newblock Schubert puzzles and integrability {I}: invariant trilinear forms.
\newblock Preprint, \arxiv{1706.10019}, 2017.

\bibitem[Las03]{Lascoux03}
Alain Lascoux.
\newblock {\em Symmetric functions and combinatorial operators on polynomials},
  volume~99 of {\em CBMS Regional Conference Series in Mathematics}.
\newblock Published for the Conference Board of the Mathematical Sciences,
  Washington, DC; by the American Mathematical Society, Providence, RI, 2003.

\bibitem[Len00]{Lenart00}
Cristian Lenart.
\newblock Combinatorial aspects of the {$K$}-theory of {G}rassmannians.
\newblock {\em Ann. Comb.}, 4(1):67--82, 2000.

\bibitem[Lin73]{Lindstrom73}
Bernt Lindstr\"om.
\newblock On the vector representations of induced matroids.
\newblock {\em Bull. London Math. Soc.}, 5:85--90, 1973.

\bibitem[LN14]{LN14}
Alain Lascoux and Hiroshi Naruse.
\newblock Finite sum {C}auchy identity for dual {G}rothendieck polynomials.
\newblock {\em Proc. Japan Acad. Ser. A Math. Sci.}, 90(7):87--91, 2014.

\bibitem[LP07]{LamPyl07}
Thomas Lam and Pavlo Pylyavskyy.
\newblock Combinatorial {H}opf algebras and {$K$}-homology of {G}rassmannians.
\newblock {\em Int. Math. Res. Not. IMRN}, 2007(24):Art. ID rnm125, 48, 2007.

\bibitem[LS82a]{LS82II}
Alain Lascoux and Marcel-Paul Sch\"{u}tzenberger.
\newblock Polyn\^{o}mes de {S}chubert.
\newblock {\em C. R. Acad. Sci. Paris S\'{e}r. I Math.}, 294(13):447--450,
  1982.

\bibitem[LS82b]{LS82}
Alain Lascoux and Marcel-Paul Sch{\"u}tzenberger.
\newblock Structure de {H}opf de l'anneau de cohomologie et de l'anneau de
  {G}rothendieck d'une vari\'et\'e de drapeaux.
\newblock {\em C. R. Acad. Sci. Paris S\'er. I Math.}, 295(11):629--633, 1982.

\bibitem[Mot20]{Motegi20}
Kohei Motegi.
\newblock Integrability approach to
  {F}eh\'{e}r-{N}\'{e}methi-{R}im\'{a}nyi-{G}uo-{S}un type identities for
  factorial {G}rothendieck polynomials.
\newblock {\em Nuclear Phys. B}, 954:114998, 2020.

\bibitem[MS13]{MS13}
Kohei Motegi and Kazumitsu Sakai.
\newblock Vertex models, {TASEP} and {G}rothendieck polynomials.
\newblock {\em J. Phys. A}, 46(35):355201, 26, 2013.

\bibitem[MS14]{MS14}
Kohei Motegi and Kazumitsu Sakai.
\newblock {$K$}-theoretic boson-fermion correspondence and melting crystals.
\newblock {\em J. Phys. A}, 47(44):445202, 2014.

\bibitem[Oko00]{Okounkov00}
Andrei Okounkov.
\newblock Random matrices and random permutations.
\newblock {\em Internat. Math. Res. Notices}, (20):1043--1095, 2000.

\bibitem[Oko01]{Okounkov01}
Andrei Okounkov.
\newblock Infinite wedge and random partitions.
\newblock {\em Selecta Math. (N.S.)}, 7(1):57--81, 2001.

\bibitem[RTY18]{RTY18}
Victor Reiner, Bridget~Eileen Tenner, and Alexander Yong.
\newblock Poset edge densities, nearly reduced words, and barely set-valued
  tableaux.
\newblock {\em J. Combin. Theory Ser. A}, 158:66--125, 2018.

\bibitem[Sag20]{sage}
The Sage Developers.
\newblock {\em {S}age {M}athematics {S}oftware ({V}ersion 9.2)}, 2020.
\newblock \url{http://www.sagemath.org}.

\bibitem[SCc08]{combinat}
The {S}age-{C}ombinat community.
\newblock {S}age-{C}ombinat: enhancing {S}age as a toolbox for computer
  exploration in algebraic combinatorics, 2008.
\newblock \url{http://combinat.sagemath.org}.

\bibitem[Sta99]{ECII}
Richard~P. Stanley.
\newblock {\em Enumerative combinatorics. {V}ol. 2}, volume~62 of {\em
  Cambridge Studies in Advanced Mathematics}.
\newblock Cambridge University Press, Cambridge, 1999.
\newblock With a foreword by Gian-Carlo Rota and appendix 1 by Sergey Fomin.

\bibitem[SU05]{SU05}
Keiichi Shigechi and Masaru Uchiyama.
\newblock Boxed skew plane partition and integrable phase model.
\newblock {\em J. Phys. A}, 38(48):10287--10306, 2005.

\bibitem[Wac85]{Wachs85}
Michelle~L. Wachs.
\newblock Flagged {S}chur functions, {S}chubert polynomials, and symmetrizing
  operators.
\newblock {\em J. Combin. Theory Ser. A}, 40(2):276--289, 1985.

\bibitem[WZJ19]{WZJ16}
Michael Wheeler and Paul Zinn-Justin.
\newblock Littlewood-{R}ichardson coefficients for {G}rothendieck polynomials
  from integrability.
\newblock {\em J. Reine Angew. Math.}, 757:159--195, 2019.

\bibitem[Yel17]{Yel17}
Damir Yeliussizov.
\newblock Duality and deformations of stable {G}rothendieck polynomials.
\newblock {\em J. Algebraic Combin.}, 45(1):295--344, 2017.

\bibitem[Yel19a]{Yel19II}
Damir Yeliussizov.
\newblock Random plane partitions and corner distributions.
\newblock Preprint, \arxiv{1910.13378}, 2019.

\bibitem[Yel19b]{Yel19}
Damir Yeliussizov.
\newblock Symmetric {G}rothendieck polynomials, skew {C}auchy identities, and
  dual filtered {Y}oung graphs.
\newblock {\em J. Combin. Theory Ser. A}, 161:453--485, 2019.

\bibitem[Yel20]{Yel20}
D.~Yeliussizov.
\newblock Dual {G}rothendieck polynomials via last-passage percolation.
\newblock {\em C. R. Math. Acad. Sci. Paris}, 358(4):497--503, 2020.

\bibitem[Yel21]{Yel19III}
Damir Yeliussizov.
\newblock Enumeration of plane partitions by descents.
\newblock {\em J. Combin. Theory Ser. A}, 178:article 105367, 2021.

\end{thebibliography}
\end{document}